\documentclass[12pt, oneside]{article}   	
\usepackage{geometry}
\usepackage{graphicx}				% Use pdf, png, jpg, or eps?¡ì with pdflatex; use eps in DVI mode

\usepackage{indentfirst}
\usepackage{needspace}
\usepackage{hyperref}
\usepackage{tikz-cd}
\usepackage{amsthm}
\usepackage{amsfonts}
\usepackage{amsmath}
\usepackage{amscd}
\usepackage{enumerate}
\usepackage{amssymb}	
\usepackage{verbatim}
\usepackage{cancel}
\usepackage{cite}
\usepackage{textcomp}								
\usepackage{color}
\usepackage{mathrsfs}
\usepackage{appendix}
\usepackage[nottoc]{tocbibind}

\makeatletter
\newcommand{\subjclass}[2][2020]{%
  \def\@subjclass{#2}%
  \def\@subjclassyear{#1}%
}
\newcommand{\keywords}[1]{\def\@keywords{#1}}

\newcommand{\printinfo}{%
  \let\oldthefootnote\thefootnote
  \renewcommand{\thefootnote}{} % 隐藏编号
  \footnotetext{%
    \hspace{-1.8em}\begin{minipage}[t]{\textwidth}
      \small
      \textbf{\@subjclassyear\ Mathematics Subject Classification.} \@subjclass. \\
      \textbf{Keywords.} \@keywords.
    \end{minipage}
  }
  \let\thefootnote\oldthefootnote
  \addtocounter{footnote}{-1}
}
\makeatother

\newtheorem{example}{Example}[section]							
\newtheorem{theorem}{Theorem}[section]
\newtheorem{definition}{Definition}[section]
\newtheorem{proposition}{Proposition}[section]
\newtheorem{lemma}{Lemma}[section]
\newtheorem{remark}{Remark}[section]

\newtheorem{corollary}{Corollary}[section]

\newtheorem{conjecture}{Conjecture}

\def\R{\mathbb R}

\def\cH{\mathcal H}

\def\rd{\mathrm d}

\def\SL{\operatorname{SL}}

\def\inter{\operatorname{int}}
\def\top{\operatorname{top}}
\def\MME{\operatorname{MME}}
\def\NES{\operatorname{NES}}
\def\pt{\operatorname{Pt}}
\def\rv{\operatorname{rv}}
\def\rot{\operatorname{Rot}}
\def\vol{\operatorname{vol}}

\def\be{\begin{equation}}
\def\ee{\end{equation}}

\def\ES{\operatorname{ES}}
\title{Finiteness of ergodic measures of maximal entropy for certain skew product systems}
\author{Changguang Dong, Qiujie Qiao}
\date{\today}

\begin{document}
\maketitle
\printinfo 

\begin{abstract}
Consider a $C^{1+\beta}$ skew product $F$ over a 
topologically mixing Anosov diffeomorphism $f$, 
with fiber dynamics driven by a $C^{1+\beta}$ $\mathbb{R}^k$-action $\{\rho_{\bf t}\}$ 
on a closed manifold $Y$ through a $C^{1+\beta}$ cocycle $\tau$. 
Suppose that the fiber action preserves an 
ergodic smooth measure of maximal entropy $m_{\MME}$.
Under a nondegeneracy condition called joint nontriviality, 
and assuming that each element $\rho_{\bf t}$ is 
either of class $C^\infty$  or hyperbolic, 
we prove that $F$ admits only finitely many ergodic measures of maximal entropy. 
 Moreover, every such measure is of the form $\mu\times m_{\MME}$, 
where $\mu$ is an ergodic equilibrium state for a H\"older potential on the base. 
If, in addition, the fiber action is exponentially mixing of all orders, 
then every ergodic MME is Bernoulli, exponentially mixing of all orders, 
and satisfies the Central Limit Theorem. 

Our main novelties are entropy formulas for skew products, obtained through
the cone property and $r$-neutralized entropy, and a relative
Ledrappier--Young theorem that converts the fiber entropy contribution
into absolute continuity of fiberwise conditional measures. Combined
with nonlinear thermodynamic formalism, rotation theory and
Yomdin theory, these yield the finiteness of the base projections
and the product rigidity of ergodic MMEs.
\end{abstract}

\tableofcontents

\section{Introduction and main results}\label{sect:Main results}

For a topologically mixing Anosov diffeomorphism, it is classical that the measure of maximal entropy is
unique, and H\"older potentials have unique equilibrium states \cite{Bow75}. An extension
of an Anosov system need not retain this uniqueness: the fiber dynamics may
both alter the maximizing measure on the base and allow several invariant
lifts of the same projection. In this paper we study these two sources of
nonuniqueness for skew products driven by smooth $\R^k$-actions. We prove
that, under explicit entropy and accessibility assumptions, there are only
finitely many ergodic measures of maximal entropy and that every one of them
is a product of a base equilibrium state and a fixed smooth measure on the fiber.

Let $X$ and $Y$ be closed manifolds, let $f\colon X\to X$ be a diffeomorphism,
and let $\rho=\{\rho_{\mathbf t}\}_{\mathbf t\in\R^k}$ be a locally free action on $Y$.
For a cocycle $\tau\colon X\to\R^k$, consider
\begin{equation}\label{intro:eq:skew}
F(x,y)=\bigl(f(x),\rho_{\tau(x)}y\bigr).
\end{equation}
Write $\pi(x,y)=x$. We denote the invariant Borel probability measures of a map
$T\colon Z\to Z$ by $\mathcal M(T,Z)$ and its ergodic invariant probability measures by $\mathcal M_e(T,Z)$.
We write $H(\mathbf t)=h_{\top}(\rho_{\mathbf t})$ for the entropy of an action element.
A measure $\zeta\in\mathcal M(F,X\times Y)$ is a \emph{measure of maximal entropy} (MME) if
$h_\zeta(F)=h_{\top}(F)$.

The product conclusion is the main rigidity feature. It determines the
conditional measure on almost every fiber, not merely the number of
maximizing measures. Our proof first identifies the possible projections
through an entropy variational problem on the base, and then uses maximal
fiber entropy to classify their lifts.

\subsection{Main theorem}\label{intro:main-theorem}

Suppose that $\rho$ preserves a smooth probability measure $m_{\MME}$ which is
ergodic for the entire $\R^k$-action and is a measure of maximal entropy for
every element $\rho_{\mathbf t}$.
If $\chi_1,\ldots,\chi_L$ are the transverse Lyapunov functionals of the action
with respect to $m_{\MME}$, with multiplicities $m_1,\ldots,m_L$, the Pesin entropy
formula gives
\begin{equation}\label{intro:eq:action-entropy}
H(\mathbf t)=h_{m_{\MME}}(\rho_{\mathbf t})
=\sum_{i=1}^L m_i\max\{\chi_i(\mathbf t),0\}.
\end{equation}
Thus $H$ is convex and piecewise linear. Since $H(-\mathbf t)=H(\mathbf t)$,
its zero set is the linear subspace
\[
Z_\rho=\{\mathbf t:H(\mathbf t)=0\}
=\bigcap_{i=1}^L\ker\chi_i.
\]

Recall that two cocycles $\tau$ and $\tau'$ are cohomologous if
$\tau-\tau'=u-u\circ f$ for a continuous map $u\colon X\to\R^k$.

\begin{definition}\label{def:jointly nontrivial}
The pair $(\tau,\rho)$ is {\bf jointly nontrivial} if
\begin{enumerate}[(C1)]
\item $\tau$ is not cohomologous to a cocycle taking values in a proper affine subspace of $\mathbb{R}^k$. 
Equivalently, for every nonzero $D\in\mathbb R^k$,
the scalar-valued function $D\cdot\tau$ is not cohomologous to a constant; 
\item $\dim Z_\rho<k$ and every $\rho_{\mathbf s}$ with $H(\mathbf s)>0$ is accessible. 
\end{enumerate}
Here accessibility means that any two points of $Y$ can be joined by a finite
concatenation of paths lying in stable or unstable leaves of the indicated map.
\end{definition}

\begin{theorem}[Finiteness and product rigidity]\label{main theorem}
Let $f$ be a topologically mixing $C^{1+\beta}$ Anosov diffeomorphism of
a closed manifold $X$, where $\beta>0$. Let $\rho$ be a locally free
$C^{1+\beta}$ $\R^k$-action on a closed manifold $Y$, preserving a smooth
probability measure $m_{\MME}$ which is ergodic for the action and maximizes entropy
for every $\rho_{\mathbf t}$. Assume that each $\rho_{\mathbf t}$ is
$C^\infty$ or hyperbolic.

If $\tau\colon X\to\R^k$ is $C^{1+\beta}$ and $(\tau,\rho)$ is jointly
nontrivial, then the skew product \eqref{intro:eq:skew} has a nonempty finite set of
ergodic MMEs. Every such measure is
\[
\zeta=\mu\times m_{\MME},
\]
where $\mu$ is an equilibrium state of $f$ for a H\"older continuous potential.
\end{theorem}

\begin{remark}\label{main theorem-remark 1}
Writing $N(F)$ for the number of ergodic MMEs, one has
\begin{equation}\label{intro:eq:bound}
1\leq N(F)\leq 2\sum_{j=0}^{k-1}\binom{\dim Y-k-1}{j}.
\end{equation}
\end{remark}

Condition (C1) rules out cohomological reduction of the cocycle
to a proper affine subspace of $\mathbb R^k$. 
Condition (C2) is used to pass from rigidity along stable and unstable leaves to rigidity on the whole fiber.
The two conditions therefore enter at different stages of the proof.
The examples in Section \ref{sect:homogeneous-bounds} show that, without joint nontriviality, uncountably
many ergodic MMEs can occur. The bound \eqref{intro:eq:bound} comes from the
Lyapunov hyperplane arrangement, rather than from an estimate on the number
of invariant lifts.

%\subsection{Homogeneous fibers}\label{subsect: Detailed statement of main results}

We now consider examples where the fibers are homogeneous spaces. 

\begin{example}\label{DDkN22-2-Example 8}
Identify $\R^k$ with the positive diagonal subgroup of
$\SL(k+1,\R)$, and let $\rho$ act by left translation on
$Y=\SL(k+1,\R)/\Gamma$, where $\Gamma$ is a cocompact lattice.
For $k=1$ this is the homogeneous model of the geodesic flow; for $k\geq2$
it is the Weyl chamber flow; see \cite[Example 8]{DDKN22-2}.
Haar probability measure is an MME for every
element, and every nonzero parameter has positive entropy. The action is exponentially
mixing of all orders \cite{BEG20}. The stable and
unstable root directions give the accessibility required below.

\end{example}

\begin{theorem}\label{main theorem-SL-corollary-1}
Let $f$ be a topologically mixing $C^{1+\beta}$ Anosov diffeomorphism,
and let $\rho$ be the diagonal action on
$\SL(k+1,\R)/\Gamma$ described above. For a $C^{1+\beta}$ cocycle
$\tau\colon X\to\R^k$ which is not cohomologous 
 to a cocycle taking values in a proper affine subspace of $\mathbb{R}^k$, 
the skew product \eqref{intro:eq:skew} has finitely many ergodic MMEs. 
Each is of the form $\mu\times m_{\operatorname{Haar}}$, 
where $\mu$ is an equilibrium state of $f$ for a H\"older continuous potential.

 Moreover, each of these MMEs is Bernoulli, exponentially
mixing of all orders, and satisfies the central limit theorem.
\end{theorem}

\iffalse
Here the cohomological hypothesis is sharp. If $\tau=u-u\circ f+{\bf c}$, the map
$\Psi(x,y)=(x,\rho_{u(x)}y)$ satisfies
$\Psi\circ F\circ\Psi^{-1}=f\times\rho_{\bf{c}}$.
If $\mu_0$ is the unique MME of $f$, the measures
$\mu_0\times\delta_y$, $y\in Y$, give uncountably many ergodic MMEs of the
conjugate system. We also remark that the proof of Bernoulli properties is given separately in \cite{DQ26}.

We also have the following dichotomy.

{\color{blue}
\begin{corollary}\label{Weyl dichotomy corollary}
In the setting of Example \ref{DDkN22-2-Example 8}, 
let $f$ be a topologically mixing $C^{1+\beta}$ Anosov diffeomorphism and let $\tau$ be any
$C^{1+\beta}$ cocycle. Precisely one of the following alternatives holds:

\begin{enumerate}[(i)]
\item $\tau$ is not cohomologous to a constant, 
the conclusions of Theorem \ref{main theorem-SL-corollary-1} hold, 
$h_{\top}(F)>h_{\top}(f)$, 
and for every ergodic MME $\mu\times m_{\operatorname{Haar}}$ of $F$, 
the base component $\mu$ is not the MME of $f$.

\item $\tau$ is cohomologous to a nonzero constant, 
$h_{\top}(F)>h_{\top}(f)$, and $F$ has a unique 
MME $\mu\times m_{\operatorname{Haar}}$, where 
$\mu$ is the MME of $f$. 

\item $\tau$ is cohomologous to zero, $h_{\top}(F)=h_{\top}(f)$, 
and $F$ has uncountably many ergodic MMEs. 
\end{enumerate}
\end{corollary} 
}
\fi

In this example $Z_\rho=\{0\}$, so $\dim Z_\rho=0<k$ and (C2) is exactly satisfied. 
%Thus the general nondegeneracy condition has a simple cohomological interpretation.
We also remark that the proof of Bernoulli properties is given separately in \cite{DQ26}.

\subsection{Entropy formulas and measure rigidity}\label{subsect:entropy-projections-intro}

For $\mu\in\mathcal M_e(f,X)$, define its drift by
\[
\mu(\tau)=\int_X\tau\,\rd\mu.
\]
The iterates of $F$ sample the fiber action at the cocycle sums
$\tau_n(x)=\sum_{j=0}^{n-1}\tau(f^jx)$, and
$\tau_n(x)/n\to\mu(\tau)$ for $\mu$-almost every $x$.
The first step is to show that this asymptotic direction determines the
fiber entropy. The distinction between a product measure and an arbitrary
invariant lift is essential.

\begin{theorem}\label{higher le:product}
Let $X$ and $Y$ be closed Riemannian manifolds. Suppose that $f\colon X\to X$,
$\tau\colon X\to\R^k$, and the locally free action $\rho$ on $Y$ are
$C^{1+\beta}$. If $\mu\in\mathcal M_e(f,X)$, $\nu$ is
$\rho_{\mu(\tau)}$-invariant, and $\mu\times\nu$ is $F$-invariant, then
\begin{equation}\label{intro:eq:product}
h_{\mu\times\nu}(F)=h_\mu(f)+h_\nu(\rho_{\mu(\tau)}).
\end{equation}
\end{theorem}

\begin{theorem}[Maximal relative entropy]\label{higher le:fiber-entropy}
Suppose that $f$ is continuous on a compact metric space, $\tau$ is
continuous, and $\rho$ is a locally free $C^{1+\beta}$ action on a closed
manifold. If every nonzero element of the action has the cone property
of Definition \ref{def:cone property}, then, for every $\mu\in\mathcal M_e(f,X)$,
\begin{equation}\label{intro:eq:relative}
\sup_{\substack{\zeta\in\mathcal M(F,X\times Y)\\\pi_*\zeta=\mu}}h_\zeta(F)
=h_\mu(f)+H(\mu(\tau)).
\end{equation}
Consequently,
\begin{equation}\label{intro:eq:top}
h_{\top}(F)=\sup_{\mu\in\mathcal M_e(f,X)}\bigl\{h_\mu(f)+H(\mu(\tau))\bigr\}.
\end{equation}
\end{theorem}

The cone property controls the entropy of action trajectories whose parameters
have a prescribed asymptotic direction. We establish it through
$r$-neutralized entropy, which uses Bowen balls with exponentially shrinking
radii; see Section \ref{sect:neutralized-entropy} and \cite{DQ25}. The smooth and hyperbolic hypotheses in Theorem \ref{main theorem} enter
at this point. The proofs of both entropy formulas are given in
Section \ref{subsect: Entropy formulas for skew products}.
The relative variational principle then yields
\eqref{intro:eq:relative}, including an upper bound for lifts whose fiber
conditional measures may depend on the base point.

Formula \eqref{intro:eq:top} reduces the possible projections of ergodic MMEs to
the maximizers of
\begin{equation}\label{intro:eq:functional}
\mathcal E(\mu)=h_\mu(f)+H\!\left(\int_X\tau\,\rd\mu\right).
\end{equation}
The second term depends on a vector of integrals and need not be the integral
of one fixed potential. This is a nonlinear thermodynamic problem of the type
studied in \cite{BKL23,Ku24}; see also \cite{Jen01,KW14,KW15} for rotation theory. Its geometry is governed by the rotation set
\[
\rot(\tau)=\left\{\int_X\tau\,\rd\mu:\mu\in\mathcal M(f,X)\right\}
\]
and the finitely many linearity regions of $H$.

On each such region, the fiber entropy is a linear combination of the
Lyapunov functionals. The corresponding potential on the base is therefore
a H\"older linear combination of the components of $\tau$. Combining this
observation with thermodynamic and rotation arguments gives finitely many
maximizing projections, each a classical equilibrium state
(Theorem \ref{goal theorem}). Counting the
regions determined by the Lyapunov hyperplanes gives \eqref{intro:eq:bound}.
The argument also shows that every maximizing projection has
$H(\mu(\tau))>0$, and hence $\mu(\tau)\ne0$.

This establishes finiteness on the base. It does not yet establish finiteness
for $F$: several, or even infinitely many, ergodic invariant lifts could
in principle project to the same $\mu$.

The second step identifies the lifts by using the equality case of the
fiber entropy bound. Let $\widehat W^u$ denote the unstable foliation within
the fibers, tangent to the positive fiber Lyapunov subspaces of $F$ and
constructed in Section \ref{subsect: Invariant foliations for the skew product}.
\begin{definition}\label{def: subordinate}
An $F$-invariant measure $\zeta$ has \emph{absolutely continuous conditional measures}
on $\widehat W^u$ if, for every measurable partition $\xi$ subordinate to
$\widehat W^u$, the conditional measure $\zeta_z^\xi$ is absolutely continuous
with respect to the induced Riemannian volume on $\xi(z)$ for $\zeta$-almost every $z$.
\end{definition}

\begin{theorem}[Relative Ledrappier--Young theorem]\label{relative SRB}
For a $C^{1+\beta}$ skew product \eqref{intro:eq:skew}, let $\zeta\in\mathcal M_e(F,X\times Y)$
and $\mu=\pi_*\zeta$. The conditional measures of $\zeta$ on
$\widehat W^u$ are absolutely continuous if and only if
\begin{equation}\label{relative SRB-(1)}
h_\zeta(F)-h_\mu(f)
=\sum_{i\in\operatorname{LY}_{\operatorname{fiber}}^+}\chi_i\dim E_i,
\end{equation}
where the sum is taken over the positive fiber Lyapunov exponents $\chi_i$,
and $E_i$ are the corresponding Oseledets subspaces with respect to $\zeta$.
\end{theorem}

The theorem is a relative version of the entropy characterization of
absolutely continuous unstable conditional measures in \cite{LY85-1,LY85-2}.
The $C^{1+\beta}$ version uses the holonomy regularity established in \cite{Brown22}.
Applied to a maximizing lift, and also to the inverse dynamics, the relative
formula supplies absolute continuity along the unstable and stable
fiber foliations. The fiberwise rigidity argument uses accessibility to
propagate this information across the fiber. Together with invariance and
the ergodicity of the reference action, it identifies the disintegration as
\[
\zeta=\int_X\zeta_x\,\rd\mu(x),\qquad
\zeta_x=m_{\MME}\quad\text{for $\mu$-almost every }x.
\]
This is the product rigidity asserted in Theorem \ref{main theorem}; see
Proposition \ref{unique measure} for the identification of the maximizing lift.

\subsection{Relation to previous work}\label{subsect: Comparison with prior works}
The entropy decomposition is rooted in the Abramov--Rokhlin formula
\cite{AbR62} and the relative variational principle of Ledrappier and Walters
\cite{LW77}. For skew products driven by abelian actions, Park studied
directional entropy formulas and exhibited a counterexample in the
continuous category \cite{Park96,Park98}. These results explain why replacing
the cocycle sums by their average direction requires an entropy argument,
in addition to the ergodic theorem.

Recent finiteness results for smooth dynamics use the symbolic theory of
strong positive recurrence (SPR) developed by Buzzi, Crovisier, and Sarig
\cite{BCS25}. Marin, Poletti, and Veiga \cite{MPV25} obtain finiteness and
statistical properties for skew products with surface fibers under an entropy
gap. Castro and Froyland \cite{CF25} study the cardinality of hyperbolic
measures of maximal relative entropy. We also refer the reader to \cite{MN79} for
the study of MMEs in the nonsmooth setting. Our fibers are not restricted to
surfaces, but their maps belong to a common abelian action. This structure
allows us to compute the fiber entropy and identify the maximizing measures
explicitly. We do not require an SPR coding of the full skew product.

The rigidity argument is also related to the use of entropy and leafwise
conditional measures in higher-rank dynamics \cite{EK03} and in random
surface dynamics \cite{BR17}. In the present setting, maximal relative
entropy supplies the equality in \eqref{relative SRB-(1)}; accessibility then enters
the identification of the fiber disintegration. This is the specific link
between thermodynamic selection and measure rigidity used here.

We also mention that,
the focus on MMEs is fundamentally intertwined with the 
requirements of measure rigidity. 
Without this ``high entropy'' threshold, the invariant measure risks 
collapsing onto the orbit of a smaller proper subgroup via Ratner-type mechanisms, 
posing a major obstruction to global rigidity.

\subsection{Statistical properties and extensions}\label{subsect:applications-intro}

As an application of our main result, we have the following corollary on the statistical properties.
\begin{corollary}\label{main theorem-mixing}
Assume the hypotheses of Theorem \ref{main theorem} and suppose that $f$ is
topologically mixing. If the action $\rho$ is exponentially mixing of all
orders with respect to $m_{\MME}$, then every ergodic MME of $F$ is Bernoulli,
exponentially mixing of all orders, and satisfies the central limit theorem.
If the action is assumed only to be exponentially mixing, every ergodic MME
is Bernoulli and exponentially mixing.
\end{corollary}

The deduction uses the product description, the fact that the base
measure is a Gibbs equilibrium state, and its nonzero drift, together with \cite[Theorems 1.1 and 1.2]{DQ26}. For diagonal
actions on the homogeneous spaces in Theorem \ref{main theorem-SL-corollary-1}, the required
multiple mixing is supplied by \cite{BEG20}.
The role of the maximizing measure is significant: the zero-drift
constructions in \cite{DDKN22-2} include smooth skew products which satisfy
the central limit theorem and are $K$ but not Bernoulli. Those examples
concern selected invariant products; here entropy maximization gives
products with positive fiber entropy and nonzero drift.

The same framework applies to homogeneous actions beyond the diagonal
examples. Haar measure supplies the smooth measure of maximal entropy for
the fiber action; see \cite[Theorem 9]{Bow71-1}.

\begin{theorem}\label{main theorem-algebra}
Let $f\colon X\to X$ be a topologically mixing $C^{1+\beta}$ Anosov
diffeomorphism of a closed manifold. Let $G$ be a semisimple Lie group
without compact factors, $\Gamma<G$ a cocompact lattice, and
$\rho\colon\R^k\to G$ a continuous homomorphism inducing a locally free
action on $Y=G/\Gamma$. If $\tau\colon X\to\R^k$ is $C^{1+\beta}$ and
$(\tau,\rho)$ is jointly nontrivial, then $F$ has finitely many ergodic MMEs.
Every such measure is $\mu\times m_{\operatorname{Haar}}$, where $\mu$ is an
equilibrium state of $f$ for a H\"older continuous potential.
Under the mixing hypotheses of Corollary \ref{main theorem-mixing}, the
corresponding statistical conclusions hold.
\end{theorem}

We also consider potentials pulled back from the base. This restriction
leaves the fiberwise entropy maximization unchanged, while replacing
$\mathcal E(\mu)$ by
\[
h_\mu(f)+\int_X\psi\,\rd\mu+H(\mu(\tau)),
\qquad \phi=\psi\circ\pi.
\]

\begin{theorem}\label{main theorem-equilibrium state}
Under the hypotheses of Theorem \ref{main theorem}, let
$\phi=\psi\circ\pi$ be H\"older continuous. Then $(F,\phi)$ has a nonempty
finite set of ergodic equilibrium states. Each is $\mu\times m_{\MME}$, where
$\mu$ is an equilibrium state of $f$ for a H\"older continuous potential.
The same conclusion holds if the smoothness or hyperbolicity assumption on
each $\rho_{\mathbf t}$ is replaced by the requirement that it satisfy at
least one of the following conditions:
\begin{enumerate}[(i)]
\item $\rho_{\mathbf t}$ is $C^\infty$;
\item $Y$ is a hyperbolic set for $\rho_{\mathbf t}$;
\item $\rho_{\mathbf t}$ preserves a $\rho_{\mathbf t}$-neutralized homogeneous
measure, in the sense of Definition \ref{neutralized homogeneous}.
\end{enumerate}
Under the mixing hypotheses of Corollary \ref{main theorem-mixing}, the corresponding
statistical conclusions hold for these equilibrium states.
\end{theorem}

For comparison, equilibrium states of potentials constant on fibers have
been studied for isometric extensions by Spatzier and Visscher
\cite{SV18}, and for center isometries by Carrasco and Rodriguez Hertz
\cite{CRH24}. Here the fibers may contribute positive entropy, and this
contribution changes the equilibrium problem on the base.

The arguments also apply over mixing hyperbolic basic sets of Axiom A
diffeomorphisms. For a base with several basic components, one applies the
results componentwise and compares their pressures; the statistical
conclusions retain their mixing hypotheses. Finally, the weak variational
principle for $r$-neutralized entropy (Conjecture \ref{weak conjecture}) would remove
the additional alternatives in Theorem \ref{main theorem-equilibrium state}. This last
extension is conditional on that conjecture.

It is worth mentioning that, 
as stated in \cite[Theorem 5.10.3]{BS02}, 
for connected closed Riemannian manifolds, 
topological mixing is equivalent to topological transitivity. 

\paragraph{Organization of the paper.}
Section \ref{sect:neutralized-entropy} develops $r$-neutralized entropy. Section \ref{sect: Entropy Formulas for the Skew Product} establishes the cone
property and the entropy formulas. Section \ref{sect: Invariant foliations and MME projections} classifies the maximizing base
measures by nonlinear thermodynamic formalism. Section \ref{sect:relative-LY} proves the relative
Ledrappier--Young theorem, and Section \ref{sect: Fiberwise and global measure rigidity} establishes fiberwise and product
rigidity. Section \ref{sect:homogeneous-bounds} treats homogeneous actions, the bound on the number of
ergodic MMEs, and examples illustrating the role of joint nontriviality.

\paragraph{Disclosure on AI.}
This work was originated while the first named author was a postdoc at UMD back in 2021, and mostly completed in 2025. AI tools (ChatGPT) were only used to check grammar and improve the writing of this manuscript.  All mathematical ideas
and results are due to the authors, who take full responsibility for the correctness.

\paragraph{Acknowledgements.}
We thank Dmitry Dolgopyat and Adam Kanigowski for introducing us to this
problem, and Snir Ben Ovadia for bringing the weak variational principle for
neutralized entropy to our attention and for his valuable remarks at the
2025 Workshop on Dynamical Systems and Mathematical Physics at Great Bay
University. This research was supported by the National Key R\&D Program
of China (No. 2024YFA1015100). C. Dong was also supported by the Nankai
Zhide Foundation, the Fundamental Research Funds for the Central Universities, the Tianjin Natural
Science Foundation Outstanding Youth Project (25JCJQJC00310), and NSFC
12526207. Q. Qiao was supported by the Nankai Zhide Foundation.

\section{Neutralized entropy}\label{sect:neutralized-entropy}

This section develops the entropy theory needed for the cone property. 
Its core topics include 
sequences of action parameters with a fixed asymptotic direction, 
and $r$-neutralized topological entropy.

Let $\{\rho_{{\bf t}}\}_{{\bf t}\in\mathbb{R}^k}$ be a locally free $C^{1+\beta}$ $\mathbb{R}^k$-action on a closed Riemannian manifold $(Y,d)$.
Given a sequence $A:=\{{\bf a}(n)\}_{n=1}^\infty\subset\mathbb{R}^k$,
we set ${\bf a}(0)={\bf 0}=(0,\dots,0)\in\mathbb{R}^k$. For each $n\in\mathbb{N}$, we define the \textbf{$n$-sequence Bowen metric} $d_n^A$ by
\begin{align*}
d_n^A(x,y):=\max\Bigl\{d\bigl(\rho_{{\bf a}_i}(x),\rho_{{\bf a}_i}(y)\bigr): 0\leq i\leq n-1\Bigr\},\quad \forall\,x,y\in Y.
\end{align*}
Given $n\in\mathbb{N}$, $\epsilon>0$ and $x\in Y$, the \textbf{sequence Bowen ball} is defined as
\begin{align*}
B(x,A,n,\epsilon):=\bigl\{y\in Y: d_n^A(x,y)<\epsilon\bigr\}.
\end{align*}
A nonempty set $E$ is said to be \textbf{$(A,n,\epsilon)$-spanning} if
$\displaystyle Y\subset \bigcup_{x\in E} B(x,A,n,\epsilon)$.

Let $S_d\bigl(\{\rho_{{\bf t}}\},A,n,\epsilon\bigr)$ denote the minimal cardinality of an $(A,n,\epsilon)$-spanning set.
The \textbf{topological sequence entropy} of $\{\rho_{{\bf t}}\}$ with respect to the sequence $A$ is defined by
\begin{align*}
h_{\operatorname{top}}^A\bigl(\{\rho_{{\bf t}}\}\bigr):=\lim_{\epsilon\to 0}\varlimsup_{n\to\infty} \frac{1}{n}\log S_d\bigl(\{\rho_{{\bf t}}\},A,n,\epsilon\bigr).
\end{align*}
For a Borel probability measure $\mu$, 
the \textbf{local sequence entropy} $h_{\mu}^A\bigl(\{\rho_{{\bf t}}\},x\bigr)$ 
is defined by
\begin{align*}
h_{\mu}^A\bigl(\{\rho_{{\bf t}}\},x\bigr)&:=\lim_{\epsilon\to 0}\varlimsup_{n\to\infty}-\frac{1}{n}\log \mu\bigl(B(x,A,n,\epsilon)\bigr).
\end{align*}

Throughout the remainder of the paper, unless otherwise specified,
all sequences $A=\{{\bf a}(n)\}_{n=1}^\infty$ we consider satisfy the following conditions:
\[
K_0=\sup_{i\geq 1}\bigl\|{\bf a}(i+1)-{\bf a}(i)\bigr\|<\infty,
\]
and there exists ${\bf a}\in\mathbb{R}^k$ such that
$\displaystyle \frac{{\bf a}(n)}{n}\to{\bf a}$ as $n\to\infty$.

\subsection{Neutralized sequence entropy}\label{subsect: Neutralized sequence entropy}

Motivated by the neutralized entropy introduced by Ben Ovadia and Rodriguez Hertz \cite{ORH23}, 
we introduce its sequence version for $\mathbb R^k$-actions. 
This will be used to link the entropy of sequences asymptotic to 
a fixed direction to the entropy of the corresponding action element.

Let $\{\rho_{\bf t}\}$ be a locally free $C^{1+\beta}$ action of $\mathbb{R}^k$ on a manifold $Y$. 
Fix a $\{\rho_{\bf t}\}$-invariant Borel probability measure $\mu$. 
The {\bf neutralized local sequence entropy} is defined by
\begin{align*}
\mathcal{E}_{\mu}^A\bigl(\{\rho_{\bf t}\},x\bigr)&:=\lim_{r\to 0}\varlimsup_{n\to\infty}-\frac{1}{n}\log \mu\bigl(B(x,A,n,e^{-nr})\bigr).
\end{align*}

We define
\[
\Upsilon  :=\sup\Bigl\{\bigl\|d\rho_{\bf t}|_{T_yY}\bigr\| \;:\; y\in Y,\ {\bf t}\in\mathbb{R}^k,\ \|{\bf t}\|\leq 1\Bigr\}+1<\infty.
\]
The supremum is finite due to the compactness of $Y$, 
and we add $1$ to guarantee that $\Upsilon$ is strictly larger than all such operator norms.
Fix an element ${\bf a}\in \mathbb{R}^k$. For all $x\in Y$, $r>0$, 
and for all sufficiently large $n$, we have 
\begin{align*}
B(\rho_{\bf a}x,A,n,e^{-2nr})
&\subset B(\rho_{\bf a}x,A,n,\Upsilon^{-|\bf a|}e^{-nr})\subset \rho_{\bf a}\bigl(B(x,A,n,e^{-nr})\bigr)\\
&\subset B(\rho_{\bf a}x,A,n,\Upsilon^{|\bf a|}e^{-nr}) \subset B(\rho_{\bf a}x,A,n,e^{-\frac{nr}{2}}).
\end{align*}
This implies that $\mathcal{E}_{\mu}^A\bigl(\{\rho_{\mathbf{t}}\},x\bigr)$ 
is a $\rho_{\mathbf{a}}$-invariant function. 
Due to the arbitrariness of $\mathbf{a}\in \mathbb{R}^k$, 
$\mathcal{E}_{\mu}^A\bigl(\{\rho_{\mathbf{t}}\},x\bigr)$ 
is constant $\mu$-almost everywhere for any ergodic 
$\{\rho_{\mathbf{t}}\}$-invariant Borel probability measure $\mu$. 
A similar result holds for $h_{\mu}^A\bigl(\{\rho_{{\bf t}}\},x\bigr)$. 
These two almost everywhere constants are denoted by 
$\mathcal{E}_{\mu}^A\bigl(\{\rho_{\mathbf{t}}\}\bigr)$ and $h_{\mu}^A\bigl(\{\rho_{\mathbf{t}}\}\bigr)$, respectively.

Fix $r>0$ and $x\in Y$. Since $\dfrac{{\bf a}(n)}{n}\to \bf a$, for all sufficiently large $n$,
\begin{align}\label{Neutralized sequence entropy-(1)}
B(x,A,n,e^{-nr})\supset B\bigl(x,\Upsilon^{-|{\bf a} n|-n}e^{-nr}\bigr).
\end{align}
Recall that $m_Y$ denotes the dimension of the manifold $Y$.
For any $\rho_{\bf a}$-invariant measure $\mu$, 
\[
\varlimsup_{\epsilon\to 0}\frac{\log \mu\bigl(B(x,\epsilon)\bigr)}{\log \epsilon}\leq m_Y
\]
holds for $\mu$-almost every $x\in Y$.
Combining this inequality with \eqref{Neutralized sequence entropy-(1)}, we conclude that 
$h_{\mu}^{A}\bigl(\{\rho_{\bf t}\},x\bigr)$ and 
$\mathcal{E}_{\mu}^{A}\bigl(\{\rho_{\bf t}\},x\bigr) $ 
are finite for $\mu$-almost every $x\in Y$.

Ben Ovadia and Rodriguez Hertz established the Besicovitch--Bowen covering lemma
and the Bowen-Lebesgue density lemma for Bowen balls $B(f,\cdot,n,e^{-nr})$ in \cite{ORH23}.
A related version of the Bowen-Lebesgue density lemma can be found in \cite[Lemma 2.2]{DQ25}.
Since Lyapunov exponents are linear functionals and $\dfrac{{\bf a}(n)}{n}\to \bf a$,
the volume ratio between $B(x,A,n,e^{-nr})$ and $B(\rho_{\bf a},x,n,e^{-nr})$
is bounded by an exponential term of small order.  
To formulate the corresponding adaptations precisely, we first recall
the Lyapunov charts and Pesin sets associated with the
$\mathbb R^k$-action.

A point $x\in Y$ is called {\bf regular} for $\{\rho_{\bf t}\}$ if there exist linear functionals
$\chi_i\colon\mathbb{R}^k\to\mathbb{R}$, $i=1,\dots, r(x)$,
and a decomposition $T_x Y=E_1(x) \oplus \cdots \oplus E_{c(x)}(x)$
such that for every ${\bf t}\in\mathbb{R}^k$ and every nonzero vector $v\in E_i(x)$,
\[
\lim_{n\to \pm \infty}\frac{1}{n}\log\|D_x(\rho(n{\bf t}))v\|=\chi_i({\bf t}).
\]
We denote by $\widehat{\Gamma}_Y$ the set of regular points.
By Proposition \ref{chart1}, $\mu(\widehat{\Gamma}_Y)=1$ for any ergodic $G$-invariant measure $\mu$.
Similarly, for a diffeomorphism $f$, 
we may define its regular points, and denote by $\widehat{\Gamma}_X$ the set of such points.

Next, we introduce the Lyapunov chart associated with the $\mathbb{R}^k$-action, 
which will be used to define the Pesin chart of the skew product. 

Let $\{{\bf t}_1, \cdots, {\bf t}_k\}$ be the standard basis 
for $\mathbb{R}^k$ with respect to 
the standard norm $\|\cdot \|$ on $\mathbb{R}^k$. 
For each ${\bf t}_i$, we denote the corresponding Lyapunov exponents by 
$$
\chi_1({\bf t}_i)>\cdots>\chi_{c({\bf t}_i)}({\bf t}_i).
$$ 
For each $i$, we define 
$$
\chi_+({\bf t}_i)=\min\{\chi_i({\bf t}_i): \chi_i({\bf t}_i)>0\},\quad 
\chi_-({\bf t}_i)=\max\{\chi_i({\bf t}_i): \chi_i({\bf t}_i)<0\},
$$ 
and 
$$
\Delta({\bf t}_i)=\min\{\chi_j({\bf t}_i)-\chi_{j+1}({\bf t}_i), j=1, \cdots, c({\bf t}_i)-1\}.
$$

Let $m_Y$ denote the dimension of the manifold $Y$ 
and let $|\cdot |$ be the usual norm on $\mathbb{R}^{m_Y}$. 
For $\rho>0$, let $B(\rho)$ denote the ball in $\mathbb{R}^{m_Y}$ 
centered at the origin of radius $\rho$.

\begin{proposition}\label{chart2}
For any $\epsilon>0$ satisfying 
$$
\epsilon\leq \epsilon_0:=\frac{1}{200m_Y k}\min\{\Delta({\bf t}_i), \chi_{+}({\bf t}_i), -\chi_{-}({\bf t}_i), \; i=1,\cdots, k\},
$$
 there exists a measurable function $l:\Gamma'\to [0,\infty)$ such that 
$\displaystyle l(\rho({\bf t})x)\leq l(x)e^{\epsilon \|{\bf t}\|}$, 
together with a family of embeddings $\Phi_x:B(l(x)^{-1})\to Y$ 
defined for each $x\in \Gamma_Y$, for which the following properties hold: 
\begin{itemize}
\item[(i)] $\Phi_x(0)=x$. The preimages
\[
D\Phi_x(0)^{-1}(T\mathcal{O}(x))
\quad\text{and}\quad
R_i(x)=D\Phi_x(0)^{-1}(E_i(x))
\]
of $T\mathcal{O}(x)$ and $E_i(x)$, respectively, are mutually orthogonal in $\mathbb{R}^m$,
where $E_i(x)$ stands for the Lyapunov subspace 
associated with the corresponding Lyapunov exponent in the Lyapunov decomposition. 

\item[(ii)] Let $\tilde{{\bf t}}_x=\Phi_{\rho({\bf t})x}^{-1}\circ \rho({\bf t})\circ \Phi_x$
be the chart map between the local coordinates at $x$ and $\rho({\bf t})x$.
Then, for any ${\bf s},{\bf t}\in\mathbb{R}^k$, we have 
\[
(\widetilde{{\bf t}+{\bf s}})_x
=\tilde{{\bf t}}_{\rho({\bf s})x}\circ \tilde{{\bf s}}_x
=\tilde{{\bf s}}_{\rho({\bf t})x}\circ \tilde{{\bf t}}_x
\]
and
\[
(\widetilde{{\bf t}+{\bf s}})_x^{-1}
=\tilde{{\bf s}}_x^{-1}\circ \tilde{{\bf t}}_{\rho({\bf s})x}^{-1}
=\tilde{{\bf t}}_x^{-1}\circ \tilde{{\bf s}}_{\rho({\bf t})x}^{-1}.
\]

\item[(iii)] For every integer $q$ with $1\leq q\leq m$, 
every nonzero vector $u\in E_q(x)$, and every $v\in T\mathcal{O}(x)$,
\[
|u|e^{\chi_q(x, {\bf t})-\epsilon \|{\bf t}\|}
\leq |D\tilde{{\bf t}}_x(0)u|
\leq |u|e^{\chi_q(x, {\bf t})+\epsilon \|{\bf t}\|},
\quad \forall\,{\bf t}\in\mathbb{R}^k,
\]
\[
|v|e^{-\epsilon \|{\bf t}\|}
\leq |D\tilde{{\bf t}}_x(0)v|
\leq 2|v|,
\]
\[
|u|e^{\chi_q(x, {\bf t})+\chi_q(x, {\bf s})
-\epsilon (\|{\bf t}\|+\|{\bf s}\|)}
\leq |D\widetilde{({\bf t}+{\bf s})}_x(0)u|
\leq |u|e^{\chi_q(x, {\bf t})+\chi_q(x, {\bf s})
+\epsilon(\|{\bf t}\|+\|{\bf s}\|)},\quad \forall\,{\bf t},{\bf s}\in\mathbb{R}^k.
\]

\item[(iv)] Let $L(\Psi)$ denote the Lipschitz constant of a map $\Psi$.
Then for any ${\bf t}\in\mathbb{R}^k$ with $\|{\bf t}\|\leq 1$,
\[
L(\tilde{{\bf t}}_x-D\tilde{{\bf t}}_x(0))\leq \epsilon,
\quad
L(D\tilde{{\bf t}}_x)\leq l(x).
\]

\item[(v)] For all $u, v\in B(l(x)^{-1})$, 
\[
K^{-1}d(\Phi_x u, \Phi_x v)
\leq |u-v|
\leq l(x)\,d(\Phi_x u, \Phi_x v),
\]
where $K$ is a universal constant independent of all other parameters. 
\end{itemize}
\end{proposition}

We call each such family of local charts $\{\Phi_z: z\in \Gamma'\}$ 
the set of $(\epsilon, l)$-charts. 
Fix a sufficiently small $\epsilon > 0$ as in Propositions \ref{chart1} and \ref{chart2}. 
For each $l \ge 1$, we define 
\begin{align*}
\Gamma_{l,\epsilon}:=\left\{z\in \Gamma_Y: l(z)\leq l,\,C_{\epsilon}(z)\leq l,\, K_{\epsilon}(z)\geq \frac{1}{l}\right\}.
\end{align*}

By adapting the arguments from \cite{ORH23,DQ25}, 
we obtain the following lemmas.

\begin{lemma}\label{ORH23-Lemma 2.2}
Fix small $\epsilon>0$ and large $l\in\mathbb{N}$,
and let $\Phi_{x_0}\bigl(B(l(x_0)^{-1})\bigr)$ denote the Lyapunov chart at $x_0\in \Gamma_{l,\epsilon}$.
Let $ F \subset \Phi_{x_0}\bigl(B(l(x_0)^{-1})\bigr) \cap \Gamma_{l,\epsilon} $ be a measurable subset.
Then $F$ can be covered by sequential Bowen balls $B(\cdot, A,n,e^{-n\epsilon})$
with covering multiplicity bounded by $ e^{3 n m_Y\epsilon} $,
where $n$ is sufficiently large depending on $\Gamma_{l,\epsilon}$.
\end{lemma}

\begin{lemma}\label{ORH23-Lemma 2.3}
Fix $\epsilon>0$, let $\mu$ be an ergodic $\rho_{\bf a}$-invariant Borel probability measure,
and let $F$ be a measurable set with $\mu(F)>0$.
Then for $\mu$-almost every $x \in F$,
\[
\lim_{n \rightarrow \infty} -\frac{1}{n} \log \frac{\mu\bigl(B(x,A,n,e^{-nr})\cap F\bigr)}{\mu\bigl(B(x,A,n,e^{-nr})\bigr)}=0.
\]
\end{lemma}

\begin{lemma}\label{ORH23-Lemma 3.2} 
Let $\mu$ be a Borel probability measure which is invariant and
ergodic with respect to the $\mathbb R^k$-action. 
Let $K$ be a measurable set of positive $\mu$-measure, 
and let $\delta > 0$.
Assume that there exists $ \Delta > 0$ such that
$h_{\mu}^{A}\bigl(\{\rho_{\bf t}\},x\bigr) \leq \Delta$ for $\mu$-almost every $x\in K$.
Then, for all sufficiently large $n$ (depending on $\delta$ and $K$),
there exist a measurable subset $K_\delta \subseteq K$,
a finite collection $\mathcal{A}_{n,\delta}$, and $0 < \rho \leq \delta$ such that
\[
\bigcup_{x \in \mathcal{A}_{n,\delta}} B(x,A,n,\rho) \supseteq K_\delta,\quad
\frac{\mu(K_\delta)}{\mu(K)} \geq e^{-\delta},\quad
\# \mathcal{A}_{n,\delta} \leq e^{n(\Delta+\delta)}.
\]
\end{lemma}

We will use the following  auxiliary lemma. Denote by $\lfloor s \rfloor$ the integer part of $s$.

\begin{lemma}\label{NSE-Lemma}
Let $\mu$ be a Borel probability measure which is invariant and
ergodic with respect to the $\mathbb R^k$-action 
such that all nonzero Lyapunov exponents of 
$\rho_{\bf a}$ with respect to $\mu$ are uniformly bounded below in absolute value by some constant $\chi_0>0$.
Let
\[
0<\epsilon<\frac{\chi_0}{100m_Y},\quad r\geq 4\epsilon,\quad
\gamma>\sqrt{r}.
\]
Choose $\ell\in\mathbb{N}$ 
such that and $\mu(\Gamma_{\ell,\epsilon})>0$, 
and fix a measurable subset $K\subset\Gamma_{\ell,\epsilon}$ and $n'$ 
satisfying $\mu(K)>0$ and 
\begin{align*}
K\subset \left\{x\in Y: \forall n\geq n',\, x\in \Gamma_{\ell,\epsilon} \cap 
\left( \bigcup_{{\bf j}:\,n(1+\gamma)\leq \frac{\|{\bf j}\|}{\|{\bf a}\|}\leq n(1+2\gamma)}  \rho_{\bf j}^{-1} \left(\Gamma_{\ell,\epsilon}\right) \right)  \right\}. 
\end{align*}
Assume that $h^A_\mu\bigl(\{\rho_{\bf t}\},x\bigr)\leq\Delta$ for $\mu$-almost every $x\in K$.
Then there exists a subset $K_\epsilon\subseteq K$ with 
$\mu(K_\epsilon)\geq e^{-\epsilon}\mu(K)$,
and for all sufficiently large $n$, 
there exists a finite set $\mathcal{A}_n\subseteq K_\epsilon$ satisfying
\[
\bigcup_{x\in\mathcal{A}_n} B(x,n,e^{-nr})\supseteq K_\epsilon
\quad\text{and}\quad
\#\mathcal{A}_n\leq e^{n(1+2\gamma)(\Delta+3m_Y r)}.
\]
\end{lemma}

\begin{proof}
By Lemma \ref{ORH23-Lemma 3.2}, 
for any sufficiently large $n$ depending on $\epsilon$ and $K$,
there exist a measurable subset $K_\epsilon\subset K$,
a finite collection $\mathcal{A}_{n,\epsilon}$, 
and a constant $0<\rho\leq\epsilon$ such that
\[
\bigcup_{x\in\mathcal{A}_{n,\epsilon}} B\bigl(x,A,\lfloor n(1+2\gamma)\rfloor,\rho\bigr)\supseteq K_\epsilon,\quad
\mu(K_\epsilon)\geq e^{-\epsilon}\mu(K)>0,\quad
\#\mathcal{A}_{n,\epsilon}\leq e^{n(1+2\gamma)(\Delta+\epsilon)}.
\]

Let $C_n$ be a finite Besicovitch cover of $K_\epsilon$ by metric balls of radius $e^{-2nr}$.
For each $B\in C_n$, the intersection $K_\epsilon\cap B$ can be covered by at most $e^{n(\Delta+\epsilon)}$ Bowen balls
of the form $B\bigl(\cdot,A,\lfloor n(1+2\gamma)\rfloor,\rho\bigr)$.
Let $B_d$ be the Besicovitch constant on $Y$, and let $C_Y>0$ be a uniform constant such that
for all $x\in Y$ and all sufficiently large $n$,
\[
\operatorname{Vol}\bigl(B(x,e^{-2nr})\bigr)\geq\frac{1}{C_Y}e^{-2m_Y n r}.
\]
Hence $K_\epsilon$ can be covered by at most 
$B_d C_Y\,e^{2m_Y n r} e^{\lfloor n(1+2\gamma)\rfloor(\Delta+\epsilon)}$ 
such Bowen balls in total.

For each $B\in C_n$, let $\mathcal{A}_{n,\epsilon}^B$ be the corresponding covering set given by Lemma \ref{ORH23-Lemma 3.2} for $B\cap K_\epsilon$.
For every $x\in\mathcal{A}_{n,\epsilon}^B$, 
 and notice that for ${\bf j}_x$ with 
$
n(1+\gamma) \|{\bf a}\|\leq \|{\bf j}_x\|\leq n(1+2\gamma) \|{\bf a}\| 
$, 
\begin{align}\label{NSE-Lemma-(1)}
B(x,A,\bigl\lfloor n(1+2\gamma)\bigr\rfloor,\rho)\cap B \subseteq 
B\!\left(x,A,\left\lfloor \frac{\|{\bf j}_x\|}{\|{\bf a}\|} \right\rfloor,\rho\right)\cap B(x,2e^{-2nr}) 
\subseteq B(x,A,n,e^{-nr})
\end{align}
for all $n$ large enough depending on $l$ and $\epsilon$. 
Inclusion \eqref{NSE-Lemma-(1)} follows from standard arguments in Pesin theory;
see, for instance, \cite[Corollary 3.3]{ORH23}.
The small ball $B(x,2e^{-2nr})$ controls the directions associated with zero Lyapunov exponents,
while the Bowen ball component dominates the directions corresponding to nonzero Lyapunov exponents,
due to the uniform lower bound $\chi_0$ on their absolute values.
Define
\[
\mathcal{A}_n:=\bigcup_{B\in C_n}\mathcal{A}_{n,\epsilon}^B.
\]
Then $\mathcal{A}_n$ covers $K_\epsilon$, and for all sufficiently large $n$, 
 the cardinality of $\mathcal{A}_n$ is bounded above by 
\[
B_d C_Y e^{2m_Y n r} e^{\lfloor n(1+2\gamma)\rfloor(\Delta+\epsilon)}\leq e^{n(1+2\gamma)(\Delta+3m_Y r)}.
\]
This completes the proof.
\end{proof}

We are now ready to establish the main result of this subsection.

\begin{theorem}\label{SNE is SE}
Let $\{\rho_{\bf t}\}_{{\bf t}\in\mathbb{R}^k}$ be a $C^{1+\beta}$ $\mathbb R^k$-action on a closed Riemannian manifold $(Y,d)$.
Let $A=\{{\bf a}(n)\}_{n=1}^\infty$ be a sequence in $\mathbb R^k$ satisfying
\[
K_0=\sup_{i\geq 1}\bigl\|{\bf a}(i+1)-{\bf a}(i)\bigr\|<\infty,
\]
and suppose there exists ${\bf a}\in\mathbb{R}^k$ such that $\dfrac{{\bf a}(n)}{n}\to\bf a$ as $n\to\infty$.
Given any ergodic invariant measure $\mu$ for $\mathbb R^k$-action, 
 for $\mu$-almost every $x\in Y$, we have 
\[
{\mathcal{E}}_{\mu}^A\bigl(\{\rho_{\bf t}\},x\bigr) = {h}_{\mu}^A\bigl(\{\rho_{\bf t}\},x\bigr).
\]
\end{theorem}

\begin{proof}
We first treat the case where all Lyapunov exponents of
$\rho_{\bf a}$ with respect to $\mu$ are non-positive.
Fix $r>0$ arbitrarily.
For $\mu$-almost every $x\in Y$, there exists 
$N=N(x,r)\in\mathbb N$ such that, for every $n\geq N$, 
\begin{align}\label{SNE is SE-(1)}
B(x,e^{-3nr})\subset B(x,A,n,e^{-2nr}).
\end{align}
Recall that $m_Y$ is the dimension of $Y$. For $\mu$-a.e.\ $x\in Y$,
we have
$$
\varlimsup_{n\to\infty}-\frac1n\log\mu\bigl(B(x,e^{-3nr})\bigr)
=3r\varlimsup_{n\to\infty}\frac{\log\mu\bigl(B(x,e^{-3nr})\bigr)}
{\log(e^{-3nr})} \leq 3rm_Y.
$$ 
Combining this with \eqref{SNE is SE-(1)}, we obtain
$$
\varlimsup_{n\to\infty}-\frac1n\log\mu\bigl(B(x,A,n,e^{-2nr})\bigr)\leq 3rm_Y.
$$
Taking the limit as $r\to 0$ and using
$h_{\mu}^A(\{\rho_{\bf t}\},x)
\leq{\mathcal E}_{\mu}^A(\{\rho_{\bf t}\},x)$, we obtain
$$
{\mathcal E}_{\mu}^A\bigl(\{\rho_{\bf t}\},x\bigr)
=
h_{\mu}^A\bigl(\{\rho_{\bf t}\},x\bigr)
=0
$$
for $\mu$-almost every $x\in Y$. 
Note that this includes the case ${\bf a}={\bf 0}$.

We now assume that $\rho_{\bf a}$ has a positive Lyapunov exponent.
Since $\mu$ is ergodic, 
there exists $\chi_0>0$ such that the absolute value 
of every nonzero Lyapunov exponent is bounded below by $\chi_0$.
Moreover, both functions
${\mathcal E}_{\mu}^A(\{\rho_{\bf t}\},x)$ and 
$h_{\mu}^A(\{\rho_{\bf t}\},x)$ are constant $\mu$-almost everywhere.
We denote the constant value of
$h_{\mu}^A(\{\rho_{\bf t}\},x)$ by $h$.

We proceed by contradiction. 
Suppose that ${\mathcal E}_{\mu}^A\bigl(\{\rho_{\bf t}\},x\bigr)>h$ 
for $\mu$-almost every $x$. 
Since ${\mathcal E}_{\mu}^A(\{\rho_{\bf t}\},x)$ is constant $\mu$-almost everywhere, 
we may choose $\lambda>0$ sufficiently small so that
\begin{align}\label{SNE is SE-(2)}
\lambda<\min\left\{\chi_0,\frac{1}{100}\right\},
\quad
8\lambda^2\left(h+\frac{\lambda}{2}\right)
<\frac{\lambda}{8},
\end{align}
and
$$
{\mathcal E}_{\mu}^A\bigl(\{\rho_{\bf t}\},x\bigr)
\geq
h_{\mu}^A\bigl(\{\rho_{\bf t}\},x\bigr)+2\lambda
=h+2\lambda
$$
for $\mu$-almost every $x\in Y$.

Choose $\displaystyle 0<r< \min\left\{ \frac{\lambda}{8m_Y},\lambda^4 \right\}$ 
sufficiently small so that $\mu(F_1)>1-\lambda^4$, where
$$
F_1 := \left\{x\in Y:\varlimsup_{n\to\infty}
-\frac1n\log\mu\bigl(B(x,A,n,e^{-nr})\bigr) >h+\frac{7\lambda}{8}\right\}.
$$
Choose $\displaystyle 0<\epsilon<\min\left\{\frac{\chi_0}{100m_Y},\frac{r}{4}\right\}$, 
and then choose $\ell\in\mathbb N$ sufficiently large so that
$$
\mu(F_2)\geq1-2\lambda^4,\quad F_2:=F_1\cap\Gamma_{\ell,\epsilon}.
$$

We next establish the recurrence property required in
Lemma \ref{NSE-Lemma}. Set $T=\rho_{\bf a}$. 
Since $\mu$ is invariant with respect to $T$, we define 
\begin{align}
q(x):=\lim_{m\to\infty} \frac{1}{m}\sum_{j=0}^{m-1}\mathbf 1_{F_2}(T^jx). 
\end{align}
By the Birkhoff ergodic theorem, $q(x)$ is well defined for almost every $x\in Y$ and 
$$
\int_Y q(x)\,\rd\mu=\mu(F_2)\geq 1-2\lambda^4. 
$$
We define 
$$
H:=\left\{ x:q(x) \,\, \text{is well defined and}\, \, q(x)>1-2\lambda^2 \right\}
$$
Then the set $H$ is $T$-invariant 
and $\mu(H)>1-2\lambda^2$. Indeed, otherwise, we have 
\begin{align*}
1-2\lambda^4 &\leq \mu(F_2)=\int_Y q(x)\,\rd\mu
= \int_{H}q(x)\,\rd\mu+\int_{Y\setminus H}q(x)\,\rd\mu \\
&\leq \mu(H) +(1-\mu(H))\left( 1-2\lambda^2 \right)
= 1+2\lambda^2 \mu(H)-2\lambda^2\\
&=1+2\lambda^2 (1-2\lambda^2)-2\lambda^2=1-4\lambda^4, 
\end{align*}
which yields a contradiction. Thus, we have $\mu(H)>1-2\lambda^2$. 
Then we obtain 
\begin{align*}
\mu(F_2\cap H)\geq 1-2\lambda^2-2\lambda^4>1-\frac{5}{2}\lambda^2. 
\end{align*}
Therefore, by Egorov's theorem, after increasing $n_0$ if necessary, 
we can choose a measurable set $F_3\subset F_2\cap H$ such that
$\mu(F_3)\geq1-3\lambda^2$ and for any $x\in F_3$ and $m\geq n_0$,  
\begin{align}\label{NSE-Lemma-(3)}
\left| \frac{1}{m}\sum_{j=0}^{m-1} \mathbf 1_{F_2}(T^jx) -q(x) \right|<\frac{\lambda^2}{100}. 
\end{align}

We claim that, after increasing $n_0$ if necessary, 
for every $x\in F_3$ and every $n\geq n_0$, 
there exists $j\in\mathbb N$ such that 
$$
n(1+4\lambda^2)\leq j\leq n(1+8\lambda^2) \quad \text{and} 
\quad T^jx\in F_2. 
$$
Indeed, otherwise, for infinitely many $n$, there would be no
visit to $F_2$ during this interval. Setting
$m_n=\lfloor n(1+8\lambda^2)\rfloor$, we would obtain
$$
\frac1{m_n}
\sum_{j=0}^{m_n-1}\mathbf 1_{F_2}(T^jx)
\leq
\frac{n(1+4\lambda^2)+1}
     {n(1+8\lambda^2)-1}
<1-3\lambda^2<1-2\lambda^2-\frac{1}{100}\lambda^2
$$
for all sufficiently large such $n$, contradicting \eqref{NSE-Lemma-(3)} and the definition of $H$.

Then for every $x\in F_3$ and every $n\geq n_0$, 
there exists $j\in \mathbb{N}$ satisfying
\begin{align}\label{NSE-Lemma-(4)}
n(1+4\lambda^2) \leq j\leq n(1+8\lambda^2)
\quad \text{and} \quad \rho_{j{\bf a}}(x)\in F_2.
\end{align}
Set $\gamma=4\lambda^2$. Since $r<\lambda^4$, we have $\gamma=4\lambda^2>\sqrt r$. 
For ${\bf j}=j{\bf a}$, \eqref{NSE-Lemma-(4)} yields 
$$
n(1+\gamma) \leq \frac{\|{\bf j}\|}{\|{\bf a}\|} \leq n(1+2\gamma),
$$
and $\rho_{\bf j}(x)\in\Gamma_{\ell,\epsilon}$.
Thus $F_3$ satisfies the recurrence assumption of
Lemma \ref{NSE-Lemma}.

Applying Lemma \ref{NSE-Lemma} to $K=F_3$ with
$ \displaystyle \Delta=h+\frac{\lambda}{8}$, 
we obtain a measurable subset $K_\epsilon\subset F_3$ with
$$
\mu(K_\epsilon)
\geq e^{-\epsilon}\mu(F_3)>0.
$$
By the construction in the proof of Lemma \ref{NSE-Lemma}, there
exists $N_1\in\mathbb N$ such that, for every measurable subset
$K'\subset K_\epsilon$ and every $n\geq N_1$, $K'$ can be covered
by sequence Bowen balls
$B(x,A,n,e^{-nr})$, $x\in\mathcal A_n(K')\subset K'$, with
$$
\#\mathcal A_n(K') \leq e^{n(1+2\gamma) \left(h+\frac{\lambda}{8}+3m_Yr\right)}.
$$

Let $N\geq\max\{n_0,N_1\}$ be sufficiently large. 
For each $x\in K_\epsilon$, we define
\[
n_x^N := \min\left\{n \geq N : -\frac{1}{n} \log \mu\bigl(B(x,A,n,e^{-nr})\bigr) > h + \frac{6\lambda}{8}\right\},
\]
and set
\[
K_n := \bigl\{x \in K_\epsilon : n_x^N = n\bigr\},\quad \forall\,n\geq N.
\]
Since $K_\epsilon\subset F_1$, $n_x^N$ is well defined for every
$x\in K_\epsilon$, and
$$
K_\epsilon=\bigsqcup_{n=N}^{\infty}K_n.
$$

Applying the preceding covering estimate to $K'=K_n$,
we obtain a finite set $\mathcal A_n\subset K_n$ such that
$$
K_n\subset
\bigcup_{x\in\mathcal A_n}B(x,A,n,e^{-nr})
$$
and
\begin{align*}
\#\mathcal A_n \leq 
e^{n(1+8\lambda^2) \left(h+\frac{\lambda}{8}+3m_Yr\right)}
\leq e^{n(1+8\lambda^2) \left(h+\frac{\lambda}{2}\right)},
\end{align*}
where the last inequality follows from
$\displaystyle r<\frac{\lambda}{8m_Y}$.

For every $x\in\mathcal A_n\subset K_n$, 
by the definition of $n_x^N$, we have 
$$ 
\mu\bigl(B(x,A,n,e^{-nr})\bigr)
<
e^{-n\left(h+\frac{6\lambda}{8}\right)}.
$$
Therefore, we obtain 
\begin{align*}
\mu(K_n) 
\leq e^{n(1+8\lambda^2) \left(h+\frac{\lambda}{2}\right)} e^{-n\left(h+\frac{6\lambda}{8}\right)}
=e^{-\frac{\lambda}{4}n} e^{8\lambda^2n \left(h+\frac{\lambda}{2}\right)}
< e^{-\frac{\lambda}{8}n},
\end{align*}
where the last inequality follows from \eqref{SNE is SE-(2)}.
Summing over $n\geq N$, we obtain
$$
0<\mu(K_\epsilon)
=
\sum_{n=N}^{\infty}\mu(K_n)
\leq
\sum_{n=N}^{\infty}e^{-\frac{\lambda}{8}n}
=
\frac{e^{-\frac{\lambda N}{8}}}
{1-e^{-\frac{\lambda}{8}}}.
$$
Letting $N\to\infty$, the right-hand side tends to zero, which
contradicts $\mu(K_\epsilon)>0$.

This contradiction yields ${\mathcal E}_{\mu}^A\bigl(\{\rho_{\bf t}\},x\bigr) 
\leq h_{\mu}^A\bigl(\{\rho_{\bf t}\},x\bigr)$ for $\mu$-a.e.\ $x \in Y$. 
Combining this with definitions, we conclude that
${\mathcal E}_{\mu}^A\bigl(\{\rho_{\bf t}\},x\bigr)
=h_{\mu}^A\bigl(\{\rho_{\bf t}\},x\bigr)$ 
for $\mu$-almost every $x\in Y$. 
\end{proof}

\subsection{$r$-neutralized topological entropy}\label{subsect: r-neutralized topological entropy}

In this subsection, we study the topological counterpart of neutralized entropy. 
Using Yomdin theory \cite{Yom87,Yom87-2}, 
we obtain an upper bound for the $r$-neutralized topological entropy 
in terms of the usual topological entropy and derivative growth.

Let $T:X\to X$ be a continuous map on a compact metric space $(X,d)$. 
For $n\in \mathbb{N}$, $r>0$, and $x\in X$, 
the \textbf{$r$-neutralized Bowen ball} is defined as 
\begin{align*}
B\left(T,x, n, e^{-nr}\right) := \left\{y \in X: d\left(T^{i}(y), T^{i}(x)\right) < e^{-nr}, \,\forall\, 0 \leq i < n \right\}.
\end{align*}
When no confusion arises, we abbreviate $B(T,x,n,e^{-nr})$ as $B(x,n,e^{-nr})$. A set $E \subset X$ is said to be $(n, e^{-nr})$-spanning if 
$$
X \subset \bigcup_{x \in E} B(x, n, e^{-nr}).
$$
Let $S_{d}(T, n, e^{-nr})$ denote the minimal cardinality of such a spanning set. 
Then, we define the \textbf{$r$-neutralized topological entropy} $h_{d}^r(T)$ by
\begin{align*}
h_{d}^r(T) := \varlimsup_{n \rightarrow \infty} \frac{1}{n} \log S_{d}(T, n, e^{-nr}).
\end{align*}

Analogously, for a fixed scale $\epsilon>0$, one can define the Bowen ball 
$B(x,n,\epsilon)$ and the corresponding minimal cardinality of 
a spanning set $S_{d}(T, n, \epsilon)$ by replacing the term $e^{-nr}$ with $\epsilon$. 
The \textbf{topological entropy} $h_{\top}(T)$ is then defined by 
\begin{align*}
h_{\top}(T) := \lim_{\epsilon\to 0}\varlimsup_{n \rightarrow \infty} \frac{1}{n} \log S_{d}(T, n, \epsilon).
\end{align*}

\begin{definition}
Let $f \colon M \to M$ be a $C^1$ map. We define $R(f)$ by 
\[
R(f) = \lim_{n\to\infty} \frac{1}{n} \log \max_{x\in M} \| Df^n|T_xM \|,
\]
where the norm is induced by the metric $d$.
\end{definition}

We state our main result of this subsection as follows. 

\begin{theorem}\label{estimate theorem}
Let $f: M \to M$ be a $C^k$ diffeomorphism of a compact manifold $(M,d)$, 
and let $r > 0$. Then we have 
\begin{align*}
h_{d}^r(f) \leq h_{\top}(f) + \frac{\operatorname{dim}M}{k}R(f) + r\operatorname{dim}M. 
\end{align*} 
\end{theorem}  

As an immediate consequence, we obtain the following result. 

\begin{corollary}\label{estimate corollary}
Let $f: M \to M$ be a $C^{\infty}$ diffeomorphism of a compact manifold $(M,d)$. 
Then we have 
\begin{align*}
\lim_{r\to 0}h_{d}^r(f) =h_{\top}(f).
\end{align*} 
\end{corollary}

We begin by sketching our modified proof, 
which leverages the bound on the growth rates of smooth volumes under iteration \cite{Yom87,Yom87-2} 
to estimate the number of $r$-neutralized Bowen balls $B(\cdot,n,e^{-nr})$ 
needed to cover a Bowen ball $B(\cdot,n,\epsilon)$. 
Our argument is built upon \cite[Theorem 2.1 and Lemma 2.3]{Yom87}, 
which were originally established using semialgebraic sets.

Before proceeding to the proof of Theorem \ref{estimate theorem}, 
we first establish some auxiliary results. 

\begin{lemma}\label{r topological entropy lemma}
Let $f: M \to M$ be a $C^1$ diffeomorphism on a compact manifold $(M,d)$ and 
let $r>0$. For any $m\in \mathbb{N}$, we have 
$h^{mr}_{d}\left(f^{m}\right)= m h^r_{d}(f)$. 
\end{lemma}

\begin{proof}
Let $l_0 := \max\{\|Df|T_xM\| : x \in M\} < \infty$. 
For $m, n \in \mathbb{N}$, we have 
\[
\max _{0 \leq i < n} d\left(f^{im}(x), f^{im}(y)\right)
\leq \max _{0 \leq i \leq mn-1} d\left(f^{i}(x), f^{i}(y)\right).
\]
This directly implies that
\[
B(f,x, mn, e^{-mnr}) \subset B(f^m,x, n, e^{-mnr})
\]
for all $x \in M$, which consequently yields
$h^{mr}_{d}\left(f^{m}\right) \leq m h^r_{d}(f)$.

To prove the reverse inequality, recall the definition of $l_0$. For any $x, y \in M$, we have
\[
B\big(f^m,x, e^{-mnr-2ml_0}\big) \subset B\big(f,x, m, e^{-mnr}\big). 
\]
Set $k = \lfloor l_0 r^{-1} \rfloor + 1$. It then follows that
\begin{align*}
B(f^m,x, n+k, e^{-rm(n+k)})
&\subset B(f^m,x, n, e^{-rm(n+k)}) \subset B(f^m,x, n, e^{-mnr}e^{-2ml_0}) \\
&= \bigcap_{i=0}^{n-1} f^{-im} B\big(f^{im}(x), e^{-mnr}e^{-2ml_0}\big) \\
&\subset \bigcap_{i=0}^{n-1} f^{-im} B\big(f,f^{im}(x), m, e^{-mnr}\big) = B(f,x, mn, e^{-mnr}),
\end{align*}
which yields $m h^r_{d}(f) \leq h^{mr}_{d}\left(f^{m}\right)$.
This completes the proof.
\end{proof} 

In the remainder of this subsection, we assume that 
 $f : M \to M$ is a $C^{k}$ ($k\geq 1$) diffeomorphism of a compact manifold $(M,d)$, 
 and we let $m_{\vol}$ denote the volume measure on $M$. 
Here, $B(x, n, e^{-nr})$ refers to the $r$-neutralized Bowen ball 
defined by the dynamical system $(M, f, d)$. 
Similarly, $B(x,\epsilon)$ denotes the ball in $M$ with respect to the metric $d$.

We denote by $m_M$ the dimension of the manifold $M$. 
We equip the Euclidean space $\mathbb{R}^{m_M}$ with the usual Euclidean metric $d_{\text{Eucl}}$.
For any $r > 0$, let $\overline{B}(r,\mathbb{R}^{m_M}) \subset \mathbb{R}^{m_M}$ denote the closed ball of radius $r$ centered at the origin. 
We fix a finite $C^\infty$-atlas $\Omega = (\Omega_p, \omega_p)$ 
with open domains $\Omega_p$ covering $M$, 
where each $\omega_p: \Omega_p \to \mathbb{R}^{m_M}$ is a $C^\infty$ 
diffeomorphism onto its image. 
On each $\Omega_p$, we equip the local region with 
the metric $d_p : = \omega_p^* d_{\text{Eucl}}$. 
Without loss of generality, 
we assume that for any $k \geq 1$, all the $k$-th derivatives 
of the transition maps $\omega_{p_1} \circ \omega_{p_2}^{-1}$ 
are uniformly bounded by $C_0$ whenever defined. 
Furthermore, let $C_1 \geq 1$ be a constant such that on each $\Omega_p$, 
the metric $d$ on $M$ is bi-Lipschitz equivalent to the induced local metric 
$d_p$, meaning that
\begin{align}\label{def:C1}
\frac{1}{C_1} d_p \leq d \leq C_1 d_p.
\end{align}

Let $d^s$ denote the $s$-th order derivative, defined locally via the charts 
in $\Omega$ and their corresponding images in $\mathbb{R}^{m_M}$. 
For each $s = 1, \dots, k$, we define the quantity $N_s(f)$ 
as the maximum of $\|d^s f\|$ taken over all points in $M$ and all charts in $\Omega$:
\begin{equation*}
N_s(f) := \max_{(\Omega_p, \omega_p) \in \Omega} \sup_{x \in \Omega_p} \|d^s (f \circ \omega_p^{-1})(\omega_p(x))\|.
\end{equation*}
%For brevity, we denote the first-order quantity $N_1(f)$ simply by $N_1$. 

Now, we define $\epsilon_0 > 0$ (depending on $(M,d)$, 
the atlas $\Omega$, and $f$) to be the supremum of the real numbers $\epsilon$ 
satisfying the following conditions:
\begin{itemize}
    \item[(i)] For any $x \in M$, the ball of radius $10 N_1(f) C_1^2 \epsilon$ 
    with respect to the metric $d$, centered at $x$, is contained in some chart $\Omega_p$ of $\Omega$.
    \item[(ii)] For any $s = 2, \dots, k$, $(3 C_1 \epsilon)^{s-1} N_s(f) \leq N_1(f)$.
\end{itemize} 

Next, we present an adapted version of \cite[Theorem 2.1]{Yom87}. 
This theorem was originally established only for $l=1$ 
and $l=2$ in \cite{Yom87}, the proof for general $l$ was subsequently provided 
in \cite{Yom87-2} by considering a suitable semialgebraic set, 
where the parameter was refined from $\frac{2l}{k}$ to $\frac{l}{k}$. 
In our adapted version, the radii $1$ and $2$ are replaced 
by $C_1^{-2}\epsilon$ and $C_1^{2}\epsilon$, respectively. 
Since $\epsilon>0$ is fixed throughout this context, 
the original argument of Yomdin remains valid, 
though the resulting parameters will consequently depend on $\epsilon$. 

\begin{theorem}[\cite{Yom87,Yom87-2} Theorem 2.1]\label{Yom87-Theorem 2.1}
Let $\epsilon>0$ and $p,l\in\mathbb{N}$ with $l\leq p$. 
Let $g : \overline{B}(C_1^2\epsilon,\mathbb{R}^p) \to \mathbb{R}^{p}$ 
be a $C^k$-mapping ($1 \leq k < \infty$) such that
\[
\max_{x \in \overline{B}(C_1^2\epsilon, \mathbb{R}^p)} \| d^s g(x) \| \leq N, \quad s = 1, \dots, k.
\]
Let $\sigma : [0,1]^l \to \overline{B}(C_1^2\epsilon, \mathbb{R}^p)$ be a $C^k$-mapping 
such that $\displaystyle \max_{x \in [0,1]^l} \| d^s \sigma(x) \| \leq 1$ for $s = 1, \dots, k$.

Then there exist at most 
$$
\kappa = \kappa(\epsilon, k, p, l, N) = A(\epsilon,k, p, l) (\log N)^{B(\epsilon,k, p, l)} \cdot N^{\frac{l}{k}}
$$ 
$C^k$-mappings $\psi_j : [0,1]^l \to [0,1]^l$ satisfying the following properties:
\begin{itemize}
    \item[(1)] each $\psi_j$ maps $[0,1]^l$ diffeomorphically onto its image, for $j = 1, \dots, \kappa$;
    \item[(2)] $\displaystyle \bigcup_{j=1}^{\kappa} \operatorname{Im}(\psi_j) \supseteq S$, 
    where $S = (g \circ \sigma)^{-1}\left(\overline{B}(C_1^{-2}\epsilon,\mathbb{R}^{p})\right)$ 
    is the set of points $x \in [0,1]^l$ such that $g \circ \sigma(x) \in \overline{B}(C_1^{-2}\epsilon,\mathbb{R}^{p})$;
    \item[(3)] $\operatorname{Im}(g \circ \sigma \circ \psi_j) \subseteq \overline{B}(C_1^2\epsilon,\mathbb{R}^{p})$ for all $j = 1, \dots, \kappa$;
    \item[(4)] $\displaystyle \max_{x \in [0,1]^l} \| d^s (g \circ \sigma \circ \psi_j(x)) \| \leq 1$ for $s = 1, \dots, k$ and $j = 1, \dots, \kappa$.
\end{itemize}
Here, the constants $A= A(\epsilon,k, p, l)$ and $B=B(\epsilon,k, p, l)$ depend only on $\epsilon$, $k$, $p$, and $l$.
\end{theorem}

Fix $n\in\mathbb{N}$ with $e^{-nr}<\epsilon$ and $x_0\in M$. 
Let $x_i = f^i(x_0)$, $i=0,1,\dots,n$. For each $i$, 
we fix some chart $\Omega_{p_i}$, containing the ball $B(x_i,\epsilon)$ 
together with the ball of the radius $10 N_1(f) C_1^2 \epsilon$ 
in metric $d$ centered at $x_i$ (see property (i)).
Let $\tilde{B}_i$ denote the $C_1\epsilon$-ball in metric $d_{p_i}$, 
centered at $x_i$. By \eqref{def:C1}, $\tilde{B}_i$ contains the 
$\epsilon$-ball at $x_i$ in $M$ with respect to the metric $d$.

Hence the set $\tilde{\mathscr{B}} = \{ x \in M : f^i(x) \in \tilde{B}^i,\ i=0,\dots,n \}$ 
contains the Bowen ball $B(x_0,n,\epsilon)$. 

We note also that each $\tilde{B}_i$ is mapped by $f$ into $\Omega_{p_{i+1}}$. 
Indeed, by the definition of $N_1(f)$, $f(\tilde{B}_i)$ is contained in the 
$N_1(f) C_1^2 \epsilon$-ball in metric $d$ at $x_{i+1}$, 
which belongs entirely to $\Omega_{p_{i+1}}$. 

 Define $f_i : \overline{B}(C_1^2\epsilon,\mathbb{R}^{m_M})\subset \mathbb{R}^{m_M} \to \mathbb{R}^{m_M}$ by
\[
f_i = \left\langle \frac{\epsilon}{C_1 e^{-nr}} \right\rangle \circ \omega_{p_i} \circ f \circ \omega_{p_{i-1}}^{-1} \circ \left\langle \frac{C_1 e^{-nr}}{\epsilon} \right\rangle, \quad i=1,\dots,n,
\]
where $\langle \beta \rangle$ denotes the linear mapping 
$x \mapsto \beta x$ of $\mathbb{R}^{m_M}$.

By the assumptions above, $f_i$ are well-defined and satisfy the following conditions:
\begin{itemize}
    \item[(1)] $\displaystyle \max_{x \in \overline{B}(C_1^2\epsilon, \mathbb{R}^{m_M})} \| d^s f_i(x) \| \leq (C_1 e^{-nr}\epsilon^{-1})^{s-1} N_s(f) \leq N_1(f)$, $s=1,\dots,k$.
    \item[(2)] Let $S_i = \{ x \in \mathbb{R}^{m_M} : F_j(x) \in \overline{B}(C_1^{-2}\epsilon, \mathbb{R}^{m_M}),\ j=0,1,\dots,i \}$, where $F_j = f_j \circ \cdots \circ f_1$, $j=1,\dots,n$, $F_0 = \mathrm{Id}$. Then
    \[
    \tilde{\mathscr{B}} = \omega_{p_0}^{-1} \circ \left\langle \frac{C_1 e^{-nr}}{\epsilon} \right\rangle (S_n)\supset B(x_0,n,\epsilon).
    \]
\end{itemize}

Following \cite[Lemma 2.3]{Yom87}, the required parameterization result can be stated as follows. 

\begin{lemma}\label{Yom87-Lemma 2.3}
Let $\tilde{\sigma} : [0,1]^l \to \overline{B}(C_1^2\epsilon, \mathbb{R}^p)$ 
be a $C^k$-mapping, satisfying $\| d^s \tilde{\sigma} \| \leq 1$, $s=1,\dots,k$. 
Then for $i=0,\dots,n$, there exist not more than $\kappa^i$ diffeomorphisms 
$\psi_{i,j} : [0,1]^l \to [0,1]^l$, $j=1,\dots,\kappa^i$ such that
\begin{itemize}
    \item[(1)] $\| d^s (F_i \circ \tilde{\sigma} \circ \psi_{i,j}) \| \leq 1$, $s=1,\dots,k$,
    \item[(2)] $F_i \circ \tilde{\sigma} \circ \psi_{i,j}\left([0,1]^l\right) \subseteq \overline{B}(C_1^2\epsilon,\mathbb{R}^{m_M})$,
    \item[(3)] $\displaystyle \tilde{\sigma}^{-1}(S_i) \subseteq \bigcup_{j=1}^{\kappa^i} \psi_{i,j}\left([0,1]^l\right)$.
\end{itemize}
Here, the constant $\kappa$ is given by 
$\kappa=A(\epsilon,k, p, l) (\log N_1(f))^{B(\epsilon,k, p, l)} \cdot N_1(f)^{\frac{l}{k}}$. 
\end{lemma}

\begin{proof}
We proceed by induction on $i$. For the base case $i=0$, the mapping $\tilde{\sigma}$ itself trivially satisfies the conditions of the lemma. 
Assume that the lemma holds for some $t \geq 0$. 
For each $j = 1, \dots, \kappa^t$, we define the mapping $\sigma_j : [0,1]^l \to \mathbb{R}^{m_M}$ by
\[
\sigma_j = F_t \circ \tilde{\sigma} \circ \psi_{t,j}.
\]
Each $\sigma_j$, together with the mapping $f_{t+1} : \overline{B}(C_1^2\epsilon, \mathbb{R}^{m_M}) \to \mathbb{R}^{m_M}$, 
satisfies the assumptions of Theorem \ref{Yom87-Theorem 2.1}. 
Hence, by applying Theorem \ref{Yom87-Theorem 2.1} to each $\sigma_j$, 
we can find at most $\kappa$ mappings $\psi_q : [0,1]^l \to [0,1]^l$ 
such that 
\begin{align*}
\| d^s (f_{t+1} \circ \sigma_j \circ \psi_q) \| \leq 1,\quad \text{for}\quad s=1,\dots,k, 
\end{align*}
and $\operatorname{Im}(f_{t+1} \circ \sigma_j \circ \psi_q) \subseteq \overline{B}(C_1^2\epsilon,\mathbb{R}^{m_M})$. 
Furthermore, the union of the images of $\psi_q$ satisfies
\[
\bigcup_{q=1}^{\kappa} \operatorname{Im}(\psi_q) 
\supseteq (f_{t+1} \circ \sigma_j)^{-1}\left(\overline{B}(C_1^{-2}\epsilon, \mathbb{R}^{m_M})\right) 
\supseteq \left(F_{t+1}\circ \tilde{\sigma}\circ \psi_{t,j}\right)^{-1}(S_{t+1}).
\]
It follows that $f_{t+1} \circ \sigma_j \circ \psi_q = F_{t+1} \circ \tilde{\sigma} \circ \psi_{t,j} \circ \psi_q$. 
Thus, the mappings $\psi_{t+1,\alpha} : [0,1]^l \to [0,1]^l$, defined by 
$\psi_{t+1,\alpha} = \psi_{t,j} \circ \psi_q$, where $\alpha = (j,q)$, 
satisfy the required conditions of the lemma for $i = t+1$. 
Their total number does not exceed $\kappa^t \cdot \kappa = \kappa^{t+1}$, 
and their images cover $\tilde{\sigma}^{-1}(S_{t+1})$.
This completes the induction step and proves the lemma. 
\end{proof} 

With these preparations, we are now ready to prove Theorem \ref{estimate theorem}. 

\begin{proof}[Proof of Theorem \ref{estimate theorem}]
Fix $\epsilon < \epsilon_0$. 
Consider a finite family of $C^{m_M}$ diffeomorphisms $\{\hat{\sigma}\}$ from $[0,1]^{m_M}$ to $M$ 
such that the union of their images covers $M$. 
%We assume the number of these diffeomorphisms is $C_2$. 
By choosing a sufficiently fine covering or scaling $\epsilon$ if necessary, 
we can assume that for any point $x' \in M$, the ball $B(x', \epsilon)$ 
is entirely contained in the image of some diffeomorphism from this family. 

Now, fix an arbitrary point $x_0 \in M$ and an integer $n \in \mathbb{N}$ with $e^{-nr} < \epsilon$. 
According to the aforementioned property, we can select and fix 
a diffeomorphism $\sigma$ from 
$\{\hat{\sigma}\}$ such that its image contains $B(x_0,\epsilon)$. 

Note that we can subdivide $[0,1]^{m_M}$ into 
$c(\sigma) \left(\frac{C_1 e^{nr}}{\epsilon}\right)^{m_M}$ subcubes 
and reparametrize them linearly by
\[
h_\beta : [0,1]^{m_M} \to [0,1]^{m_M} 
\]
in such a way that the mappings
\[
\left\langle \frac{\epsilon}{C_1e^{-nr}} \right\rangle \circ \omega_{p_0} \circ \sigma \circ h_\beta : [0,1]^{m_M} \to \mathbb{R}^{m_M}
\]
have all derivatives up to order $k$ bounded by $1$, 
and those of their images, which intersect 
$\overline{B}(C_1^{-2}\epsilon,\mathbb{R}^{m_M})$, 
are contained in $\overline{B}(C_1^{2}\epsilon,\mathbb{R}^{m_M})$, 
where $c(\sigma)$ is a constant depending only on the diffeomorphism $\sigma$. 

By Lemma \ref{Yom87-Lemma 2.3}, 
on each of the $c(\sigma) \left(\frac{C_1 e^{nr}}{\epsilon}\right)^{m_M}$ subcubes, 
there are at most $\kappa^n$ $C^k$-mappings $\psi_{n,j}:[0,1]^{m_M}\to [0,1]^{m_M}$, where 
$$
\kappa = A(\epsilon,k, m_{M}, m_{M}) (\log N)^{B(\epsilon,k, m_{M}, m_{M})} N_1(f)^{\frac{m_M}{k}}. 
$$ 
For brevity, we fix one subcube (corresponding to a fixed $\beta$) 
and denote the mapping $\left\langle \frac{\epsilon}{C_1e^{-nr}} \right\rangle \circ \omega_{p_0} \circ \sigma \circ h_\beta$ 
simply by $\tilde{\sigma}$. 
Then, we have $\| d^s \tilde{\sigma} \| \leq 1$ for all $s=1,\dots,k$. 
According to Lemma \ref{Yom87-Lemma 2.3}, 
on each of the $c(\sigma) \left(\frac{C_1 e^{nr}}{\epsilon}\right)^{m_M}$ 
subcubes, there exist at most $\kappa^n$ such $C^k$-mappings $\psi_{n,j}$. 
Consequently, ranging over all subcubes, the total number of these mappings 
is bounded by $\tilde{\kappa} = \kappa^n \cdot c(\sigma) \left(\frac{C_1 e^{nr}}{\epsilon}\right)^{m_M}$, 
where $j = 1, \dots, \tilde{\kappa}$. 
By applying Lemma \ref{Yom87-Lemma 2.3} to each $\tilde{\sigma}$, 
we deduce that for each $j$, 
\begin{align*}
F_n \circ \tilde{\sigma} \circ \psi_{n,j}\left([0,1]^{m_M} \right)
&=\left\langle \frac{\epsilon}{C_1 e^{-nr}} \right\rangle \circ \omega_{p_n} \circ f^n \circ \omega_{p_{0}}^{-1} \circ \left\langle \frac{C_1 e^{-nr}}{\epsilon} \right\rangle \circ \tilde{\sigma} \circ \psi_{n,j}\left([0,1]^{m_M} \right)\\
&=\left\langle \frac{\epsilon}{C_1 e^{-nr}} \right\rangle \circ \omega_{p_n} \circ f^n \circ \sigma \circ h_\beta \circ \psi_{n,j}\left([0,1]^{m_M} \right)\\
&\subset \overline{B}(C_1^2\epsilon,\mathbb{R}^{m_M}).  
\end{align*}
This implies that for some $z_j \in M$, 
\begin{align*}
f^n\left( \sigma \circ h_\beta \circ \psi_{n,j}\left([0,1]^{m_M} \right)\right)
\subset \omega_{p_n}^{-1} \left( \overline{B}(C_1^2e^{-nr},\mathbb{R}^{m_M}) \right)
\subset B(f^n z_j,C_1^3e^{-nr}). 
\end{align*}
By the construction of $\psi_{n,j}$, every $\psi_{p+1,\alpha}=\psi_{p,j}\circ \psi_q$; 
thus, $\operatorname{Im}(\psi_{p+1,\alpha})\subset \operatorname{Im} (\psi_{p,j})$. 
This yields that for any $0 \leq k < n$, 
\begin{align*}
f^k\left( \sigma \circ h_\beta \circ \psi_{n,j}\left([0,1]^{m_M} \right)\right)
=f^k\left( \sigma \circ h_\beta \circ \psi_{k,j'}\left([0,1]^{m_M} \right)\right)
\subset B(f^k z_j,C_1^3e^{-nr}). 
\end{align*}
Thus, by choosing $n_* \in \mathbb{N}$ satisfying $e^{n_*r} \geq C_1^3$, 
we immediately conclude that
\begin{align*}
\sigma \circ h_\beta \circ \psi_{n,j}\left([0,1]^{m_M} \right)
&\subset \bigcap_{k=0}^n f^{-k}\left( B(f^kz_j,C_1^3e^{-nr}) \right)\\
&=B(z_j,n,C_1^3e^{-nr})\subset B\left(z_j,n,e^{-(n-n_*)r}\right)
\subset B\left(z_j,n-n_*,e^{-(n-n_*)r}\right). 
\end{align*}

In the opposite direction, recalling condition (2) that 
$\omega_{p_0}^{-1} \circ \langle \frac{C_1 e^{-nr}}{\epsilon} \rangle (S_n) \supset B(x_0,n,\epsilon)$, 
we obtain 
\begin{align*}
\bigcup_{j=1}^{\tilde{\kappa}} \omega_{p_0}^{-1}\circ \left\langle \frac{C_1 e^{-nr}}{\epsilon} \right\rangle \circ \tilde{\sigma}\circ \psi_{i,j}\left([0,1]^{m_M} \right)
\supset \omega_{p_0}^{-1}\circ \left\langle \frac{C_1 e^{-nr}}{\epsilon} \right\rangle S_n
\supset B(x_0,n,\epsilon). 
\end{align*}
It follows that 
\begin{align*}
    \bigcup_{j=1}^{\tilde{\kappa}} B\left(z_j,n-n_*,e^{-(n-n_*)r}\right) \supset 
\bigcup_{j=1}^{\tilde{\kappa}} {\sigma}\circ h_\beta \circ \psi_{n,j}\left([0,1]^{m_M} \right)
\supset B(x_0,n,\epsilon).
\end{align*}
By the arbitrariness of $x_0 \in M$ and the fact that the images of the 
family $\{\hat{\sigma}\}$ cover $M$, we conclude that 
every Bowen ball $B(\cdot,n,\epsilon)$ can be covered by at most 
\[
\kappa^n \left(\frac{C_1 e^{nr}}{\epsilon}\right)^{m_M} \cdot \max_{\hat{\sigma}}c(\hat{\sigma})
\]
balls of the form $B(\cdot, n-n_*,e^{-(n-n_*)r})$. Thus, we obtain
\begin{align*}
&S_d(n-n_*,e^{-(n-n_*)r})\\
\leq & \kappa^n \left(\frac{C_1 e^{nr}}{\epsilon}\right)^{m_M} S_d(n,\epsilon) \cdot \max_{\hat{\sigma}}c(\hat{\sigma})\\
\leq & A(\epsilon,k, m_M, m_M)^n (\log N_1(f))^{nB(\epsilon,k, m_M, m_M)} 
N_1(f)^{\frac{nm_M}{k}} \left(\frac{C_1 e^{nr}}{\epsilon}\right)^{m_M} S_d(n,\epsilon) \cdot \max_{\hat{\sigma}}c(\hat{\sigma}) . 
\end{align*}
Due to the finiteness of $\{\hat{\sigma}\}$, 
we have $\max_{\hat{\sigma}}c(\hat{\sigma}) < \infty$. 
This implies that   
\begin{align*}
h_{d}^r(f)
=& \varlimsup_{n\to\infty} \frac{1}{n-n_*} \log S_d(f,n-n_*,e^{-(n-n_*)r})\\
\leq & h_{\top}(f)+ \log A(\epsilon,k, m_M, m_M)+ B(\epsilon,k, m_M, m_M)\log \log N_1(f)  \\
&\quad\quad\quad  +\frac{m_M }{k}\log N_1(f)+rm_M.
\end{align*}

Next, for any fixed $q \in \mathbb{N}$, we apply the above inequality to 
$f^q$, $e^{-qnr}$, and $N_1(f^q)$ instead of $f$, $e^{-nr}$, and $N_1(f)$.  
Combining this with Lemma \ref{r topological entropy lemma}, we obtain
\begin{align*}
h_{d}^r(f)=\frac{h_{d}^{qr}(f^q)}{q}
%&\leq \frac{h_{\top}(f^q)}{q}+\frac{A(\epsilon,k, m_M, m_M)}{q}
%+B(\epsilon,k, m_M, m_M)\frac{\log \log N_1(f^q)}{q}+\frac{m_M}{qk}\log N_1(f^q)+\frac{qrm_M}{q}\\
=h_{\top}(f)&+\frac{\log A(\epsilon,k, m_M, m_M)}{q}\\
&+B(\epsilon,k, m_M, m_M)\frac{\log \log N_1(f^q)}{q}+\frac{m_M}{qk}\log N_1(f^q)+rm_M.
\end{align*}
Since $q$ can be chosen arbitrarily large, by taking $q \to \infty$, we immediately deduce 
\begin{align*}
h_{d}^r(f) \leq h_{\top}(f)+\frac{m_M}{k}R(f)+rm_M. 
\end{align*}
Thus, we complete the proof of the theorem. 
\end{proof}

At the end of this subsection, 
we present corresponding results for linear maps on $\mathbb{R}^m$ 
and endomorphisms on Lie groups, based on the arguments in \cite{Bow71-1}. 
The proofs heavily rely on the fact that the Haar measure of an 
$r$-neutralized Bowen ball is independent of its center. 
Following \cite{Bow71-1}, topological and $r$-neutralized entropies 
on noncompact spaces are defined via the supremum over compact subsets. 
Since the arguments are almost identical to those in \cite{Bow71-1}, 
we omit the detailed proofs. 

\begin{lemma}\label{Bow71-1-Proposition 12}
    Let $T: M \to M$ be a differentiable map of an $m$-dimensional Riemannian manifold $(M,d)$. 
    Then
    \[
    h^r_{d}(T) \leq \max\left\{0,\, m \log \sup_{x\in M} \|dT|_{T_x M}\|\right\}+mr.
    \]
\end{lemma}

\begin{proposition}\label{Bow71-1-Theorem 15}
    Let $T\colon \mathbb{R}^m\to\mathbb{R}^m$ be a linear map 
    and let $d$ be the standard Euclidean metric on $\mathbb{R}^m$. Then 
    \[
    h^r_{d}(T) = \sum_{|\lambda_i|>1} \log |\lambda_i|+mr=h_{\top}(T)+mr, 
    \]
    where $\{\lambda_i\}$ are the eigenvalues of $T$. 
\end{proposition}

\begin{proposition}\label{Bow71-1-Corollary 16}
If \( T \) is an endomorphism of a $m$-dimensional Lie group \( G \) 
and \( d \) is a right-invariant metric, then
\[
 h^r_{d}(T) =\sum_{|\lambda_i|>1} \log |\lambda_i|+mr= h_{\top}(T)+mr,
\]
where $\{\lambda_i \}$ denote the eigenvalues of the 
tangent map \( dT|_{T_eG} \) on the tangent space $T_eG$ at the identity $e\in G$.
\end{proposition}

\section{Cone property and entropy formulas}\label{sect: Entropy Formulas for the Skew Product}

In this section, using the results of \S \ref{sect:neutralized-entropy}, 
we first investigate the cone property and then complete the proofs 
of Theorem \ref{higher le:product} and Theorem \ref{higher le:fiber-entropy}. 

\subsection{Cone property}\label{subsect: Cone property}

This subsection is devoted to the cone property. 
It identifies the entropy along sequences with a prescribed asymptotic direction 
with the topological entropy of the corresponding action element. 
Our main purpose is to derive this property from the weak variational principle
for $r$-neutralized entropy.

Given ${\bf a}\in\mathbb{R}^k$ and $\delta>0$, 
 we first define the cone with axis ${\bf a}$ as 
 \begin{align*}
 \mathcal{C}_{\bf a}(\delta):=\left\{{\bf v}\in\mathbb{R}^k :\angle({\bf v},{\bf a})<\delta\right\},
 \end{align*}
 where $\angle({\bf v},{\bf a})$ denotes the angle between vectors ${\bf v}$ and ${\bf a}$. 
 Let $B_{\mathbb{R}^k}(n)$ denote the closed ball in $\mathbb{R}^k$ of radius $n$ centered at the origin. 
We define the $n$-cone metric as  
\begin{align*}
d^{{\bf a},\delta}_n(x,y):=\sup\left\{ d(\rho_{\bf t}x,\rho_{\bf t}y): {\bf t}\in \mathcal{C}_{{\bf a}}(\delta)\cap B_{\mathbb{R}^k}(n|{\bf a}|) \right\},\quad \forall\,x,y\in Y. 
\end{align*}
Given ${\bf a}\in\mathbb{R}^k$, $n\in\mathbb{N}$, $\epsilon>0$ and $x\in Y$, 
the \textbf{cone Bowen ball} $B(x,{\bf a},\delta,n,\epsilon)$ is given by
\begin{align*}
B(x,{\bf a},\delta,n,\epsilon):=\{y\in Y: d_n^{{\bf a},\delta}(x,y)<\epsilon\}.
\end{align*}
A nonempty set $E\subset Y$ is said to be $({\bf a},\delta ,n,\epsilon)$-spanning if
$$
 Y\subset \bigcup_{x\in E} B(x,{\bf a},\delta,n,\epsilon).
$$
Let $S_d(\{\rho_{\bf t}\},{\bf a},\delta,n,\epsilon)$ be the minimal cardinality of an 
$({\bf a},\delta,n,\epsilon)$-spanning set.

The \textbf{topological cone entropy} of $\{\rho_{\bf t}\}$ along the direction ${\bf a}$ is defined by
\begin{align*}
h_{\top}^{\bf a}(\{\rho_{\bf t}\}):=\lim_{\delta\to 0}\lim_{\epsilon\to 0}\varlimsup_{n\to\infty} \frac{1}{n}\log S_d(\{\rho_{\bf t}\},{\bf a},\delta,n,\epsilon).
\end{align*}

\begin{definition}\label{def:cone property}
A locally free $C^{1+\beta}$ $\mathbb{R}^k$-action $\{\rho_{\bf t}\}_{t\in\mathbb{R}^k}$ on a closed Riemannian manifold $Y$ 
is said to have the {\bf cone property} if, 
for  any vector ${\bf a}\in\mathbb{R}^k\setminus \{{\bf 0}\}$ and 
any sequence $A:=\{{\bf a}(n)\}_{n=1}^\infty$ satisfying 
\begin{align*}
K_0:=\sup_{i\geq 1} \|{\bf a}(i+1)-{\bf a}(i)\|<\infty,\quad \text{and}\quad 
\frac{{\bf a}(n)}{n}\to {\bf a}\,\text{ as }\, n\to\infty, 
\end{align*}
we have $ h_\mathrm{top}^{\bf a}(\{\rho_{\bf t}\})
=h_\mathrm{top}^A(\{\rho_{\bf t}\})=h_\mathrm{top}(\rho_{\bf a})$. 
\end{definition}

It follows directly from the definitions of sequence entropy and cone entropy that 
\begin{align}\label{lower cone entropy}
h_\mathrm{top}^{\bf a}(\{\rho_{\bf t}\})\geq h_\mathrm{top}^A(\{\rho_{\bf t}\}),\quad h_\mathrm{top}^{\bf a}(\{\rho_{\bf t}\})\geq h_\mathrm{top}(\rho_{\bf a}).
\end{align}
The opposite inequalities are much more delicate to verify. 
In the subsequent discussion, we present several examples 
satisfying the cone property. Moreover, the cone property is closely related 
to a weak form of the variational principle for $r$-neutralized entropy, 
which was introduced by the authors in \cite{DQ25}.

Next, we introduce the notion of $r$-neutralized entropy, 
which was proposed in \cite{DQ25}. 

 For each $\mu\in \mathcal{M}(f,X)$, 
the {\bf $r$-neutralized local entropy} $h_{\mu,d}^{r}(f,x)$ and 
  the {\bf lower $r$-neutralized local entropy} $\underline{h}_{\mu,d}^{r}(f,x)$ 
  are defined respectively by
  \begin{align*}
    h_{\mu,d}^{r}(f,x):=\varlimsup_{n\to\infty}-\frac{1}{n}\log \mu (B(x,n,e^{-nr})), \quad \underline{h}_{\mu,d}^{r}(f,x):=\varliminf_{n\to\infty}-\frac{1}{n}\log \mu (B(x,n,e^{-nr})). 
  \end{align*}
They are $f$-invariant, and are constant $\mu$-a.e. if $\mu$ is ergodic. 
We then define the {\bf $r$-neutralized entropy} $h_{\mu,d}^{r}(f)$ and
the {\bf lower $r$-neutralized entropy} $\underline{h}_{\mu,d}^{r}(f)$ as 
\[
  h_{\mu,d}^{r}(f):=\int h_{\mu,d}^{r}(f,x)\, \rd \mu, \quad \underline{h}_{\mu,d}^{r}(f):=\int \underline{h}_{\mu,d}^{r}(f,x) \,\rd \mu.
\]

In \cite{DQ25}, we introduced a general variational principle for $r$-neutralized entropy. 
In this work, we focus on its smooth version and formulate it as the following conjecture.

\begin{conjecture}[Variational Principle \cite{DQ25}]\label{conjecture}
Let $f : X \to X$ be a $C^{1+\beta}$ diffeomorphism of a closed manifold $(X,d)$. 
Then, for any $r > 0$,
  \begin{align*}
    h_{d}^r(f) = \sup\{h_{\nu, d}^r(f) : \mu \in \mathcal{M}(f, X)\}, \quad 
    \underline{h}_{d}^r(f) = \sup\{\underline{h}_{\nu, d}^r(f) : \mu \in \mathcal{M}(f, X)\}.
  \end{align*}
\end{conjecture}

Before presenting the variational principle for 
the $r$-neutralized entropy related to the cone property, 
we first introduce the Brin-Katok local entropy and 
entropy formulas for the $r$-neutralized entropy.

We recall the following result by Brin and Katok \cite{BK83}. 

\begin{theorem}[\cite{BK83}] \label{thm:Brin-Katok}
Let $(X, d)$ be a compact metric space and $f : X \to X$ a continuous map 
preserving a Borel probability measure $\mu$ with finite metric entropy $h_\mu(f)$. 
Then for $\mu$-a.e. $x \in X$, the following limit exists:
\begin{equation*}
h_\mu(f, x) := \lim_{\epsilon \to 0} \liminf_{n \to \infty} -\frac{\log \mu(B(x, n, \epsilon))}{n} = \lim_{\epsilon \to 0} \limsup_{n \to \infty} -\frac{\log \mu(B(x, n, \epsilon))}{n}.
\end{equation*}
We call $h_\mu(f, x)$ the {\bf Brin-Katok local entropy} of $f$ at $x$.

Furthermore, the function $h_\mu(f, \cdot)$ is $f$-invariant and satisfies
$\displaystyle \int_X h_\mu(f, x) \, \mathrm{d}\mu(x) = h_\mu(f)$. 
\end{theorem}

The original concept of neutralized entropy was first introduced by 
Ben Ovadia and Rodriguez Hertz in \cite{ORH23}, 
and their work inspires our generalization to the $r$-neutralized entropy 
proposed in \cite{DQ25}. One central result from \cite{ORH23} is stated below. 

\begin{theorem}[\cite{ORH23} Theorem 3.4]\label{ORH23-Theorem 3.4}
Let $ f: M \rightarrow M $ be a $ C^{1+\beta} $ diffeomorphism on a closed Riemannian manifold $ M $
and $ \mu $ an $f$-invariant Borel probability measure. Then 
  \begin{align*}
   \lim_{r\to 0} h^r_{\mu,d}(f,x) =\lim_{r\to 0} \underline{h}^r_{\mu,d}(f,x)= h_\mu(f,x). 
  \end{align*}
\end{theorem}

Later, the authors established entropy formulas for $r$-neutralized entropy 
that hold without taking the limit as $r\to 0$.

\begin{theorem}[\cite{DQ25} Theorem 1.1]\label{DQ25-theorem 1.1 1.2}
Let $ f: X \to X$ be a $ C^{1+\beta} $ diffeomorphism on a closed Riemannian manifold $ X $ 
and let $ \mu $ be an $f$-invariant Borel probability measure. 
Let $d^u_\mu(x)$ and $d^s_\mu(x)$ denote the unstable 
and stable pointwise dimensions of $\mu$, respectively, 
and let $d^c(x)$ be the multiplicity of zero Lyapunov exponents at $x$. 

Then, for any given $r>0$ and $\mu$-a.e. $x\in X$, we have 
  \begin{align*}
    h^r_{\mu,d}(f,x) \leq h_\mu(f,x)+r(d^u_\mu(x)+d^s_\mu(x)+d^c(x)).
  \end{align*}
 \end{theorem}

In \cite{DQ25}, we verified that linear Anosov diffeomorphisms 
and topological Markov shifts satisfy Conjecture \ref{conjecture}. 
Furthermore, Proposition \ref{Bow71-1-Corollary 16} implies that 
any endomorphism of a Lie group satisfies Conjecture \ref{conjecture} 
with respect to a right-invariant metric. 
Recently, Chen, Xie and Zhou constructed a counterexample 
to this conjecture (see \cite[Example 2.4]{CXZ26}).

By applying Theorem \ref{ORH23-Theorem 3.4} (or Theorem \ref{DQ25-theorem 1.1 1.2}) to Conjecture \ref{conjecture}, 
we formulate a {\bf Weak Variational Principle} for $r$-neutralized entropy.
This weaker version of the principle was originally suggested 
by Professor Ben Ovadia. 

\begin{conjecture}[Weak Variational Principle]\label{weak conjecture}
Let $f : M \to M$ be a $C^{1+\beta}$ diffeomorphism of a closed manifold $(M,d)$. 
Then 
\begin{align*}
\lim_{r\to 0}h^r_{d}(f) =
\lim_{r\to 0} \sup\{h^r_{\mu,d}(f) : \mu \in \mathcal{M}(f, M)\} = 
\sup\{h_{\mu,d}(f) : \mu \in \mathcal{M}(f, M)\} = h_{\top}(f).
\end{align*}
\end{conjecture}

We now briefly outline some important classes of dynamical systems that satisfy 
Conjecture \ref{weak conjecture}, such as Axiom A diffeomorphisms on basic sets 
(see Example \ref{weak conjecture-hyperbolic}) and 
$C^{\infty}$ diffeomorphisms on compact manifolds 
(see Example \ref{weak conjecture-infty}). 

\begin{example}\label{weak conjecture-hyperbolic}
If $X$ is a hyperbolic set of $f$, then $f$ satisfies 
Conjecture \ref{weak conjecture} with respect to any metric $d$ on $X$.

To see this, fix a metric $d$ on $X$. 
By \cite[Proposition 20.2.6]{KH95}, there exist $\epsilon > 0$ and $0 < \gamma < 1$ 
such that, for any $x, y \in X$, the condition $d(f^k x, f^k y) < \epsilon$ 
for all $k \in [-N, N]$ implies $d(x, y) < \gamma^N$.
For a fixed $r > 0$, letting $c = -2r / \log \gamma$, 
we obtain the following ball inclusions:
\begin{align*}
B(x, -cn, cn+n, \epsilon) \subset B(x, n, e^{-nr}) \subset B(x, n, \epsilon), 
\end{align*}
where 
$$ 
B(x, -cn, cn+n, \epsilon):=\bigcap_{k=-cn}^{cn+n} f^{-k}\bigl(B(f^k x, \epsilon)\bigr).
$$
The assertion for hyperbolic sets then follows immediately from these inclusions.

Furthermore, this argument extends naturally to the global setting. 
When $f$ is an Axiom A diffeomorphism, a similar local estimate holds 
by \cite[Lemma 4.2]{Bow75}. Consequently, 
Axiom A systems also satisfy Conjecture \ref{weak conjecture} for any metric. 
\end{example}

\begin{example}\label{weak conjecture-infty}
Corollary \ref{estimate corollary} implies that 
Conjecture \ref{weak conjecture} holds for $C^{\infty}$ diffeomorphisms. 
In particular, endomorphisms of compact Lie groups satisfy Conjecture \ref{weak conjecture}. 
\end{example}

Beyond the aforementioned systems, motivated by the notion of $T$-homogeneity 
introduced by Bowen in \cite[Definition 6]{Bow71-1}, we introduce the definition of 
the neutralized homogeneous measure, which leads to another significant 
class of dynamical systems satisfying Conjecture \ref{weak conjecture}. 

\begin{definition}\label{neutralized homogeneous}
Let $T: X \to X$ be a uniformly continuous map on a metric space $(X, d)$. 
  A Borel measure \(\mu\) on \(X\) is said to be 
{\bf \(T\)-neutralized homogeneous} if it satisfies the following conditions: 
\begin{itemize}
\item[(i)] $\mu(K)<\infty$ for every compact subset $K\subset X$. 
\item[(ii)] $\mu(K)>0$ for some compact subset $K\subset X$. 
\item[(iii)] For each \(r > 0\), there exist \(r'=r'(r) > 0\) and \(c=c(r) > 0\) such that
    \[
    \mu(B(y, n, e^{-nr'})) \leq c \mu(B(x, n,e^{-nr}))
    \]
    for all \(n \geq 0\) and all \(x, y \in X\). 
 Moreover, $r'$ may be chosen so that $r'\to 0$ as $r\to 0$. 
\end{itemize}
If $X$ is compact, we additionally require $\mu$ to be a probability measure, 
in which case conditions (i) and (ii) are naturally satisfied. 
\end{definition}

This concept is particularly natural in the context of algebraic dynamical systems. 
For instance, when $X$ is a Lie group equipped with a right-invariant metric 
and $T$ is a group endomorphism, the right Haar measure on $X$ satisfies 
this neutralized homogeneity property. 
Furthermore, our notion of neutralized homogeneity should not be confused with 
the algebraic sense used in Ratner's theory, i.e., Haar measures on homogeneous spaces.

The following lemma utilizes the neutralized homogeneous measure 
to establish the relationship between $r$-neutralized entropy 
and $r$-neutralized topological entropy. 

\begin{lemma}\label{homogeneous measure lemma}
  Let $T$ be a uniformly continuous map on a locally compact metric space $(X, d)$. 
  If $T$ preserves a neutralized $T$-homogeneous measure $\mu$, 
  then for any $x \in X$,
  \begin{align*}\lim_{r \to 0} h^r_d(T) = \lim_{r \to 0} h^r_{\mu, d}(T, x).
  \end{align*}
\end{lemma}

  \begin{proof}
    First, let $K \subset X$ be an arbitrary compact subset. 
    By condition (i) and the local compactness of the space, 
    there exists a neighborhood $U$ of $K$ such that $\mu(U) < \infty$. 
    Choose $\varepsilon > 0$ sufficiently small 
    so that the $\varepsilon$-neighborhood 
    $$
    B_\varepsilon(K) := \{ y \in X : d(y, K) < \varepsilon \}
    $$ 
    is contained in $U$. 
    Fix $r > 0$. If $E \subset K$ is an $(n, e^{-nr})$-separated set, 
    it follows from the definitions that the collection of balls 
    $\displaystyle \left\{ B\left(x, n, \frac{1}{2}e^{-nr}\right) \right\}_{x \in E}$
     consists of pairwise disjoint sets contained in $U$. 
     For this fixed $r$, there exists a sufficiently large $n_r \in \mathbb{N}$ 
     such that for all $n > n_r$ and $x \in X$, the following inclusions hold: 
     \begin{align*}
      B\left(x, n, \frac{1}{2}e^{-nr}\right) \supset B(x, n, e^{-(n+n_r)r}) \supset B(x, n+n_r, e^{-(n+n_r)r}).
    \end{align*}
    By the $T$-neutralized homogeneity of $\mu$, 
    there exist $r' > 0$ and $c > 0$ such that 
    $$ 
     \mu\left(B\left(y, m, e^{-mr'}\right)\right) \leq c \mu\left(B\left(x, m, e^{-mr}\right)\right)
     $$ 
     for all $x, y \in X$ and $m \geq 1$. 
     Applying this to our setting with $m = n+n_r$, we obtain
     $$ 
      N_d(T, K, n, e^{-nr}) \cdot \frac{1}{c} \mu\left(B\left(y, n+n_r, e^{-(n+n_r)r'}\right)\right) \leq \mu(U),
      $$
      which implies
      \begin{align*}h^r_d(T, K) &= \varlimsup_{n \to \infty} \frac{1}{n} \log N_d(T, K, n, e^{-nr}) \\
        &\leq \varlimsup_{n \to \infty} -\frac{1}{n} \log \mu(B(y, n+n_r, e^{-(n+n_r)r'}))= h_{\mu, d}^{r'}(T, y).
      \end{align*}
      Since $K$ is arbitrary, 
      it follows that $h^r_d(T) \leq h_{\mu, d}^{r'}(T, y)$ holds for any $y \in X$. 
      Taking the limit as $r \to 0$ (which implies $r' \to 0$), we have
      \begin{align}\label{homogeneous measure lemma-(1)}
        \lim_{r \to 0} h^r_d(T) \leq \lim_{r' \to 0} h_{\mu, d}^{r'}(T, y).
      \end{align}
      
      Conversely, consider a compact set $K$ such that $\mu(K) > 0$, 
      which exists by condition (ii). 
      Given $r' > 0$, choose $r > 0$ and $c > 0$ according to the 
      neutralized homogeneity such that
      $$ 
      \mu\left(B\left(y, n, e^{-nr}\right)\right) \leq c \mu\left(B\left(x, n, e^{-nr'}\right)\right)
      $$
      for all $x, y \in X$ and $n \geq 0$. 
      If $F \subset X$ is an $(n, e^{-nr})$-spanning set for $K$, 
      then 
      $$  
      K \subset \bigcup_{x \in F} B(x, n, e^{-nr}).
      $$
Thus, we have 
      $$  
      \mu(K) \leq \sum_{x \in F} \mu\left(B\left(x, n, e^{-nr}\right)\right) \leq S_d(T, K, n, e^{-nr}) \cdot c \mu\left(B\left(y, n, e^{-nr'}\right)\right),
      $$
      where $y$ is any point in $X$. 
      This yields
      \begin{align*}
        h^{r'}_{\mu, d}(T, y) = \varlimsup_{n \to \infty} -\frac{1}{n} \log \mu\left(B\left(y, n, e^{-nr'}\right)\right) 
        \leq \varlimsup_{n \to \infty} \frac{1}{n} \log S_d(T, K, n, e^{-nr}) = h^r_d(T, K).
      \end{align*}
      Taking the limit as $r' \to 0$ and noting that $h^r_d(T, K) \leq h^r_d(T)$, 
      we obtain
      \begin{align}\label{homogeneous measure lemma-(2)}
        \lim_{r' \to 0} h^{r'}_{\mu, d}(T, y) \leq \lim_{r \to 0} h^r_d(T).
      \end{align}

This lemma immediately follows from \eqref{homogeneous measure lemma-(1)} and \eqref{homogeneous measure lemma-(2)}. 
    \end{proof} 

Thus, we have the following result. 

\begin{example}\label{weak conjecture-homogeneous}
Every $C^{1+\beta}$ diffeomorphism $f$ on a closed manifold $(X,d)$ 
that preserves an $f$-neutralized homogeneous measure 
satisfies Conjecture \ref{weak conjecture}. 

This example follows directly from Lemma \ref{homogeneous measure lemma} 
and Theorem \ref{ORH23-Theorem 3.4}.
\end{example}

\begin{remark}\label{remark-2-lie group-homogeneous measure}
Combining Lemma \ref{homogeneous measure lemma}, 
Theorem \ref{ORH23-Theorem 3.4}, and Conjecture \ref{weak conjecture}, 
we conclude that if a $C^{1+\beta}$ diffeomorphism $f$ on a closed manifold preserves an $f$-neutralized homogeneous measure $\mu$, 
then $\mu$ is the measure of maximal entropy of $f$. 
\end{remark}

\begin{lemma}\label{topological sequence entropy bounds sequence entropy}
Let $\mu$ be a Borel probability measure which is invariant and
ergodic with respect to the $\mathbb R^k$-action. 
Let $A=\{\mathbf{a}(n)\}_{n=1}^\infty$ be a sequence in $\mathbb R^k$ satisfying
\[
K_0=\sup_{i\geq 1}\bigl\|\mathbf{a}(i+1)-\mathbf{a}(i)\bigr\|<\infty,
\]
and suppose there exists $\mathbf{a}\in\mathbb{R}^k$ satisfying $\dfrac{\mathbf{a}(n)}{n}\to\mathbf{a}$ as $n\to\infty$. 
Then $h_{\top}^{A}(\{\rho_{\mathbf{t}}\}) \geq h_{\mu}^{A}(\{\rho_{\mathbf{t}}\})$.
\end{lemma}

\begin{proof}
Since $h_\mu^A(\{\rho_{\bf t}\},x)$ is invariant under the $\mathbb R^k$-action 
and $\mu$ is ergodic with respect to this action, 
it is constant for $\mu$-almost every $x$.
Denote its constant value by $h$. 
Fix $\delta>0$. There exist $\epsilon>0$ and a measurable set $K\subset Y$ with $\mu(K)>0$ such that
\[
\varlimsup_{n\to\infty} -\frac1n\log\mu(B(x,A,n,2\epsilon)) > h-\delta \quad \text{for every } x\in K.
\]
For $n\in\mathbb N$, set
\[
K_n := \left\{ x\in K: \mu(B(x,A,n,2\epsilon)) < e^{-n(h-2\delta)} \right\}.
\]
Clearly, every $x\in K$ belongs to infinitely many $K_n$, which implies $\displaystyle K\subset\limsup_{n\to\infty}K_n$.

We claim that $\mu(K_n)>e^{-n\delta}$ for infinitely many $n$. 
Otherwise, $\mu(K_n)\leq e^{-n\delta}$ for all large $n$, 
which implies $\displaystyle \sum_{n=1}^{\infty}\mu(K_n)<\infty$. 
By the Borel--Cantelli lemma, we conclude that 
$\displaystyle \mu\left(\limsup_{n\to\infty}K_n\right)=0$, 
contradicting $\mu(K)>0$.

For each such $n$, let $E_n$ be a minimal $(A,n,\epsilon)$-spanning set. If $B(y,A,n,\epsilon)\cap K_n\neq\varnothing$, choose $x_y$ in this intersection. Then $B(y,A,n,\epsilon) \subset B(x_y,A,n,2\epsilon)$, so
\[
\mu(B(y,A,n,\epsilon)) < e^{-n(h-2\delta)}.
\]
Since these balls $\{B(y,A,n,\epsilon)\}_{y\in E_n}$ cover $K_n$, we obtain
\[
e^{-n\delta} < \mu(K_n) \leq S_d(\{\rho_{\mathbf{t}}\},A,n,\epsilon) e^{-n(h-2\delta)}.
\]
Therefore, $S_d(\{\rho_{\mathbf{t}}\},A,n,\epsilon) > e^{n(h-3\delta)}$ for infinitely many $n$. 
It follows that 
$$
h_{\top}^{A}(\{\rho_{\mathbf{t}}\}) \geq h-3\delta. 
$$
Taking $\delta\to0$ yields 
$h_{\top}^{A}(\{\rho_{\mathbf{t}}\}) \geq h = h_{\mu}^{A}(\{\rho_{\mathbf{t}}\})$, 
completing the proof. 
\end{proof}

Next, we prove the main result of this subsection.

\begin{proposition}\label{homogeneous proposition}
Let $\{\rho_{\bf t}\}_{{\bf t}\in\mathbb R^k}$ be a locally free
$C^{1+\beta}$ $\mathbb R^k$-action on a closed Riemannian manifold
$Y$. Suppose that the action preserves a Borel probability measure
$m_{\rm MME}$ which is ergodic with respect to the
$\mathbb R^k$-action and is a measure of maximal entropy for
$\rho_{\bf t}$ for every ${\bf t}\in\mathbb R^k$. 
If each $\rho_{\bf t}$ satisfies Conjecture \ref{weak conjecture} for all ${\bf t}\in\mathbb{R}^k$, 
then $\{\rho_{\bf t}\}$ has the cone property. 
\end{proposition}

\begin{proof}
Fix a sequence $A=\{{\bf a}(n)\}_{n=1}^\infty$ satisfying
$\displaystyle K_0=\sup_{i\geq 1} \|{\bf a}(i+1)-{\bf a}(i)\|<\infty$, 
and assume that there exists ${\bf a}\in\mathbb{R}^k$
such that $\displaystyle \frac{{\bf a}(n)}{n}\to {\bf a}$ as $n\to\infty$.

Recall that 
\[
\Upsilon:=\sup\bigl\{\|d\rho_{\bf t}|_{T_yY}\| : y\in Y,\ {\bf t}\in\mathbb{R}^k,\ \|{\bf t}\|\leq 1\bigr\}+1<\infty.
\]
Given $\epsilon,\delta>0$, let $B(\rho_{\bf a},x,n,\epsilon)$ denote the Bowen ball for the diffeomorphism $\rho_{\bf a}$.
For ${\bf t}\in \mathcal{C}_{{\bf a}}(\delta)\cap B(n|{\bf a}|)$, 
we write ${\bf t}={\bf t}_1+{\bf t}_2$, where ${\bf t}_1=\tilde{t}{\bf a}$ 
for some $0\leq \tilde{t}\leq n|{\bf a}|$ and $|{\bf t}_2|<n|{\bf a}|\delta$.
Then, for sufficiently large $n$, we have 
\[
B(\rho_{\bf a},x,n,\Upsilon^{-n|{\bf a}|\delta})\subset B(x,{\bf a},\delta,n,\epsilon).  
\] 
Combining this with the assumption that $\rho_{\mathbf{a}}$ 
satisfies Conjecture \ref{weak conjecture}, we deduce 
\begin{align}\label{lower cone entropy-2}
h_\mathrm{top}^{{\bf a}}(\{\rho_{\bf t}\})\leq \lim_{r\to 0}h^r_{\top}(\rho_{\bf a})=h_{\top}(\rho_{\bf a}).
\end{align}

Let $\mu$ be an ergodic measure for this $\mathbb{R}^k$-action. 
Given $\epsilon,\delta>0$, there exists $N\in\mathbb{N}$ such that 
$|{\bf a}(n)-{\bf a}n |<n\delta$ holds for all $n\geq N$. 
For any sufficiently large $n\geq N$ such that 
$\Upsilon^{-n\delta}\leq \epsilon \Upsilon^{-N |{\bf a}|\delta}$, it follows that 
\begin{align}
B(x,A,n,\Upsilon^{-2n\delta}) &=B(x,\Upsilon^{-2n\delta})\cap B(x,A,n,\Upsilon^{-2n\delta})\nonumber\\
&\subset  B(\rho_{{\bf a}},\epsilon \Upsilon^{-N |{\bf a}|\delta})\cap
\rho_{{\bf a}}^{-N}\left(B(\rho_{\bf a},\rho_{{\bf a}}^{N}x,n-N,\epsilon)\right)\nonumber\\
&\subset B(\rho_{{\bf a}},x,N,\epsilon)\cap
\rho_{{\bf a}}^{-N}\left(B(\rho_{{\bf a}},\rho_{{\bf a}}^{N}x,n-N,\epsilon)\right)
= B(\rho_{{\bf a}},x,n,\epsilon).\label{inclusion}
\end{align}
By Theorem \ref{SNE is SE}, taking the limits 
 $n\to\infty$, $\delta\to 0$, and $\epsilon\to 0$ successively, 
we conclude that
\begin{align}\label{lower cone entropy-3}
{h}_{\mu}^{A}(\{\rho_{\bf t}\})=
{\mathcal{E}}_{\mu}^{A}(\{\rho_{\bf t}\})\geq h_\mu(\rho_{{\bf a}}).
\end{align}
Taking the supremum over all ergodic measures $\mu\in \mathcal{M}_e(\{\rho_{\bf t}\},X)$,
it follows from \eqref{lower cone entropy-3} and Lemma \ref{topological sequence entropy bounds sequence entropy} that 
\begin{align}
h_{\top}^{A}(\{\rho_{\bf t}\})&\geq \sup\{{h}_{\mu}^{A}(\{\rho_{\bf t}\}):\mu\in \mathcal{M}_e(\{\rho_{{\bf t}}\},Y) \}\nonumber\\
&\geq \sup\{h_{\mu}(\rho_{{\bf a}}):\mu\in \mathcal{M}_e(\{\rho_{{\bf t}}\},Y) \}
=h_{m_{\MME}}(\rho_{{\bf a}})=h_{\top}(\rho_{{\bf a}}).\label{lower cone entropy-4}
\end{align}
%where the last equality holds since $\mu$ is a measure of maximal entropy for $\rho_{{\bf a}}$.
Therefore, the proposition follows immediately from \eqref{lower cone entropy}, \eqref{lower cone entropy-2} 
and \eqref{lower cone entropy-4}.
\end{proof}

\subsection{Entropy formulas for skew products}
\label{subsect: Entropy formulas for skew products}

In this subsection, we complete the proofs of the entropy formulas for skew products. 

We first prove the entropy formula for the invariant product measures (Theorem \ref{higher le:product}). 

\begin{proof}[Proof of Theorem \ref{higher le:product}]
Since $\mu$ is ergodic, by the Birkhoff Ergodic Theorem,
for $\mu$-a.e. $x\in X$, we have 
\[
\lim_{n\to\infty}
\frac1n\sum_{i=0}^{n-1}\tau(f^ix)
=\int_X\tau\,\rd\mu=\mu(\tau).
\]
We denote by $X_0$ the set of points $x$ satisfying the above equality.

Fix $\epsilon,\delta>0$ and $x\in X_0$.
Let $\displaystyle A=\left\{\sum_{i=0}^{j-1}\tau(f^ix)\right\}_{j=1}^{\infty}$
 and ${\bf a}=\mu(\tau)$.
By an argument similar to the proof of \eqref{inclusion} in
Proposition \ref{homogeneous proposition}, 
for any $y\in Y$ and sufficiently large $n$, we have 
\[
B(y,A,n,\Upsilon^{-2n\delta})
\subset
B(\rho_{\mu(\tau)},y,n,\epsilon).
\]

We first explain how to pass from the fiber sequence Bowen ball to
the Bowen ball of the skew product. 
Since $\tau$ is $C^{1+\beta}$ on the compact manifold $X$, 
there exists $L_\tau>0$ such that
\[
\|\tau(x_1)-\tau(x_2)\|
\leq L_\tau d(x_1,x_2)
\quad\text{for all }x_1,x_2\in X.
\]
Moreover, since the action is $C^1$ and $Y$ is compact, 
there exists $C_\rho>0$ such that
\[
d(\rho_{\bf u}z,z)\leq C_\rho\|{\bf u}\|
\]
for every $z\in Y$ and every sufficiently small ${\bf u}\in\mathbb R^k$.

Suppose that
\[
(x',y') \in B\bigl(F,(x,y),n,\Upsilon^{-3n\delta}\bigr).
\]
Then, for every $0\leq j<n$, we have $d(f^jx,f^jx')<\Upsilon^{-3n\delta}$.
Therefore, 
\begin{align}\label{higher le:product-(5)}
\|S_j\tau(x')-S_j\tau(x)\| 
\leq \sum_{i=0}^{j-1} \|\tau(f^ix')-\tau(f^ix)\|  
\leq L_\tau j\Upsilon^{-3n\delta} 
\leq L_\tau n\Upsilon^{-3n\delta}.
\end{align}
Using the commutativity of the $\mathbb R^k$-action, we have
$\rho_{S_j\tau(x')}
= \rho_{S_j\tau(x')-S_j\tau(x)} \circ\rho_{S_j\tau(x)}$. Thus,
\begin{align*}
d\bigl( \rho_{S_j\tau(x)}y, \rho_{S_j\tau(x)}y' \bigr)
&\leq d\bigl( \rho_{S_j\tau(x)}y, \rho_{S_j\tau(x')}y' \bigr) 
+ d\bigl(\rho_{S_j\tau(x')}y', \rho_{S_j\tau(x)}y' \bigr) \\
&\leq \Upsilon^{-3n\delta} + C_\rho \|S_j\tau(x')-S_j\tau(x)\| \\
&\leq (1+C_\rho L_\tau n)\Upsilon^{-3n\delta} <\Upsilon^{-2n\delta}
\end{align*}
for all sufficiently large $n$. 
Hence, we obtain $y'\in B(y,A,n,\Upsilon^{-2n\delta})$.
Together with the preceding inclusion and the fact that
$\Upsilon^{-3n\delta}<\epsilon$ for sufficiently large $n$, this gives
\[
B\bigl(F,(x,y),n,\Upsilon^{-3n\delta}\bigr) 
\subset B(f,x,n,\epsilon) \times B(\rho_{\mu(\tau)},y,n,\epsilon).
\] 
Applying Theorem \ref{thm:Brin-Katok} to $f$ and $\rho_{\mu(\tau)}$, 
and Theorem \ref{ORH23-Theorem 3.4} to $F$,
then letting $n\to\infty$, $\epsilon\to0$, and finally
$\delta\to0$, we obtain for $\mu\times\nu$-almost every $(x,y)$, 
\[
h_{\mu\times\nu}(F,(x,y)) \geq h_\mu(f,x) + h_\nu(\rho_{\mu(\tau)},y).
\]
Since $\mu$ is ergodic,
$h_\mu(f,x)=h_\mu(f)$ for $\mu$-a.e. $x$. Integrating over
$X\times Y$, we conclude that
\begin{align}\label{higher le:product-(1)}
h_{\mu\times\nu}(F) \geq h_\mu(f) + h_\nu(\rho_{\mu(\tau)}).
\end{align}

For the reverse direction, 
a similar argument yields that for all sufficiently large $n$, 
\begin{align}\label{higher le:product-(3)}
B(\rho_{\mu(\tau)},y,n,\Upsilon^{-2n\delta}) \subset B(y,A,n,\epsilon/2).
\end{align}
We claim that for all sufficiently large $n$,
\begin{align}\label{higher le:product-(4)}
B(f,x,n,\Upsilon^{-3n\delta})  \times B(\rho_{\mu(\tau)},y,n,\Upsilon^{-3n\delta})
\subset B(F,(x,y),n,\epsilon).
\end{align}
Indeed, let
$x'\in B(f,x,n,\Upsilon^{-3n\delta})$ and  
$y'\in B(\rho_{\mu(\tau)},y,n,\Upsilon^{-3n\delta})$. 
Since
\[
B(\rho_{\mu(\tau)},y,n,\Upsilon^{-3n\delta}) 
\subset B(\rho_{\mu(\tau)},y,n,\Upsilon^{-2n\delta}),
\]
it follows from \eqref{higher le:product-(3)} that
\[
d\bigl( \rho_{S_j\tau(x)}y, \rho_{S_j\tau(x)}y' \bigr)
< \frac{\epsilon}{2}
\quad \text{for every }0\leq j<n.
\]
By \eqref{higher le:product-(5)}, for all sufficiently large $n$, we have 
\begin{align*}
d\bigl( \rho_{S_j\tau(x)}y', \rho_{S_j\tau(x')}y' \bigr)
\leq C_\rho \|S_j\tau(x')-S_j\tau(x)\|  
\leq C_\rho L_\tau n\Upsilon^{-3n\delta} < \frac{\epsilon}{2}.
\end{align*}
Therefore, we conclude that 
\begin{align*}
d\bigl( \rho_{S_j\tau(x)}y, \rho_{S_j\tau(x')}y' \bigr)
\leq d\bigl( \rho_{S_j\tau(x)}y, \rho_{S_j\tau(x)}y' \bigr) 
+ d\bigl( \rho_{S_j\tau(x)}y', \rho_{S_j\tau(x')}y' \bigr) 
< \frac{\epsilon}{2}+\frac{\epsilon}{2} =\epsilon.
\end{align*}
Since $d(f^jx,f^jx')<\Upsilon^{-3n\delta}<\epsilon$, 
we obtain \eqref{higher le:product-(4)}.

Using \eqref{higher le:product-(4)}, Theorem \ref{thm:Brin-Katok}, and
Theorem \ref{ORH23-Theorem 3.4}, and then letting
$n\to\infty$, $\epsilon\to0$, and $\delta\to0$, we obtain 
\begin{align*}%\label{higher le:product-(2)}
h_{\mu\times\nu}(F) \leq h_\mu(f) + h_\nu(\rho_{\mu(\tau)}).
\end{align*}
Combining this with \eqref{higher le:product-(1)}, we complete the proof. 
\end{proof}

We now turn from product measures to measures of maximal relative entropy, 
starting with the following relative variational principle due to Ledrappier and Walters \cite{LW77}.

\begin{theorem}[\cite{LW77} (1.2) and Theorem 2.1]\label{LW77-Theorem 2.1}
Let $ X, Y $ be compact metric spaces, and let $ T : X \to X $, $ S : Y \to Y $, $ \pi : X \to Y $ be continuous maps such that $ \pi $ is surjective and $ \pi \circ T = S \circ \pi $.
Given $ \mu \in \mathcal{M}(Y, S) $, we have 
\begin{align*}
\sup \left\{ h_\zeta(T):\zeta \in \mathcal{M}(T, X) \text{ and } \zeta \circ \pi^{-1} = \mu \right\}= h_\mu(S) + \int_Y h_{\top}(T, \pi^{-1}(y))  \,\rd\mu(y).
\end{align*}
%Moreover, for $ f \in C(X) $, 
%\begin{align*}
%\sup \left\{ h_\mu(T|S) + \mu(f) : \mu \in \mathcal{M}(T, X) \text{ and } \mu \circ \pi^{-1} = \nu \right\} = \int_Y P_{\top}(T, f, \pi^{-1}(y)) \, \rd\nu(y).
%\end{align*}
\end{theorem}

For any $x \in X$ and $n \in \mathbb{N}$, 
we define the Birkhoff sum $S_n\tau$ by 
$\displaystyle S_n\tau(x):=\sum_{i=0}^{n-1}\tau(f^ix)$. 

With these preparations and the cone property, 
we now prove Theorem \ref{higher le:fiber-entropy}. 

\begin{proof}[Proof of Theorem \ref{higher le:fiber-entropy}]
By the Birkhoff ergodic theorem, for $\mu$-a.e. $x\in X$, 
setting $A=\left\{S_j\tau(x)\right\}_{j=1}^{\infty}$ and ${\bf a}=\mu(\tau)$,
we have  $\displaystyle \frac{S_n\tau(x)}{n}\longrightarrow {\bf a}$.

We first consider the case ${\bf a}\neq{\bf 0}$. By Proposition
\ref{homogeneous proposition}, for $\mu$-a.e. $x\in X$, we have 
\begin{equation*}
 h_{\top}(F,\pi^{-1}(x))
 =
 h_{\top}^A(\{\rho_{\bf t}\})
 =
 h_{\top}(\rho_{\mu(\tau)}).
\end{equation*}

We now turn to the case ${\bf a}={\bf 0}$. 
We claim that $h_{\top}^A(\{\rho_{\bf t}\})=0$.

Recall that
\[
\Upsilon =\sup\left\{ \left\|D\rho_{\bf t}|_{T_yY}\right\|:
y\in Y,\ {\bf t}\in\mathbb R^k,\ \|{\bf t}\|\leq 1 \right\}+1.
\]
Fix $\delta>0$. Since $\displaystyle \frac{S_j\tau(x)}{j}\to {\bf 0}$,
there exists $N_\delta\in\mathbb N$ such that
$\|S_j\tau(x)\|<j\delta$ for every $j\geq N_\delta$. 
Then, for every sufficiently large $n$ and every $0\leq j<n$,
\[
\|S_j\tau(x)\|\leq n\delta+C_\delta,
\]
where $\displaystyle C_\delta := \max_{0\leq j<N_\delta}\|S_j\tau(x)\|$. 
Fix $\epsilon>0$. If $d(y,z) <\epsilon \Upsilon^{-n\delta-C_\delta-1}$,
then, for every $0\leq j<n$,
\begin{align*}
d\left( \rho_{S_j\tau(x)}y, \rho_{S_j\tau(x)}z \right)
\leq \Upsilon^{\|S_j\tau(x)\|+1}d(y,z)
< \Upsilon^{n\delta+C_\delta+1} \epsilon \Upsilon^{-n\delta-C_\delta-1} 
= \epsilon.
\end{align*}
This yields 
\begin{align}\label{higher le:fiber-entropy-(1)}
B\left( y,\epsilon \Upsilon^{-n\delta-C_\delta-1} \right) \subset B(y,A,n,\epsilon).
\end{align}

Let $m_Y=\dim Y<\infty$. Since $Y$ is a closed Riemannian manifold, 
there exists a constant $C_0>1$ such that, for every sufficiently small 
$r>0$ and every $y\in Y$,
$$
C_0^{-1}r^{m_Y} \leq \operatorname{Vol}(B(y,r)) \leq C_0r^{m_Y}.
$$
Set $r_n:=\epsilon \Upsilon^{-n\delta-C_\delta-1}$. 
By the Besicovitch covering lemma, 
for all sufficiently large $n\in \mathbb{N}$, 
there exists a finite collection of balls $\{B(y_i,r_n)\}_{i=1}^{N_n}$
covering $Y$ with covering multiplicity bounded by a constant $C_Y$
depending only on the Riemannian manifold $Y$, and in particular
independent of $n$. Therefore, we obtain 
\begin{align*}
N_n C_0^{-1}r_n^{m_Y} 
\leq \sum_{i=1}^{N_n}\operatorname{Vol}(B(y_i,r_n)) 
\leq C_Y\operatorname{Vol}(Y).
\end{align*}
Hence, we have 
$$
N_n \leq C_0C_Y\operatorname{Vol}(Y)\,r_n^{-m_Y}
=C_0C_Y\operatorname{Vol}(Y)\,
\epsilon^{-m_Y}\Upsilon^{m_Y(n\delta+C_\delta+1)}.
$$

On the other hand, by \eqref{higher le:fiber-entropy-(1)},
\begin{align}\label{higher le:fiber-entropy-(2)}
B(y_i,r_n) \subset B(y_i,A,n,\epsilon)
\end{align}
for every $1\leq i\leq N_n$. Since the balls
$\{B(y_i,r_n)\}_{i=1}^{N_n}$ cover $Y$, the collection
$\{B(y_i,A,n,\epsilon)\}_{i=1}^{N_n}$ also covers $Y$. Consequently,
\begin{align*}
S_d(\{\rho_{\bf t}\},A,n,\epsilon) \leq N_n \leq
C_0C_Y\operatorname{Vol}(Y)\, \epsilon^{-m_Y}\Upsilon^{m_Y(n\delta+C_\delta+1)}.
\end{align*}
It follows that
$$
\varlimsup_{n\to\infty} \frac{1}{n} \log S_d(\{\rho_{\bf t}\},A,n,\epsilon)
\leq m_Y\delta\log \Upsilon.
$$
Since $\delta>0$ is arbitrary and
$S_d(\{\rho_{\bf t}\},A,n,\epsilon)\geq1$, we conclude that
$$
\varlimsup_{n\to\infty} \frac{1}{n} \log S_d(\{\rho_{\bf t}\},A,n,\epsilon) =0.
$$
Letting $\epsilon\to0$, we obtain
$h_{\top}^A(\{\rho_{\bf t}\})=0$. 
Since $\rho_{\bf 0}$ denotes the identity map, we also have
$h_{\top}(\rho_{\bf 0})=0$. 
Therefore, in the case ${\bf a}={\bf 0}$, for $\mu$-a.e. $x\in X$, we have 
$$
h_{\top}(F,\pi^{-1}(x)) = h_{\top}^A(\{\rho_{\bf t}\}) = 0
=h_{\top}(\rho_{\mu(\tau)}).
$$

Combining both cases, for $\mu$-a.e. $x\in X$, we have
\begin{equation*}
 h_{\top}(F,\pi^{-1}(x))
 =
 h_{\top}^A(\{\rho_{\bf t}\})
 =
 h_{\top}(\rho_{\mu(\tau)}).
\end{equation*}
Note that the continuous surjective map $\pi:X\times Y\to X$ 
satisfies $\pi\circ F=f\circ\pi$.
According to Theorem \ref{LW77-Theorem 2.1}, we have
\begin{align*}
\sup\left\{ h_\zeta(F) : \zeta\in \mathcal{M}(F,X\times Y) \text{ and } \pi_*\zeta=\mu\right\}
&= h_\mu(f) + \int_X h_{\top}(F, \pi^{-1}(x))  \,\rd\mu(x)\\
&=h_\mu(f) + \int_X h_{\top}(\rho_{\mu(\tau)})  \,\rd \mu(x)\\
&= h_\mu(f) + h_{\top}(\rho_{\mu(\tau)}) . 
\end{align*}
This proves the first assertion.

The second assertion follows directly by taking the supremum over all
ergodic measures $\mu\in\mathcal M_e(f,X)$ and applying the variational
principle for the skew product system $F$.
\end{proof}

\begin{remark}
When $\mu(\tau)\neq {\bf 0}$, the above proposition can be extended to the case 
where $Y$ is a compact metric space. 
When $\mu(\tau)={\bf 0}$, the same conclusion holds provided that 
$Y$ is a compact metric space with finite upper box dimension. 
Indeed, combining \eqref{higher le:fiber-entropy-(2)} with 
the definition of the upper box dimension, 
one can apply the preceding argument directly. 
\end{remark}

It is worth noting that, for the skew product considered above, 
the $\mathbb{R}^k$-action on the fiber satisfies the conditions 
in Examples \ref{weak conjecture-hyperbolic}, \ref{weak conjecture-infty}, 
and \ref{weak conjecture-homogeneous}. 
Thus, we derive the following corollary. 

\begin{corollary}\label{higher le:fiber-entropy-corollary}
Let $f: X \to X$ be a continuous map on a compact metric space $X$. 
Suppose $\{\rho_{\bf t}\}_{{\bf t} \in \mathbb{R}^k}$ is a locally free $C^{1+\beta}$ $\mathbb{R}^k$-action 
on a closed Riemannian manifold $Y$, and let $\tau: X \to \mathbb{R}^k$ be a continuous cocycle. 
Assume that for each ${\bf t} \in \mathbb{R}^k$, the map $\rho_{\bf t}$ satisfies at least one of the following conditions:
\begin{itemize}
\item $\rho_{\bf t}$ is of class $C^{\infty}$;
\item $Y$ is a hyperbolic set for $\rho_{\bf t}$; 
\item $\rho_{\bf t}$ preserves a $\rho_{\bf t}$-neutralized homogeneous measure.
\end{itemize}
Further assume that the $\mathbb{R}^k$-action preserves 
a Borel probability measure $m_{\mathrm{MME}}$ 
which is ergodic with respect to the $\mathbb{R}^k$-action 
and is a measure of maximal entropy for $\rho_{\mathbf{t}}$ for every $\mathbf{t}\in\mathbb{R}^k$. 
Then, we have
\begin{equation*}
\sup \left\{ h_\zeta(F) : \zeta \in \mathcal{M}(F, X \times Y) \text{ and } \pi_* \zeta = \mu \right\} = h_\mu(f) + h_{\operatorname{top}}(\rho_{\mu(\tau)}).
\end{equation*}
Consequently, the topological entropy of $F$ satisfies 
\begin{equation*}
h_{\operatorname{top}}(F) = \sup \left\{ h_\mu(f) + h_{\operatorname{top}}(\rho_{\mu(\tau)}) : \mu \in \mathcal{M}_e(f, X)\right\}.
\end{equation*}
\end{corollary}

\begin{remark}
Unlike Theorem \ref{higher le:product}, 
Theorem \ref{higher le:fiber-entropy} 
and Corollary \ref{higher le:fiber-entropy-corollary} 
require the cone property. 
This is because the connection between $r$-neutralized entropy and 
Brin-Katok local entropy is already established (Theorem \ref{ORH23-Theorem 3.4}), 
whereas the link between $r$-neutralized topological entropy and 
topological entropy remains incomplete. 
The latter relies on Conjecture \ref{weak conjecture}, 
which, in turn, necessitates the cone property. 
\end{remark}

\section{Thermodynamic classification of base projections}\label{sect: Invariant foliations and MME projections}

Using the nonlinear thermodynamic formalism, 
we establish Theorem \ref{goal theorem}, 
he main result of this section: 
the projections of ergodic MMEs form a nonempty finite set of 
equilibrium states for H\"older potentials. 

\subsection{Nonlinear thermodynamic formalism}\label{subsect: Nonlinear thermodynamic formalism}

In this subsection, we recall some background on nonlinear thermodynamic formalism. 
We refer the reader to \cite{BKL23, Ku24} for details.

Let $T: X \to X$ be a continuous map on a compact metric space $X$, and let
$\Phi: X \to \mathbb{R}^k$ be a vector-valued continuous function defined by
$\Phi(x) = \bigl(\phi_1(x), \dots, \phi_k(x)\bigr)$. 
Given a continuous function $G:\mathbb{R}^{d}\to \mathbb{R}$, we define an energy
$\mathcal{E}:\mathcal{M}(T,X)\to\mathbb{R}$ associated with the potentials $\{\phi_i\}_{i=1}^d$ by
\[
\mathcal{E}(\mu) := G\bigl(\mu(\phi_1), \mu(\phi_2), \dots, \mu(\phi_k)\bigr).
\]
The nonlinear pressure with respect to the invariant measure $\mu$ is defined by
\[
\Pi^{G}(T,\mu) := h_\mu(T) + \mathcal{E}(\mu).
\]

\begin{definition}[\cite{BKL23}]
Given $n\in\mathbb{N}$ and $\epsilon$, 
we define 
\[ 
 \zeta(\varepsilon, n) = \sup\{ \omega_n(\mathcal{C}): \mathcal{C} \text{ is a $(n, \varepsilon)$-separated set}\},
\]
where 
$$
\omega_n(\mathcal{C}) = \sum_{x \in \mathcal{C}} \exp \left(n\mathcal{E}(\Delta_x^n)\right), 
\quad \text{ and }\quad \Delta_x^n=\frac{1}{n}\sum_{i=0}^{n-1} \delta_{T^ix}. 
$$
The {\bf nonlinear topological pressure} 
%and {\bf lower nonlinear topological pressure} 
weighted by the energy $\mathcal{E}$ is defined as 
\begin{align*}
\Pi_{\mathrm{top}}^{G}(T) = \lim_{\varepsilon \to 0} \varlimsup_{n \to \infty} \frac{1}{n} \log \zeta(\varepsilon, n). 
%\underline{\Pi}_{\mathrm{top}}^{G}(T) &= \lim_{\varepsilon \to 0} \varliminf_{n \to \infty} \frac{1}{n} \log \zeta(\varepsilon, n).
\end{align*}
\end{definition}

 We will say that the system $(T, \mathcal{E})$ has an {\bf abundance of ergodic measures} if for 
 any $T$-invariant measure $\mu$ and $\varepsilon > 0$, 
 there is an ergodic $T$-invariant measure $\nu$ such that
\[
h_\nu(T) + \mathcal{E}(\nu) > h_\mu(T) + \mathcal{E}(\mu) - \varepsilon.
\]
Buzzi, Kloeckner and Leplaideur \cite{BKL23} and Kucherenko \cite{Ku24} 
introduced two related yet distinct conditions, 
both referred to as the abundance of ergodic measures.
In the present paper, we adopt the original definition 
given by Buzzi, Kloeckner and Leplaideur.
We note that uniformly hyperbolic diffeomorphisms satisfy both conditions;
see \cite{BKL23,Ku24} for further discussion.

\begin{theorem}[\cite{BKL23} Theorem A, \cite{Ku24} Theorem 2, Remark 2]\label{BKL23-Theorem A}
Let $T: X \to X$ be a continuous map on a compact metric space $X$ and 
$J: \mathbb{R}^k \to \mathbb{R}$ be a continuous nonlinearity function. 
Suppose $\Phi: X \to \mathbb{R}^k$ is a continuous potential such that the system $(T, \Phi)$ has an abundance of ergodic measures. Then
\begin{align*}
\Pi_{\mathrm{top}}^J(T) %=\underline{\Pi}_{\mathrm{top}}^J(T)
&= \sup\bigl\{h_\mu(T) + J\bigl(\mu(\Phi)\bigr):\mu \in \mathcal{M}(T,X)\bigr\}\\
&= \sup\bigl\{h_\mu(T) + J\bigl(\mu(\Phi)\bigr):\mu \in \mathcal{M}_e(T,X)\bigr\}.
\end{align*}
\end{theorem}

An invariant measure $\mu$ is called a {\bf nonlinear equilibrium state} for $G$ and $\Phi$ if it satisfies
\[
\Pi_{\mathrm{top}}^J(\Phi) = h_\mu(T) + J\bigl(\mu(\Phi)\bigr).
\]
We denote by $\NES(J,\Phi)$ the set consisting of nonlinear equilibrium states. 

Let $D(\omega,r) $ denote the Euclidean ball with center $\omega$ and radius $r$. 
We denote the corresponding set of points $w$ by 
$\operatorname{Rot}_{\mathrm{pt}}(\Phi)$, 
called the {\bf pointwise rotation set} of $\Phi$; that is
\[
\operatorname{Rot}_{\pt}(\Phi) = \left\{ {\bf w} \in \mathbb{R}^k : 
\forall r > 0\ \forall N\ \exists\, n \geq N\ \exists\, x \in X \, \text{ s.t. }\, \frac{1}{n} S_n \Phi(x) \in D({\bf w}, r) \right\}.
\]
Next, we discuss a measure-theoretic approach to rotation sets and 
define the {\bf generalized rotation set} of $\Phi$ by
\[
\operatorname{Rot}(\Phi) = \{\operatorname{rv}(\mu) : \mu \in \mathcal{M}(T,X)\}, 
\]
where the rotation vector $\operatorname{rv}(\mu)$ of the measure $\mu$ is defined by 
$$
\operatorname{rv}(\mu) = \left( \int \phi_1\, \rd\mu, \dots, \int \phi_k\, \rd\mu \right).
$$ 
Since $\mathcal{M}(f,X)$ is convex and weak* compact, and $\rv(\cdot)$ is affine and continuous, 
the generalized rotation set $\rot(\Phi)$ is compact and convex. 
We call $\mathcal{M}_\Phi({\bf w}) = \{ \mu \in \mathcal{M} : \operatorname{rv}(\mu) = {\bf w} \}$ 
the {\bf rotation class} of ${\bf w}\in\mathbb{R}^k$. 

\begin{remark}\label{KW14-remark}
  By \cite[Proposition 3 and the remark on page 814]{KW14}, 
  for topological mixing subshifts of finite type, 
    Axiom A basic sets and expansive systems with specification, 
    we can conclude that for any continuous function $\Phi:X\to \mathbb{R}^m$, 
    $\rot(\Phi)=\rot_{\pt}(\Phi)$. 
\end{remark}

In \cite{BKL23}, Buzzi, Kloeckner and Leplaideur introduced the following notions. 

\begin{definition}[\cite{BKL23} Definition 3.2]
    Let $T: X \to X$ be a continuous map on a compact metric space $X$, and let
$\Phi: X \to \mathbb{R}^k$ be the continuous function 
defined by
$\Phi(x) = \bigl(\phi_1(x), \dots, \phi_k(x)\bigr)$, where each $\phi_i: X \to \mathbb{R}$ is continuous. 
We define the {\bf entropy function} $\mathsf{h} : \mathbb{R}^k \to \mathbb{R} \cup \{-\infty\}$ by 
\[
\mathsf{h}({\bf w}) = \sup\bigl\{h_\mu(T):\mu\in \mathcal{M}_\Phi({\bf w})\bigr\}.
\]
The usual convention $\sup(\varnothing)=-\infty$ is understood. 

The {\bf entropy-potential diagram} is the hypograph of the entropy function:
\[
\mathcal{D} = \bigl\{ (z_0; z_1, \dots, z_d) \in [0,+\infty) \times \mathbb{R}^k :
\exists\, \mu \in \mathcal{M}(T,X),\ h(T,X) \geq z_0,\, \forall\, i,\, \mu(\phi_i) = z_i \bigr\}.
\]
If $h_{\mathrm{top}}(T) < \infty$, we have 
$\{ {\bf z} \in \mathbb{R}^k : \mathsf{h}({\bf z}) \ne -\infty \} = \rot(\Phi)$.
Since the Kolmogorov-Sinai entropy is affine, 
$\mathcal{D}$ is a convex set, i.e., $\mathsf{h}$ is concave.
\end{definition}

We define the rotation set $\mathcal{V}$ of 
the set of nonlinear equilibrium states by 
\begin{align*}
    \mathcal{V}:=\left\{ \mu(\Phi): \mu\in\NES(J,\Phi)\right\}. 
\end{align*}

The following result is contained implicitly in the proof of \cite[Claim 3]{BKL23}, 
without the smoothness of $J$.  
We provide the proof for the convenience of the reader.

\begin{lemma}\label{not boundary}
Let $f\colon X\to X$ be a topologically mixing Anosov diffeomorphism on a closed manifold $X$ and 
let $J: \mathbb{R}^k \to \mathbb{R}$ be a Lipschitz continuous function. 
Then for any H\"older continuous potential $\Phi: X \to \mathbb{R}^k$ satisfying the independence condition, 
we have $\mathcal{V}\subset \inter \rot(\Phi)$. 
\end{lemma}

\begin{proof} 
By \cite[Corollary 3.12]{BKL23}, the system 
$(X,f,\Phi)$ is a $C^\omega$ Legendre system
(see \cite[Definition 3.10]{BKL23}). 
In particular, the entropy function $\mathsf h$ is concave on
$\rot(\Phi)$ and satisfies
$$
|\nabla\mathsf h(z)|\to\infty
\quad\text{as $z$ approaches the boundary of $\rot(\Phi)$}.
$$

We first note that every point in $\mathcal{V}$ is a maximizer of
$\mathsf{h}+J$ on $\rot(\Phi)$. Indeed, let
$z=\mu(\Phi)\in\mathcal{V}$, where $\mu$ is a nonlinear equilibrium state.
By the definition of $\mathsf{h}$, we have
$$
\mathsf{h}(z)\geq h_\mu(f).
$$
If the inequality is strict, there exists 
$\nu\in\mathcal{M}_\Phi(z)$ such that $h_\nu(f)>h_\mu(f)$,
and hence
$$
h_\nu(f)+J(z)>h_\mu(f)+J(z),
$$
contradicting the fact that $\mu$ is a nonlinear equilibrium state.
Therefore, we conclude that 
$$
\mathsf{h}(z)=h_\mu(f),
$$
and consequently $z$ is a maximizer of $\mathsf{h}+J$ on $\rot(\Phi)$.
Thus, it suffices to show that no point on the boundary of
$\rot(\Phi)$ can be such a maximizer.

Fix a point $ z_0 $ on the boundary of $\rot(\Phi)$ 
and choose $\mathbf u\in\mathbb R^k$ such that
$z_0+\mathbf u\in\operatorname{int}\operatorname{Rot}(\Phi)$.
By convexity of $\operatorname{Rot}(\Phi)$, we have 
$$
z_0+t\mathbf u\in\operatorname{int}\operatorname{Rot}(\Phi)
\quad \text{for all $t\in(0,1]$.}
$$
Set $g(t): = \mathsf{h}(z_0 + t\mathbf u) $. 
Since $\mathsf{h}$ is concave, the derivative of $ g $ has a limit 
(finite or infinite) as $ t \to 0^{+} $. 
For all sufficiently small $ t > 0 $, we have 
$$ 
g'(t) = \nabla \mathsf{h}(z_0 + t \mathbf u) \cdot \mathbf u .
$$

For each sufficiently small $t>0$, let $H_t$ be the tangent
hyperplane to the upper boundary of $\mathcal D$ at the point
$\bigl(\mathsf h(z_0+t\mathbf u),z_0+t\mathbf u\bigr)$. 
Its normal vector is $(1,-\nabla\mathsf h(z_0+t\mathbf u))$. 
As $ t \to 0 $, $ |\nabla \mathsf{h}| \to \infty $ so any 
accumulation point $ H_0 $ of $ H_t $ is vertical, 
of the form $ \mathbb{R} \times L $ where $ L $ is a hyperplane of 
$ \mathbb{R}^k $ and is normal to an accumulation point of the direction of 
$ \nabla \mathsf{h}(z_0 + t\mathbf u) $. 
Since $ \mathcal{D} $ is contained in a half-space delimited by 
$ H_0 $, $ L $ must be a supporting hyperplane of $ \rot(\Phi) $ at $ z_0 $. 
Since $ \mathbf u $ points into the interior of $ \rot(\Phi) $, 
the angle between $ \mathbf u $ and $ L $ is bounded away from $0$. 
It follows that for some constant $ c > 0 $ and all sufficiently small $ t > 0 $,
\[
\nabla \mathsf{h}(z_0 + t\mathbf u) \cdot \mathbf u 
\geq c \, \bigl|\nabla \mathsf{h}(z_0 + t\mathbf u)\bigr| \, |\mathbf u| 
\to \infty.
\]
Thus, $g'(t)\to +\infty$ as $t\to0^+$.

Since $g$ is concave, we have 
$$
\frac{\mathsf{h}(z_0+t\mathbf u)-\mathsf{h}(z_0)}{t}=\frac{g(t)-g(0)}{t}\geq g'(t)
\to+\infty \quad \text{as $t\to0^+$}. 
$$
Let $L_J>0$ be a Lipschitz constant for $J$. Then
$$
J(z_0+t\mathbf u)-J(z_0) \geq -L_Jt|\mathbf u|.
$$
Combining the above estimates yields 
\begin{align*}
\frac{(\mathsf{h}+J)(z_0+t\mathbf u)-(\mathsf{h}+J)(z_0)}{t}
\geq 
\frac{\mathsf{h}(z_0+t\mathbf u)-\mathsf{h}(z_0)}{t} 
-L_J|\mathbf u|
\to+\infty \quad \text{as $t\to0^+$}.
\end{align*}
Hence, for all sufficiently small $t>0$,
$$
(\mathsf{h}+J)(z_0+t\mathbf u)>(\mathsf{h}+J)(z_0).
$$
This demonstrates that 
no point $ z_0 $ on the boundary of $\rot(\Phi)$ 
can be a maximizer of the functional $\mathsf{h}(\cdot)+J(\cdot) $, 
completing the proof of the lemma. 
\end{proof}

The result above  is also related to the theory of 
localized topological entropy introduced by Kucherenko and Wolf in \cite{KW14,KW15}.

\begin{definition}[\cite{KW14,KW15}]
Given $n\in\mathbb{N}$ and $\varepsilon, r > 0$, 
and for each point ${\bf w}$ in the rotation set, 
a set $E \subset X$ is called an $(n, \varepsilon, w, r)$-set 
if $E$ is $(n, \varepsilon)$-separated and for all $x \in E$, 
\[
\frac{1}{n} S_n \Phi(x) = \left( \frac{1}{n} S_n \phi_1(x), \dots, \frac{1}{n} S_n \phi_k(x) \right) \in D({\bf w}, r)
\]
For all $n \in \mathbb{N}$ and $\varepsilon, r > 0$, 
let $E_n(\varepsilon, {\bf w}, r)$ be a maximal $(n, \varepsilon, {\bf w}, r)$-set. %(with respect to set inclusion). 
We define
\[
\mathrm{h}({\bf w}, r) = \lim_{\varepsilon \to 0} \limsup_{n \to \infty} 
\frac{1}{n} \log \operatorname{card} E_n(\varepsilon, {\bf w}, r),
\]
with the convention $\log 0 = -\infty$.
The \textbf{localized topological entropy} of $w$ is defined by
\[
\mathrm{h}({\bf w}) = \lim_{r \to 0} \mathrm{h}({\bf w}, r).
\]
  
\end{definition}

Kucherenko and Wolf proved the following variational principle, 
which asserts that for topologically mixing Anosov diffeomorphisms, 
 the localized topological entropy $\mathsf{h}(w)$ coincides with 
the entropy function $\mathrm{h}(w)$. 

\begin{theorem}[\cite{KW15} Corollary 1]\label{KW15-Corollary 1}
Let $f:X\to X$ be a topologically mixing Anosov diffeomorphism on a closed manifold $X$.
Let $\Phi : X \to \mathbb{R}^m$ be continuous and let $w \in \operatorname{Rot}(\Phi)$. Then
\[
\mathsf{h}({\bf w}) = \sup\bigl\{h_\mu(f):\mu\in \mathcal{M}_\Phi({\bf w})\bigr\}
=\mathrm{h}({\bf w}).
\]
\end{theorem}

In \cite{Ku24}, Kucherenko established the following relation between the nonlinear topological pressure and the localized topological entropy.

\begin{theorem}[\cite{Ku24} Theorem 1]\label{Ku24-Theorem 1}
    Let $T: X \to X$ be a continuous map on a compact metric space $X$ and 
    $J: \mathbb{R}^k \to \mathbb{R}$ be a continuous nonlinearity function. 
    Then for any continuous potential $\Phi: X \to \mathbb{R}^k$, we have 
   \begin{align*}
    \Pi_{\top}^J(\Phi) = \sup\left\{\mathsf{h}({\bf w}) + J({\bf w}): {\bf w} \in \operatorname{Rot}_{\pt}(\Phi)\right\}.
 \end{align*}
    \end{theorem}

We introduce the following characterization for $\operatorname{Rot}(\Phi)$ that
has a nonempty interior.

\begin{proposition}[\cite{KW14}, Proposition 6]\label{KW14-Proposition-6}
For a continuous potential 
$\Phi = (\phi_1, \dots, \phi_k): X \to \mathbb{R}^k$,
the following statements are equivalent:
\begin{itemize}
    \item[(i)] No non-trivial linear combination of $\{\phi_i\}_{i=1}^k$ is cohomologous to a constant.
    \item[(ii)] $\operatorname{int} \operatorname{Rot}(\Phi) \neq \varnothing$.
\end{itemize}
\end{proposition}

We say that $\Phi$, or equivalently the family $\{\phi_i\}_{i=1}^k$, satisfies the \textbf{independence condition} 
if it satisfies conditions (i) and (ii) of Proposition \ref{KW14-Proposition-6}.

Given ${\bf v}=(v_1,\cdots,v_d) \in \mathbb{R}^k$, 
we denote by ${\bf v}\cdot \Phi$ the linear combination $v_1 \phi_1 + \cdots + v_d \phi_d$. 
Let $\operatorname{ES}({\bf v}\cdot\Phi)$ denote the set of equilibrium states of $v\cdot\Phi$. 
We write $\ES({\bf v})$ as shorthand for $ES({\bf v}\cdot\Phi)$ when no confusion arises, 
meaning that 
$\ES(\cdot)$ is a map from $\mathbb{R}^k$ to the power set of $\mathcal{M}$. 
We will be interested in the $d$-parameter family $\ES(\mathbb{R}^k)$,
which we call the family of \textbf{$\Phi$-equilibrium states}.
If for every ${\bf v}\in\mathbb{R}^k$, the set $\ES({\bf v}\cdot \Phi)$ is a singleton, 
the system $(X,T,\Phi)$ is called the system with {\bf unique linear equilibrium states}. 

Jenkinson characterized the interior of the rotation set via these equilibrium states in \cite{Jen01}. 
The corresponding result is stated below. 

\begin{proposition}[\cite{Jen01} Corollary 3]\label{Jen01-Corollary 3}
Let $(X, T)$ be a dynamical system for which the entropy map 
$\mu \mapsto h_\mu(T)$ is upper semicontinuous, 
and suppose $\Phi: X \to \mathbb{R}^k$ is continuous. 
If the map ${\bf v} \mapsto P_{\top}(T,{\bf v}\cdot \Phi)$ is strictly convex, then 
\[
\operatorname{int}\bigl(\operatorname{Rot}(\Phi)\bigr) = \rv\left(\ES(\mathbb{R}^k)\right).
\]
\end{proposition}

%Evidently, uniformly hyperbolic systems satisfy the 
%upper semicontinuity of the entropy map in Proposition \ref{Jen01-Corollary 3}.

\subsection{Projections of MMEs}\label{subsect: Projections for measures of maximal entropy}

In this subsection, we use the nonlinear thermodynamic formalism 
to prove that the set of projections of ergodic measures of maximal entropy is finite. 

\begin{theorem}\label{goal theorem}
Let $f\colon X\to X$ be a topologically mixing $C^{1+\beta}$ 
Anosov diffeomorphism on a closed manifold $X$. 
Let $\rho_{\bf t}$ be a locally free $\mathbb{R}^k$-action on a closed manifold $Y$ 
preserving a smooth measure $m_{\MME}$. 
Assume that for each ${\bf t} \in \mathbb{R}^k$, the map $\rho_{\bf t}$ satisfies at least one of the following conditions:
\begin{itemize}
\item $\rho_{\bf t}$ is of class $C^{\infty}$;
\item $Y$ is a hyperbolic set for $\rho_{\bf t}$; 
\item $\rho_{\bf t}$ preserves a $\rho_{\bf t}$-neutralized homogeneous measure, 
\end{itemize} 
and $m_{\mathrm{MME}}$ is an ergodic measure with respect to 
the $\mathbb{R}^k$-action which is a measure of maximal entropy for 
$\rho_{\mathbf{t}}$ for every $\mathbf{t} \in \mathbb{R}^k$. 

Suppose the cocycle $\tau:X\to \mathbb{R}^k$ is $C^{1+\beta}$ and the pair $(\tau,\rho_{\bf t})$ is jointly nontrivial. 
Then, for the skew product $F$, the following set is nonempty and has finite cardinality: 
\[
\mathcal{P}(F)=\left\{ \pi_*\zeta : \zeta \text{ is an ergodic measure of maximal entropy for } F \right\}.
\]
Moreover, each element in $\mathcal{P}(F)$ is the unique equilibrium state 
for some H\"older potential. 
Furthermore, if $f$ is topologically mixing, 
then for every element $\mu$ in $\mathcal{P}(F)$,  
the system $(X,f,\mu)$ is Bernoulli and exponentially mixing of all orders. 
\end{theorem}

In \cite[Theorem D]{BKL23}, 
Buzzi, Kloeckner, and Leplaideur established the finiteness of nonlinear equilibrium states 
for $(f,H)$, where $f$ is a mixing Axiom A diffeomorphism and 
 $H$ is a real analytic function in the case of a single potential ($k=1$). 
Their proof relies on the fact that 
a proper analytic subvariety of a compact line segment is finite.

In contrast, our primary objective is to establish the finiteness 
of ergodic nonlinear equilibrium states for a specific nonlinear potential.   
Due to the lower regularity of the potential function 
and the higher-dimensional setting ($k \geq 1$), 
the analytic subvariety techniques employed in previous works are no longer applicable. 
Instead, we leverage rotation set theory to 
overcome these challenges and achieve the desired finiteness. 

\begin{lemma}\label{at most one nonlinear}
    Let $f\colon X\to X$ be an Anosov diffeomorphism on a closed manifold $X$, and let
    $J: \mathbb{R}^k \to \mathbb{R}$ be a continuous function.
    Given a H\"older continuous potential $\Phi: X \to \mathbb{R}^k$ satisfying the independence condition,
    for each ${\bf w} \in \inter \rot(\Phi)$, 
    there exists at most one nonlinear equilibrium state 
    $\mu\in\mathcal{M}_{\Phi}({\bf w})$. Moreover, this  
    $\mu$ is the unique measure in $\mathcal{M}_{\Phi}({\bf w})$ 
    that satisfies $\mathsf{h}({\bf w})=h_\mu(f)$. 

    Furthermore, 
    if $f$ is topologically mixing, then $(X,f,\mu)$ is exponentially mixing of all orders 
    and Bernoulli. 
\end{lemma}

\begin{proof}
  Assume there exists a nonlinear equilibrium state $\zeta \in \mathcal{M}_\Phi({\bf w})$. 
    By Propositions \ref{KW14-Proposition-6} and Proposition \ref{Jen01-Corollary 3}, 
    there exists ${\bf v}\in\mathbb{R}^k$ such that the unique equilibrium state 
    $\mu_{\bf v}$ for the potential ${\bf v}\cdot \Phi$ satisfies $\mu_{\bf v}(\Phi)={\bf w}$.
    Combining this with Theorem \ref{Ku24-Theorem 1}, we have 
    \begin{align*}
        h_\zeta(f)+J({\bf w})=\Pi_{\mathrm{top}}^J(\Phi)\geq h_{\mu_{\bf v}}(f)+J({\bf w}),
    \end{align*}
    which implies $h_\zeta(f) \geq h_{\mu_{\bf v}}(f)$. On the other hand,
    \begin{align*}
        h_\zeta(f)+{\bf v}\cdot {\bf w} = h_\zeta(f)+\int {\bf v} \cdot \Phi\,\rd\zeta
        \leq h_{\mu_{{\bf v}}}(f)+\int {\bf v}\cdot \Phi\,\rd\mu_{\bf v} 
        = h_{\mu_{{\bf v}}}(f)+{\bf v}\cdot{\bf w}.
    \end{align*}
    This yields $h_\zeta(f) \leq h_{\mu_{\bf v}}(f)$, 
    hence, $h_\zeta(f) = h_{\mu_{\bf v}}(f)$. 
    Consequently, $\zeta$ is also an equilibrium state for the potential ${\bf v}\cdot \Phi$.
    By the uniqueness of the equilibrium state for ${\bf v}\cdot \Phi$, 
    we conclude that $\zeta=\mu_{{\bf v}}$. 
Thus, $\mu_{\bf v}$ is the unique nonlinear equilibrium state in  
    $\mathcal{M}_{\Phi}({\bf w})$. 

    Now suppose there exists another invariant measure $\nu\in\mathcal{M}_{\Phi}({\bf w})$ such that 
    $\mathsf{h}({\bf w})=h_\nu(f)$. Then 
    \begin{align*}
       \Pi_{\mathrm{top}}^J(\Phi)\geq h_\nu(f)+J({\bf w})=
       \mathsf{h}({\bf w})+J({\bf w})\geq h_{\mu_{\bf v}}(f)+J({\bf w})=
        \Pi_{\mathrm{top}}^J(\Phi).
    \end{align*}
This implies that $\nu$ is also a nonlinear equilibrium state in $\mathcal{M}_{\Phi}({\bf w})$, 
and thus $\nu=\mu_{\bf v}$. 
Therefore, $\mu_{\bf v}$ is the unique measure in $\mathcal{M}_{\Phi}({\bf w})$
satisfying $\mathsf{h}({\bf w})=h_{\mu_{\bf v}}(f)$. 

By \cite[Theorem 1.7 and Theorem 4.1]{Bow75}, if $f$ is topologically mixing on $X$, 
then the system $(X,f,\mu)$ is Bernoulli and exponentially mixing of all orders. 
Thus, the lemma is proved. 
\end{proof}

Before demonstrating how to establish the independence condition in Theorem \ref{goal theorem}, 
we present the following result characterizing the Lyapunov exponents for $\mathbb{R}^k$-actions. 

Let $\{\rho_{\bf t}\}$ be a locally free $C^{1+\beta}$ action of $\mathbb{R}^k$ on a manifold $Y$. 
For ${\bf t}\in \mathbb{R}^k$, 
$\rho({\bf t})$ (or $\rho_{{\bf t}}$) induces a tangent map 
$D_x \rho({\bf t}): T_xY\to T_{\rho({\bf t})x}Y$. 

\begin{proposition}[\cite{KKH11} Proposition 2.1]\label{chart1}
Let $\rho_{\bf t}$ be a locally free $C^{1+\beta}$ action of $\mathbb{R}^k$ 
on a manifold $Y$ preserving an ergodic invariant measure $\mu$. 
There are linear functionals $\chi_i$, $i=1,\cdots, L$, 
on $\mathbb{R}^k$ and an $\rho$-invariant measurable splitting called the Lyapunov decomposition, 
of the tangent bundle of $Y$: 
$$
TM=T\mathcal{O}\oplus\bigoplus^L_{i=1}E_i
$$ 
over a set of full measure $\widetilde{\Gamma}$, where $T\mathcal{O}$ is the distribution tangent to the $\mathbb{R}^k$ orbits, such that for any ${\bf t}\in \mathbb{R}^k$ and any nonzero vector $v\in E_i$ the Lyapunov exponent of $v$ is equal to $\chi_i({\bf t})$, i.e. 
\begin{align*}
\lim_{n\to \pm \infty}\frac{\log{\|D_x(\rho(n{\bf t}))v\|}}{n}=\chi_i({\bf t}),
\end{align*}
where $\|\cdot\|$ is any continuous norm on $TM$. 

Furthermore, for any $\epsilon>0$ there exist positive measurable functions $C_{\epsilon}(x)$ and $K_{\epsilon}(x)$ such that for all $x\in \Gamma$, $v\in E_i(x)$, ${\bf t}\in \mathbb{R}^k$, and $i=1, \cdots, L$, 
\begin{itemize}
\item[(1)] $C_{\epsilon}^{-1}(x)e^{\chi_i({\bf t})-\frac{1}{2}\epsilon \|\rho({\bf t})\|}\|v\|\leq \|D_x(\rho({\bf t}))v\|\leq C_{\epsilon}(x)e^{\chi_i({\bf t})+\frac{1}{2}\epsilon \|\rho({\bf t})\|}\|v\|$;
\item[(2)] For each $i$, $\angle(E_i(x), T\mathcal{O})\geq K_{\epsilon}(x)$; 
\item[(3)] For any $i\neq j$, $\angle(E_i(x), E_j(x))\geq K_{\epsilon}(x)$;
\item[(4)] $C_{\epsilon}(\rho({\bf t})x)\leq C_{\epsilon}(x)e^{\epsilon\|\rho({\bf t})\|}$ and $K_{\epsilon}(\rho({\bf t})x)\geq K_{\epsilon}(x)e^{-\epsilon\|\rho({\bf t})\|}$.
\end{itemize}
\end{proposition}

Recall from Proposition \ref{chart1} that 
the Lyapunov functionals $\chi_i$, for $i=1,\cdots, L$,
 are linear functionals on $\mathbb{R}^k$. 
Let $m_i$ denote the multiplicity of the corresponding Lyapunov exponent. 
 For each $i$, there exists a vector ${\bf u}_i=(u_{i,1},\cdots,u_{i,k})\in\mathbb{R}^k$ 
 such that $\chi_i(\cdot)=\langle\cdot,{\bf u}_i\rangle$. 
We define $\phi_i(x):=m_i\langle \tau(x), {\bf u}_i\rangle$. 
Each $\phi_i$ is of class $C^{1+\beta}$, and hence H\"older continuous.

Since $(\tau, \rho_{\bf t})$ is jointly nontrivial, 
 not all $\phi_i$ vanish identically. 
We take a maximal subset of $\{\phi_i : 1\leq i\leq L\}$ that 
satisfies the independence condition, which is nonempty. 
Without loss of generality, we assume this maximal subset is
$\{\phi_i : 1\leq i\leq l\}$ for some $1\leq l\leq L$. 
Then for every $j\in \{l+1,\cdots,L\}$, 
there exist constants $a_{j,i}$, $c_j$ and continuous functions 
$\rho_j:X\to \mathbb{R}$ such that
\[
\phi_j = \sum_{i=1}^l a_{j,i}\phi_i + c_j+\rho_j-\rho_j\circ f.
\]

\iffalse
By introducing a continuous map $\widetilde{\tau}:X\to \mathbb{R}^k$ defined by 
$$
\widetilde{\tau}(x)=(0,\cdots,0,\rho_{l+1}(x),\cdots,\rho_{k}(x)), 
$$
and a change of variables $h(x, y)= \left(x, \rho_{\widetilde{\tau}(x)}y\right)$, 
we obtain 
\begin{align*}
h\circ F\circ h^{-1}(x,y)&=h\circ F\left( x, \rho_{-\widetilde{\tau}(x)}y \right)=
h\left( fx, \rho_{\tau(x)-\widetilde{\tau}(x)}y \right)\\
&=\left( fx, \rho_{\tau(x)-\widetilde{\tau}(x)+\widetilde{\tau}(fx)}y \right).
\end{align*}
Thus $F$ is topologically conjugated to $\bar F$ where
$\bar{F}(x,y)=\left( fx, \rho_{\tau(x)-\widetilde{\tau}(x)+\widetilde{\tau}(fx)}y \right)$. 
This conjugation allows us to eliminate the coboundary terms. 
Therefore, without loss of generality, we may assume that
\[
\phi_j = \sum_{i=1}^l a_{j,i}\phi_i + c_j
\quad
\text{for } l+1\leq j\leq k.
\]
\fi

We are now able to conclude the proof of Theorem \ref{goal theorem}.

\begin{proof}[Proof of Theorem \ref{goal theorem}]
Recalling Proposition \ref{chart1}, 
we use $\chi_i$ and $m_i$ ($i=1,\dots,L$) to denote the Lyapunov functionals 
and corresponding multiplicities of the $\mathbb{R}^k$-action with respect to the 
smooth ergodic measure of maximal entropy $m_{\MME}$. 

We define a Lipschitz continuous function $J:\mathbb{R}^{l}\to \mathbb{R}$ by 
 $$
 J(x_1,\cdots, x_l)=\sum_{j=1}^l \max\{x_j,0\}+\sum_{j=l+1}^L \max \left\{ \sum_{i=1}^l a_{j,i} x_i+c_j,0\right\}. 
 $$
Since $m_{\MME}$ is the measure of maximal entropy for $\rho(\mu(\tau))$, 
by the definition of $J$, for every $\mu\in\mathcal M(f,X)$, we have 
\begin{align*}
h_{\top}\left(\rho_{\mu(\tau)}\right)
&=h_{m_{\MME}}\left(\rho_{\mu(\tau)}\right)
=\sum_{j=1}^L \max\{m_j \chi_j(\mu(\tau)),0\}
=\sum_{j=1}^L \max\{\mu(\phi_j),0\}\\
&=\sum_{j=1}^l \max\{\mu(\phi_j),0\}+\sum_{j=l+1}^L\max\left\{ \sum_{i=1}^l a_{j,i} \mu(\phi_i) +c_j,0\right\}\\
&=J(\mu(\phi_1),\cdots,\mu(\phi_l))=J(\mu(\Phi)).
\end{align*} 
Since Anosov diffeomorphisms have an abundance of ergodic measures,
Theorem \ref{BKL23-Theorem A} together with Corollary
\ref{higher le:fiber-entropy-corollary} gives 
\begin{align*}
h_{\top}(F)
&=\sup \left\{ h_\mu(f)+h_{\top}\left(\rho_{\mu(\tau)}\right): \mu\in\mathcal M_e(f,X)\right\}\\
&=\sup \left\{ h_\mu(f)+J(\mu(\Phi)): \mu\in\mathcal M_e(f,X)\right\}\\
&=\sup \left\{ h_\mu(f)+J(\mu(\Phi)): \mu\in\mathcal M(f,X)\right\}. 
\end{align*}
Thus, if $\zeta$ is an ergodic measure of maximal entropy for $F$
and $\mu=\pi_*\zeta$, then
$$
h_{\top}(F)=h_\zeta(F)
\leq h_\mu(f)+h_{\top}(\rho_{\mu(\tau)})
=h_\mu(f)+J(\mu(\Phi))
\leq h_{\top}(F).
$$
Hence, all the above inequalities are equalities, and $\mu$ is an
ergodic nonlinear equilibrium state for $(f,J)$. 
Hence, to prove Theorem \ref{goal theorem}, it suffices to show that 
$(f,J)$ has at least one and at most finitely many ergodic nonlinear equilibrium states. 

First, we show the existence of a nonlinear equilibrium state. 
Indeed, since $f$ is an Anosov diffeomorphism, 
the entropy map $\mu\mapsto h_\mu(f)$
is upper semicontinuous on the compact space $\mathcal M(f,X)$.
Moreover, since $\Phi$ and $J$ are continuous, the map 
$\mu\mapsto J(\mu(\Phi))$ is continuous. Therefore, the map 
$$
\mu\mapsto h_\mu(f)+J(\mu(\Phi))=h_\mu(f)+h_{\top}\left(\rho_{\mu(\tau)}\right)
$$
is upper semicontinuous on $\mathcal M(f,X)$ and hence attains its maximum. 
Thus, there exists at least one nonlinear equilibrium state. 
Combining this with Lemmas \ref{not boundary} and \ref{at most one nonlinear},
any such nonlinear equilibrium state is ergodic. 
Clearly, $\mu\times m_{\MME}$ is $F$-invariant and, by Theorem
\ref{higher le:product}, it is a measure of maximal entropy for $F$.
Taking an ergodic component if necessary, $F$ therefore admits at
least one ergodic measure of maximal entropy.

Based on the definition of $J$, 
we can cover the domain $\rot(\Phi)$ by at most $2^L$ compact convex regions. 
Each region is of the form 
\begin{align*}
\biggl\{(x_1,\cdots,x_l)\in \rot(\Phi): &
x_i \in \mathbb{R}_{\geq 0} \text{ or } \mathbb{R}_{\leq 0} \ \ \forall\, 1\leq i\leq l,\\
&\sum_{i=1}^l a_{j,i}x_i +c_j \in \mathbb{R}_{\geq 0} \text{ or } \mathbb{R}_{\leq 0} \ \ \forall\, l+1\leq j\leq L \biggr\} .
\end{align*}
On each such region $\mathcal{R}$, $J$ is an affine function of the form  
$$ 
 J(\omega_1,\cdots,\omega_l)=\sum_{j=1}^l B_j \omega_j+C 
$$ 
for some constants $B_j$ and $C$. Note that each region is compact and convex. 

By the discussion in the proof of Lemma \ref{not boundary}
(see also \cite[Corollary 3.12]{BKL23}), the entropy function
$\mathsf h(\cdot)$ is strictly concave on $\inter\rot(\Phi)$,  
implying that $\mathsf{h}(\cdot)+J(\cdot)$ 
is also strictly concave on $\mathcal{R}\cap \inter \rot (\Phi)$. 
According to Lemma \ref{not boundary}, 
every global maximizer of $\mathsf{h}(\cdot)+J(\cdot)$ 
contained in $\mathcal{R}$ belongs to the interior $\inter \rot(\Phi)$. 

Suppose there exist two distinct points ${\bf w}_1,{\bf w}_2\in\mathcal{R}$ 
attaining the global maximum of $\mathsf{h}(\cdot)+J(\cdot)$, 
then ${\bf w}_1,{\bf w}_2\in \mathcal{R}\cap \inter \rot (\Phi)$. 
By the strict concavity of $\mathsf{h}(\cdot)+J(\cdot)$ on 
$\mathcal{R}\cap \inter \rot (\Phi)$, for any $t\in (0,1)$, 
we have 
$t{\bf w}_1+(1-t){\bf w}_2\in\mathcal{R}\cap \inter \rot (\Phi)$ and 
\begin{align*}
\mathsf{h}(t{\bf w}_1+(1-t){\bf w}_2)+J(t{\bf w}_1+(1-t){\bf w}_2)&> 
t\left[\mathsf{h}({\bf w}_1)+J({\bf w}_1)\right]+(1-t)\left[\mathsf{h}({\bf w}_2)+J({\bf w}_2)\right]\\
&=\mathsf{h}({\bf w}_1)+J({\bf w}_1), 
\end{align*}
which contradicts the fact that ${\bf w}_1$ attains the global maximum of $\mathsf{h}(\cdot)+J(\cdot)$. 
Thus, $\mathsf{h}(\cdot)+J(\cdot)$ has at most one global maximizer on each region $\mathcal{R}$. 
Combining this with Lemma \ref{at most one nonlinear}, 
each global maximizer gives rise to at most one ergodic nonlinear equilibrium state, 
and every such state is the unique equilibrium state for a
H\"older continuous potential of the form ${\bf v}\cdot\Phi$.
Since there are at most $2^L$ regions, $(f,J)$ possesses at most
$2^L$ ergodic nonlinear equilibrium states. 
This establishes the first part of the theorem.

Furthermore, if $f$ is topologically mixing, 
Lemma \ref{at most one nonlinear} further ensures that 
the dynamical systems $(X, f, \mu)$ 
associated with every element $\mu\in\mathcal{P}(F)$ are Bernoulli 
and exponentially mixing of all orders, 
thereby completing the proof of Theorem \ref{goal theorem}. 
\end{proof}

\begin{remark}\label{goal theorem-remark}
It is worth mentioning that in the proof of Theorem \ref{goal theorem}, 
the number of projections of these ergodic measures of maximal entropy 
for $F$ ranges from $1$ and $2^{\dim Y-k}$. 
This upper bound is a direct consequence of the estimate $L\leq \dim Y-k$. 
Since the signs defining each nonempty region are precisely the 
signs of the Lyapunov functionals evaluated at the drift vector, 
each such region is naturally associated with a unique Weyl chamber. 
This observation allows us to further improve the above bound to
$$
2\sum_{j=0}^{k-1}\binom{\dim Y-k-1}{j},
$$
as discussed in \S \ref{subsect: On the number of ergodic MMEs}.
Moreover, by comparing the maximum values of $\mathsf{h}+J$ over these regions, 
we conclude that the projection of every ergodic measure of maximal entropy 
is associated with a region on which $\mathsf{h}+J$ attains its global maximum, 
and each such region gives rise to at most one such projection.
Furthermore, if the entropy contribution from the fiber is zero, 
then the projection of any ergodic measure of maximal entropy 
of the skew product $F$ is the unique measure of maximal entropy of $f$.
\end{remark}

\begin{lemma}\label{tau-k1-dimension}
Let $\tau: X \to \mathbb{R}^k$ be a H\"older cocycle which is not
cohomologous to a cocycle taking values in any affine subspace of
$\mathbb{R}^k$ of dimension strictly less than $k_1$.  Let
$$
S^\perp:=
\left\{D \in \mathbb{R}^k :
D \cdot \tau : X \to \mathbb{R}
\text{ is cohomologous to a constant}\right\},
$$
and let $S$ denote the orthogonal complement of $S^\perp$. Then
$$
\operatorname{codim} S^\perp = \dim S \geq k_1.
$$
\end{lemma}

\begin{proof}
Suppose for contradiction that $\dim S < k_1$. 
Let $j := \dim S^\perp = k - \dim S$, 
and let $\{D_1,\ldots,D_j\}$ be an orthonormal basis of $S^\perp$.
By the definition of $S^\perp$, for each $1 \leq i \leq j$, 
there exist a continuous function $P_i : X \to \mathbb{R}$ and $c_i\in \mathbb{R}$ such that
$$
D_i \cdot \tau(x) = P_i(f(x)) - P_i(x)+c_i.
$$
Define
$$
\eta(x):=\tau(x)-\sum_{i=1}^j \bigl(P_i(f(x)) - P_i(x)\bigr) D_i,\quad 
{\bf c}=\sum_{i=1}^j c_iD_i.
$$
Since $\{D_1,\ldots,D_j\}$ is an orthonormal set, 
taking the inner product with $D_l$ for any $1 \leq l \leq j$ gives
$$
D_l \cdot \eta(x)
=D_l \cdot \tau(x) - \bigl(P_l(f(x)) - P_l(x)\bigr) = c_l.
$$
Hence $\eta(x) \in {\bf c}+S$ for every $x \in X$. 
Setting 
$$
P(x) := \sum_{i=1}^j P_i(x) D_i, 
$$
we obtain $P : X \to \mathbb{R}^k$ and
$$
\tau(x) - \eta(x) = P(f(x)) - P(x).
$$
Thus, $\tau$ is cohomologous to the cocycle $\eta$, 
which takes values in ${\bf c}+S$. Since $\dim S < k_1$, 
this contradicts the assumption on $\tau$. This completes the proof. 
\end{proof}

\begin{corollary}\label{projection not MME}
Under the same assumptions as in Theorem \ref{goal theorem}, 
if the pair $(\tau, \rho_{\bf t})$ is jointly nontrivial, 
then the measure $\pi_*\zeta$ 
is not the measure of maximal entropy $\mu_{0}$ of $f$. 
In addition, the vector $\pi_*\zeta(\tau)\neq {\bf 0}$ and $h_{\top}(\rho_{\pi_*\zeta(\tau)})>0$. 
\end{corollary}

\begin{proof} 
Suppose $\mu_0(\Phi)$ lies in a region $\mathcal R$ 
with nonempty interior on which $J$ is affine, given by  
\[ 
J(\omega_1,\dots,\omega_l)=\sum_{j=1}^l B_j \omega_j+C 
\] 
for some constants $B_j$ and $C$.  
Let $B=(B_1,\dots,B_l)$ and let $\mu_{B}$ denote the equilibrium state of $f$ 
for the potential $B\cdot \Phi$.

We first show that $B\neq (0,\dots,0)$. 
Suppose otherwise that $B=(0,\dots,0)$. 
Then, by Lemma \ref{tau-k1-dimension} and the condition (C1), 
the topological entropy  $h_{\top}(\rho_{\bf t})$ is constant on 
a nonempty open subset  of a $k$-dimensional affine subspace of $\mathbb R^k$. 
Since, for every ${\bf t}\in\mathbb R^k$ and $s>0$, we have 
$$ 
h_{\top}(\rho_{s{\bf t}}) = s\,h_{\top}(\rho_{\bf t}), 
$$ 
this constant must be zero. Hence, we obtain 
$$ 
\dim Z_\rho=\dim\{{\bf t}\in\mathbb R^k: 
h_{\top}(\rho_{\bf t})=0\}= k, 
$$ 
which contradicts $\dim Z_\rho<k$. 
Therefore, $B\neq(0,\dots,0)$.

Due to the independence of $\Phi$ and the fact that $B \neq (0,\dots,0)$,  
it follows that $\mu_B \neq \mu_0$.  

Moreover, by the definition of $J$, the affine function
$\omega \mapsto B\cdot\omega+C$ corresponding to the region $\mathcal R$ satisfies
$$
J(\omega)\geq B\cdot\omega+C
$$
for every $\omega\in\rot(\Phi)$, with equality at 
$\omega=\mu_0(\Phi)$. Therefore,
by the variational principle for the potential $B \cdot \Phi$, we obtain  
\begin{align*} 
h_{\mu_B}(f) + J(\mu_B(\Phi))  
&\geq h_{\mu_B}(f) + \mu_B(B \cdot \Phi) + C \\ 
&> h_{\mu_0}(f) + \mu_0(B \cdot \Phi) + C \\ 
&= h_{\mu_0}(f) + J(\mu_0(\Phi)). 
\end{align*} 
The above inequality shows that $\mu_0$ fails to maximize  
$h_\mu(f) + J(\mu(\Phi))$.  
Consequently, $\mu_0$ cannot be the marginal measure $\pi_*\zeta$  
constructed in Theorem \ref{goal theorem}.

Finally, let $\mu=\pi_*\zeta$. Suppose that $h_{\top}(\rho_{\mu(\tau)})=0$.
Since $\mu$ is a nonlinear equilibrium state, we have
$$
h_{\top}(F) =h_\mu(f)+h_{\top}(\rho_{\mu(\tau)})=h_\mu(f).
$$
On the other hand, we obtain
$$
h_{\top}(F)
\geq
h_{\mu_0}(f)+h_{\top}(\rho_{\mu_0(\tau)})
\geq h_{\mu_0}(f)\geq h_{\mu}(f)=h_{\top}(F).
$$
Thus, we have $h_{\mu_0}(f)= h_{\mu}(f)$. 
Since $\mu_0$ is the unique measure of maximal entropy for $f$, 
it follows that $\mu=\mu_0$, which is a contradiction. Therefore, we conclude 
$$
h_{\top}(\rho_{\pi_*\zeta(\tau)})>0.
$$
Since $h_{\top}(\rho_{\bf 0})=0$, it follows that
$\pi_*\zeta(\tau)\neq {\bf 0}$. 
This completes the proof.
\end{proof}

\section{Relative Ledrappier--Young theory}\label{sect:relative-LY}

In this section, we prove Theorem \ref{relative SRB} concerning leafwise measure rigidity. 

\subsection{Invariant foliations and charts for the skew product}\label{subsect: Invariant foliations for the skew product}

In this subsection, we present some results concerning 
invariant foliations and Lyapunov exponents adapted to the skew product setting. 
Neither smoothness nor neutralized-homogeneity on $\zeta$, 
nor joint nontriviality, is required here. 

For the skew product $F$, 
let $T_{(x,y)}(X\times Y)$ and $T_yY$ denote the tangent spaces to 
$X\times Y$ at $(x,y)$ and to $Y$ at $y$, respectively. 
For a foliation $\mathcal{W}$ and a point $z$ on a leaf, 
$T_z\mathcal{W}$ denotes the tangent space to the leaf through $z$. 
For each $x\in X$, let $W^s(x)$ and $W^u(x)$ denote 
the stable and unstable leaves of $f$ passing through $x$. 
For any $x\in X$ and $y\in W^s(x)$, we define 
$$
\beta^s(x,y)=\sum_{k=0}^\infty \left(\tau(f^kx)-\tau(f^ky)\right). 
$$
The sum $\beta^s$ is well-defined since $\tau(f^kx)-\tau(f^ky)$ decays exponentially. 
We define 
\begin{align}\label{def:widetildeWs}
\widetilde{W}^s(x,z) = \{(y, \rho_{\beta^s(x,y)}z) : y \in W^s(x)\}. 
\end{align}  

Next, we recall several auxiliary results for the skew product studied in \cite{DQ26}. 

\begin{lemma}[\cite{DQ26} Lemma 3.1]\label{Dong23-Lemma 3.1}
For any $(x,z)\in X\times Y$,
$\widetilde{W}^s(x,z)$ is a smooth invariant submanifold of $X\times Y$.
Moreover, it is a stable manifold of $F$ whose projection onto $X$ is $W^s(x)$.
\end{lemma}

\begin{lemma}[\cite{DQ26} Lemma 3.2]\label{base LY}
Let $\zeta$ be an ergodic $F$-invariant measure, 
and let $\mu := \pi_* \zeta$ be the ergodic $f$-invariant measure. 
Fix $(x,z) \in X \times Y$. 
Then the Lyapunov exponents of $f$ along $W^s(x)$ with respect to $\mu$ 
coincide with those of $F$ along $\widetilde{W}^s(x,z)$ with respect to $\zeta$. 
\end{lemma}

\begin{lemma}[\cite{DQ26} Lemma 3.3]\label{all LY}
Suppose $\zeta = \mu \times \nu$ is an ergodic $F$-invariant measure on $X \times Y$. 
Then the Lyapunov exponents of $F$ with respect to $\zeta$ are the 
Lyapunov exponents of $f$ with respect to $\mu$ on $X$ 
and the Lyapunov exponents of $\rho_{\mu(\tau)}$ with respect to $\nu$ on $Y$.
\end{lemma}

For a general ergodic measure $\zeta$ with $\pi_*\zeta=\mu$, 
we only know the existence of Lyapunov exponents for the tangent space $\{0\}\times T_yY$. 
Since such a measure does not necessarily admit a product decomposition, 
it is difficult to identify the specific fiber measures 
needed to define the Lyapunov exponents of the $\rho$-action. 
Furthermore, the unstable foliation associated with 
the positive Lyapunov exponents corresponding to the tangent
space $\{0\} \times T_yY$ depends only on the $Y$ coordinate. 

The case for the unstable manifold is analogous. 
We omit the proof here, but present the expression below. 
For any $x,y\in X$, we define 
\begin{align*}
\beta^u(x,y):=-\sum_{k=1}^\infty \left( \tau(f^{-k}x)-\tau(f^{-k}y) \right). 
\end{align*}
The unstable foliation corresponding to $W^u(x)$ is given by 
\begin{align*}
\widetilde{W}^u_f(x,z) = \{(y, \rho_{\beta^u(x,y)}z) : y \in W^u(x)\}. 
\end{align*}

We now turn to the corresponding pointwise dimensions for $f$ and $F$.

\begin{lemma}\label{base same dimension}
Let $\zeta$ be an $F$-invariant measure on $X\times Y$ and let
$\mu=\pi_*\zeta$. Then, for $\zeta$-a.e.\ $(x,z)\in X\times Y$,
the stable pointwise dimension of $\zeta$ along
$\widetilde W^s(x,z)$ coincides with the stable pointwise dimension
of $\mu$ along $W^s(x)$.
A similar result holds for the unstable case.
\end{lemma}

\begin{proof}
By the definition of $\widetilde W^s(x,z)$, we note that the
coordinate transformation in the $Y$-component depends only on
the base coordinates and is independent of $z$
(see \eqref{def:widetildeWs}). Consequently,
$$
\pi\bigl(\widetilde W^s(x,z)\bigr)=W^s(x).
$$
Moreover, the restriction of $\pi$ to every sufficiently small
local plaque of $\widetilde W^s(x,z)$ is a smooth bijection onto
a local plaque of $W^s(x)$. Hence, the corresponding leafwise
metrics are locally bi-Lipschitz equivalent.

We note that no ergodicity assumption is needed here. Indeed, the
argument may be applied to each ergodic component separately, since
the entire stable leaf through a regular point is contained in the
corresponding ergodic component; see \cite[Section 6.2]{LY85-1}. 

Recall from \cite{LY85-1,LY85-2} that there exists a measurable
partition $\eta^s$ of $X\times Y$ such that, for $\zeta$-a.e.\
$(x,z)\in X\times Y$,
$$
\eta^s(x,z)\subset\widetilde W^s(x,z),
$$
and $\eta^s(x,z)$ contains the intersection of an open
neighborhood of $(x,z)$ with $\widetilde W^s(x,z)$.
Let $\widetilde\zeta^s_{(x,z)}$ be the corresponding conditional
measure.

Projecting the local atoms of $\eta^s$ by $\pi$, 
we obtain a measurable partition $\widehat\eta^s$ of $X$ 
subordinate to the stable foliation $W^s$, 
up to a set of $\mu$-measure zero. 
Let $\mu_x^s$ denote the corresponding conditional measure of
$\mu$.

Since $\pi_*\zeta=\mu$ and $\pi$ sends the local subordinate
plaques of $\eta^s$ onto those of $\widehat\eta^s$, and since
the stable pointwise dimension is independent of the choice of
the subordinate measurable partition \cite{LY85-2}, we conclude
that, for $\zeta$-a.e.\ $(x,z)$, 
\begin{align*}
\lim_{\epsilon\to0}
\frac{
\log\widetilde\zeta^s_{(x,z)}
\bigl(\widetilde W^s((x,z),\epsilon)\bigr)}
{\log\epsilon}
&=
\lim_{\epsilon\to0}
\frac{
\log\mu_x^s\bigl(W^s(x,\epsilon)\bigr)}
{\log\epsilon}.
\end{align*}
Here we use the local bi-Lipschitz equivalence between
$\widetilde W^s(x,z)$ and $W^s(x)$; replacing the radius
$\epsilon$ by a fixed multiple of $\epsilon$ does not affect
the pointwise dimension.

Thus, for $\zeta$-a.e.\ $(x,z)$, the stable pointwise dimension
of $\zeta$ along $\widetilde W^s$ coincides with the stable
pointwise dimension of $\mu$ along $W^s$. 
The unstable case is similar. 
\end{proof}

For convenience, we use $\widetilde{W}^u$ to denote the unstable manifold 
whose projection is the unstable manifold of $f$, 
and $\widehat{W}^u$ to denote the unstable manifold 
corresponding to the tangent space $\{0\}\times T_yY$. 
Similar definitions and notations are adopted for the stable case. 

\begin{corollary}\label{all LY-corollary}
  Let $\zeta = \mu \times \nu$ be an ergodic $F$-invariant measure on $X \times Y$. 
The transversal Hausdorff dimensions of $F$ on $\widehat{W}^s$ with respect to $\zeta$ coincide with 
those of $\rho_{\mu(\tau)}$ with respect to $\nu$. A similar result holds for $\widehat{W}^u$. 
\end{corollary}

\begin{proof}
From the proof of Lemma \ref{all LY}, 
 the Oseledec decomposition of the subspace $\{0\}\times T_yY $ for $F$
coincides with the decomposition of $T_yY$ for $\rho_{\mu(\tau)}$ 
(we identify the set $\{0\}\times A$ with $A$), 
and this decomposition is independent of $x\in X$ for $\zeta$-a.e. $(x,y)$. 
Since the foliation $\widehat{W}^s$ is integrable to 
some coarse Lyapunov foliations and  
 $\zeta$ is the product measure $\mu\times \nu$, 
we conclude the desired result. 
\end{proof}

Recall that $\widehat{\Gamma}_X$ and $\widehat{\Gamma}_Y$ 
are the sets of regular points for $f$ and $\{\rho_{\bf t}\}$, respectively. 
Let $\Gamma'$ denote the set of regular points for the skew product $F$. 
Clearly, $\Gamma'\subset \widehat{\Gamma}_X\times \widehat{\Gamma}_Y$. 
The construction of the following Lyapunov charts on $\Gamma'$ follows 
the same argument as \cite[Proposition 4.1]{Hu93}, 
and we therefore omit the details.

For $x\in \Gamma'$, we denote by $\lambda_i(x), i=1, \ldots, q(x)$ 
the Lyapunov exponents of $F$ and set 
$$\lambda^+(x)=\min\{\lambda_i(x), \lambda_i(x)>0\},\quad  
\lambda^-(x)=\max\{\lambda_i(x), \lambda_i(x)<0\},
$$ 
and
\begin{align*}
    E^s (x) =\underset{\lambda_i(x)<0}{\bigoplus} E_i(x), 
    \quad E^u(x)  =\underset{\lambda_i(x)>0}{\bigoplus} E_i(x), 
    \quad E^c(x) & =E_{i_0}(x) \,\text{ with } \lambda_{i_0}(x)=0.
   \end{align*}
For any $  (x, y, z) \in \mathbb{R}^u \times \mathbb{R}^c \times \mathbb{R}^s$,
we define 
\[
|(x, y, z)|=\max \left\{|x|_u,|y|_c,|z|_s\right\}
\]
where $|\cdot|_u,|\cdot|_c$ and $|\cdot|_s$ are the Euclidean norms on $\mathbb{R}^u, \mathbb{R}^c$ and $\mathbb{R}^s$ respectively.
The closed disk in $\mathbb{R}^u \times \mathbb{R}^c \times \mathbb{R}^s$ of radius $\rho$ centered at $0$ is denoted by $R(\rho)$. 

We assume the measure $\zeta$ is ergodic and denote by 
$\lambda^+$ and $\lambda^-$ the constant values of 
$\lambda^+(x)$ and $\lambda^-(x)$ for $\zeta$-almost every $x$. 
For $0<\epsilon<\min\{\frac{1}{100}\lambda^{+},-\frac{1}{100}\lambda^{-}\}$, 
there exists a measurable function $l: \Gamma^{\prime} \rightarrow[1, \infty)$ 
with $l\left(F^{ \pm} x\right) \leq$ $e^{\epsilon} l(x)$
and an embedding $\Phi_x: R\left(l(x)^{-1}\right) \rightarrow M$ 
such that 
\begin{itemize}
  \item[(i)] $\Phi_x 0=x $, $D \Phi_x(0)$ takes $\mathbb{R}^u, \mathbb{R}^c$ and $\mathbb{R}^s$ to $E^u(x), E^c(x)$ and $E^s(x)$, respectively.

  \item[(ii)] 
We define the lifting map $\tilde{F}_x=\Phi_{F(x)}^{-1} \circ F \circ \Phi_x$
 whenever it is well-defined. 
Similarly, we define the function $\tilde{F}_x^{-1}= \Phi_{F^{-1} x}^{-1} \circ F^{-1} \circ \Phi_x$. 
Then  
\begin{align*}
        e^{\lambda^{+}-\epsilon}|v| & \leq\left|D \tilde{F}_x(0) v\right| ,       &&             \text { for } v \in \mathbb{R}^u; \\
        e^{-\epsilon}|v|           & \leq\left|D \tilde{F}_x(0) v\right| \leq e^{\epsilon}|v|,  &&\text { for } v \in \mathbb{R}^c; \\
      \left|D \tilde{F}_x(0) v\right|& \leq e^{\lambda^{-}+\epsilon}|v|,  && \text { for } v \in \mathbb{R}^s .
\end{align*}

  \item[(iii)] We use $L(g)$ to denote the Lipschitz constant of the function $g$. Then
    \begin{align*}
      L\left(\tilde{F}_x-D \tilde{F}_x(0)\right) \leq \epsilon,
      \quad L\left(\tilde{F}_x^{-1}-D \tilde{F}_x^{-1}(0)\right) \leq \epsilon,
      \quad   L\left(D \tilde{F}_x\right), L\left(D \tilde{F}_x^{-1}\right) \leq l(x) .
    \end{align*}

  \item[(iv)] For any $z, z^{\prime} \in R\left(l(x)^{-1}\right)$, 
  there exists a constant $K>0$ such that
    $$
      K^{-1} d\left(\Phi_x z, \Phi_x z^{\prime}\right) \leq\left|z-z^{\prime}\right| \leq l(x) d\left(\Phi_x z, \Phi_x z^{\prime}\right).
    $$
\end{itemize}

Since $f$ is an Anosov diffeomorphism 
and $\tau$ is $C^{1+\beta}$ on the compact space $X$, 
the existence of such a function $l$ follows immediately from Proposition \ref{chart2}.
For $l>0$ and $\epsilon>0$, 
we define the {\bf Pesin set (or Pesin block)} $\widehat{\Gamma}_{l,\epsilon}$ by 
\begin{align*}
\widehat{\Gamma}_{l,\epsilon}:=\left\{x\in \Gamma': l(x)\leq l \right\}.
\end{align*}
It follows that the set of regular points $\Gamma'$ satisfies
$$
\Gamma' = \bigcup_{l=1}^\infty \widehat{\Gamma}_{l,\epsilon}.
$$
Consequently, we have $\displaystyle \zeta\left(\bigcup_{l=1}^\infty \widehat{\Gamma}_{l,\epsilon}\right) = \zeta(\Gamma') = 1$ for any $F$-invariant measure $\zeta$. 
In particular, for sufficiently large $l$, 
we have $\zeta(\widehat{\Gamma}_{l,\epsilon})>0$.
Note that if $\zeta$ is ergodic, the Lyapunov exponents 
and their corresponding transverse Hausdorff dimensions 
are constant $\zeta$-almost everywhere. 

\subsection{Fiberwise SRB and conditional measure rigidity}\label{subsect: Fiberwise SRB and Conditional Measure Rigidity}

In this subsection, we establish a relative version of the 
Ledrappier--Young theorem for the skew product. 
Specifically, it characterizes the absolute continuity of 
conditional measures along the fiberwise unstable foliation in terms of 
the fiber contribution to metric entropy. 

We first establish the relative Ledrappier-Young theorem in the ergodic case, 
namely Theorem \ref{relative SRB}. 
In fact, when $\zeta$ is an ergodic product measure $\mu\times \nu$ on $X\times Y$, 
Theorem \ref{relative SRB} immediately follows from \cite[Theorem A]{LY85-1}. 

\begin{proof}[Proof of Theorem \ref{relative SRB}]
Combining the Ledrappier-Young entropy formula (\cite[Theorem C]{LY85-2}), 
Lemma \ref{base LY}, and Lemma \ref{base same dimension}, 
the ``if'' part follows. 

We now turn to the ``only if'' part.
Without loss of generality, 
we assume that the unstable subspace 
associated with positive Lyapunov exponents has a positive dimension; 
otherwise, the theorem holds trivially. 

By \cite[Lemma 3.1.1]{LY85-1}, there exists a measurable partition $\eta$ on $X\times Y$ 
satisfying the following properties: 
\begin{itemize} 
\item[(i)] $\eta$ is an increasing partition subordinate to the unstable manifold $\widehat{W}^u$ of $F$;
\item[(ii)] $\displaystyle \bigvee_{n=0}^\infty F^{-n}\eta$ is the partition into points. 
\end{itemize}
%By \cite[Corollary 5.3]{LY85-1}, we have $H_\zeta(\eta|F\eta)=h_\zeta(F,\eta)$. 
Let $\{\zeta_z\}=\{\zeta^\eta_z\}$ be the conditional probability measure of $\zeta$ 
with respect to the partition $\eta$.

For $\zeta$-almost every $z_1 \in M$, there exist constants $l_0, \epsilon_0 > 0$, possibly depending on $z_1$, such that $z_1 \in \widehat{\Gamma}_{l_0, \epsilon_0}$. For any such point $z_1$ and $z_2 \in \eta(z_1) \subset \widehat{W}_{\mathrm{loc}}^u(z_1)$, we define
\begin{align*}
\kappa^u(z_1,z_2)
:=
\prod_{i=1}^\infty
\frac{
\operatorname{Jac}\bigl(DF^{-1} \mid \widehat{E}^u(F^{-i}z_2)\bigr)
}{
\operatorname{Jac}\bigl(DF^{-1} \mid \widehat{E}^u(F^{-i}z_1)\bigr)
},
\end{align*}
where $\operatorname{Jac}$ denotes the Jacobian of the corresponding linear map.

The unstable distribution $\widehat E^u$ varies H\"older
continuously along local unstable manifolds. Combining this with 
the exponential decay of $d(F^{-i}z_1,F^{-i}z_2)$ and 
the H\"older continuity of $DF^{-1}$, we conclude that $\kappa^u(z_1,z_2)$ converges. 
Moreover, for each fixed $z_1$, 
 $\kappa^u(z_1,z_2)$ is H\"older continuous and 
uniformly bounded away from $0$ and $\infty$ on $\widehat{W}^u(z_1)$ 
with respect to  $z_2$. 
The same holds for the function $z_2 \mapsto \log \kappa^u(z_1,z_2)$. 
We set 
\[
\rho(z_1,z_2) = \frac{\kappa^u(z_1,z_2) }{ \int_{\eta(w)} \kappa^u(z_1,\tau) \, \rd m_{z_1}(\tau)}.
\]
We define a new measure $\nu$ on $X\times Y$ such that 
\begin{itemize}
\item[(1)] The conditional measures $\nu^u_{z_1}$ induced by $\nu$ on $\widehat{W}^u(z_1)$ 
are equivalent to the conditional measure $m_{z_1}$ with density function 
$\rho(z_1,z_2)$. That is, 
\[
\rd\nu^u_{z_1} = \rho(z_1,z_2) \, \rd m_{z_1};
\]
\item[(2)] The measure $\nu$ coincides with $\zeta$ 
on the $\sigma$-algebra $\mathcal{B}_\eta$ consisting of measurable sets 
that are unions of elements of $\eta$; 
\item[(3)] $\pi_*\nu=\mu$.
\end{itemize} 
We set 
\[
q(z_1) = \nu^u_{z_1}((F^{-1}\eta)(z_1)), \quad \text{and} \quad 
L(z_1) = \int_{\eta(z_1)} \kappa^u(z_1,\tau) \, \rd m_{z_1}(\tau).
\]
Then, we deduce 
\begin{align*}
q(z_1) = \frac{\int_{(F^{-1}\eta)(z_1)} \kappa^u(z_1,\tau) \, \rd m_{z_1}(\tau)}
{\int_{\eta(z_1)} \kappa^u(z_1,\tau) \, \rd m_{z_1}(\tau)} 
= \frac{L(F(z_1))}{L(z_1)} \cdot \frac{1}{|\operatorname{Jac}(DF|\widehat{E}^u(z_1))|}.
\end{align*}
Since $L$ is a positive finite-valued measurable function and 
$\eta$ is an increasing partition, we have 
\begin{align*}
\frac{L(F(z_1))}{L(z_1)} &=
\frac{\int_{\eta(F z_1)} \kappa^u(F(z_1),\tau) \, \rd m_{F(z_1)}(\tau)}
{\int_{\eta(z_1)} \kappa^u(z_1,\tau) \, \rd m_{z_1}(\tau)}\\
&\leq \frac{\int_{F\eta(z_1)} \kappa^u(F(z_1),\tau) \, \rd m_{F(z_1)}(\tau)}
{\int_{\eta(z_1)} \kappa^u(z_1,\tau) \, \rd m_{z_1}(\tau)}
=|\operatorname{Jac}(DF|\widehat{E}^u(z_1))|. 
\end{align*}
This implies that $\displaystyle \log^{+}\frac{L(F(\cdot))}{L(\cdot)}$ is integrable. 
It follows that 
\begin{align}
  \int_{X\times Y} \log \nu^u_{z_1}((F^{-1}\eta)(z_1))\, \rd\zeta(z_1)
&=  \int_{X\times Y} \log q(z_1) \, \rd\zeta(z_1) \nonumber\\
&= \int_{X\times Y} \log \left( \frac{L(F(z_1))}{L(z_1)}\frac{1}{|\operatorname{Jac}(DF|\widehat{E}^u(w))|} \right) \, \rd\zeta(z_1)\nonumber\\
&=-\int_{X\times Y} \log|\operatorname{Jac}(DF|\widehat{E}^u(z_1))| \, \rd\zeta(z_1). \label{relative SRB-(2)}
\end{align}

Next, we will show that the entropy formula \eqref{relative SRB-(1)} yields 
$\nu = \zeta$ on the whole $\sigma$-algebra of Borel sets in $X\times Y$. 
This will finish the proof of this theorem. 

For $\zeta$-almost every $z\in X\times Y$, 
the partition $(F^{-1}\eta)|\eta(z)$ is countable, 
and we write 
$\displaystyle \eta(z) = \bigcup_{j \in J} A_j$ with $A_j \in F^{-1}\eta$ for each $j$. 
For $w \in \eta(z)$, we define the function 
\[
\left. \frac{\rd\nu}{\rd\zeta} \right|_{F^{-1}\eta}(w) := 
\frac{\nu^u_w((F^{-1}\eta)(w))}{\zeta^u_w((F^{-1}\eta)(w))}.
\] 
This function is well-defined since $\zeta^u_w((F^{-1}\eta)(w)) \neq 0$ 
for $\zeta$-almost every $w \in \eta(z)$. 
For $w \in \eta(z)$, we have  $\nu^u_w = \nu^u_z$ and $\zeta^u_w= \zeta^u_z$. 
Thus, we derive 
\begin{align*}
\int_{\eta(z)} \left. \frac{\rd\nu}{\rd\zeta} \right|_{F^{-1}\eta}(w) \, \rd\zeta_z(w)
&= \int_{\eta(z)} \frac{\nu^u_w((F^{-1}\eta)(w))}{\zeta^u_w((F^{-1}\eta)(w))} \, \rd\zeta^u_z(w) \\
&= \int_{\eta(z)} \frac{\nu^u_z((F^{-1}\eta)(w))}{\zeta^u_z((F^{-1}\eta)(w))} \, \rd\zeta^u_z(w) \\
&= \sum_{j:\zeta^u_z(A_j) \neq 0} \frac{\nu^u_z(A_j)}{\zeta^u_z(A_j)} \zeta^u_z(A_j) \\
&\leq \nu^u(z)(\eta(z)) = 1.
\end{align*}
Combining this with the convexity of the logarithm, we derive 
\begin{align}\label{relative SRB-(3)}
\int_{X\times Y} \log \left. \frac{\rd\nu}{\rd\zeta} \right|_{F^{-1}\eta} \, \rd\zeta 
\leq \log \int_{X\times Y} \left. \frac{\rd\nu}{\rd\zeta} \right|_{F^{-1}\eta} \, \rd\zeta \leq 0. 
\end{align}

On the other hand, by \eqref{relative SRB-(1)}, Lemma \ref{base LY}, Lemma \ref{base same dimension} 
and Ledrappier-Young entropy formula (\cite[Theorem C']{LY85-2}), we derive 
\begin{align*}
h_\zeta(F)-h_\mu(f)=\sum_{i\in\operatorname{LY}_{\operatorname{fiber}}^+}\chi_i \operatorname{dim}E_i=H_\zeta(F^{-1}\eta | \eta), 
\end{align*}
where the last equality follows from \cite[(9.2), (9.3)]{LY85-2}. 
Combining this with \eqref{relative SRB-(2)}, 
we conclude that 
\begin{align*}
-\int_{X\times Y} \log \nu^u_z((F^{-1}\eta)(z)) \, \rd\zeta(z)
&= \int_{X\times Y} \log|\operatorname{Jac}(DF|\widehat{E}^u(z))| \, \rd\zeta(z) \\
&=\sum_{i\in\operatorname{LY}_{\operatorname{fiber}}^+}\chi_i \operatorname{dim}E_i\\
&= H_\zeta(F^{-1}\eta | \eta) \\
&= -\int_{X\times Y} \log \zeta^u_z((F^{-1}\eta)(z)) \, \rd\zeta(z).
\end{align*}
Thus, we have 
$\displaystyle \int_{X\times Y} \log \left. \frac{\rd\nu}{\rd\zeta} \right|_{F^{-1}\eta} \, \rd\zeta = 0$. 
It follows from \eqref{relative SRB-(3)} that for $\zeta$-almost every $w$, 
\begin{align*}
\left. \frac{\rd\nu}{\rd\zeta} \right|_{F^{-1}\eta}(w) = 1.
\end{align*}
 Thus, $\nu = \zeta$ on the $\sigma$-algebra $\mathcal{B}_{F^{-1}\eta}$. 
By induction, $\nu = \zeta$ on $\mathcal{B}_{F^{-n}\eta}$ for each $n \in \mathbb{N}$.  
By Property (ii) of the partition $\eta$, we conclude that $\nu = \zeta$ on the Borel $\sigma$-algebra. 
Combining this with the definition of $\nu$, 
we complete the proof of the theorem. 
\end{proof}

We remark that the explicit density formula in Corollary \ref{cor:positive_density} below 
is implicitly contained in Theorem \ref{relative SRB}. 
Although in practice one often cites Theorem \ref{relative SRB} 
directly for the equivalence of conditional measures, 
we state this result explicitly here for the reader's convenience; 
see \cite[Theorem A and Corollary 6.1.4]{LY85-1} for a similar situation.

\begin{corollary}\label{cor:positive_density}
Let $\zeta$ be an ergodic measure satisfying \eqref{relative SRB-(1)}, 
let $\eta$ be the measurable partition as in the proof of Theorem \ref{relative SRB}, 
and let $\rho$ be the density of the conditional measure $\zeta_{(x,y)}^\eta$ 
with respect to $\mu_{(x,y)}^{\eta}$. 
Then at $\zeta$-a.e. $(x,y)\in X \times Y$, 
$\rho$ is a strictly positive function on $\eta(x,y)$ satisfying
\[
\frac{\rho((x,y_1))}{\rho((x,y_2))} = \kappa^u((x,y_1),(x,y_2))=
\prod_{i=1}^\infty 
\frac{
\operatorname{Jac}\bigl(DF^{-1} \mid \widehat{E}^u(F^{-i}(x,y_2))\bigr)
}{
\operatorname{Jac}\bigl(DF^{-1} \mid \widehat{E}^u(F^{-i}(x,y_1))\bigr)
}
\]
\end{corollary}

For completeness, we also record the corresponding extension to general invariant measures.
Although it is not needed in the proofs of the main theorems,
it may have applications in future research. 
The proof relies on \cite[Theorem A]{LY85-1} 
and the preceding results established earlier in this paper. 
Thus, we omit the detailed proof. 

\begin{theorem}\label{relative SRB-general}
Let $\mu$ be an $f$-invariant measure on $X$ and 
let $\zeta$ be an $F$-invariant measure on $X\times Y$ with $\pi_*\zeta=\mu$. 
Then $\zeta$ has absolutely continuous conditional measures on 
$\widehat{W}^u$ if and only if 
\begin{align}\label{relative SRB-general-(1)}
h_\zeta(F)=h_\mu(f)+\int_{X\times Y} \sum_{i\in\operatorname{LY}_{\operatorname{fiber}}^+} \chi_i(x) \operatorname{dim}E_i(x) \,\rd\zeta(x). 
\end{align}
\end{theorem}

\section{Fiberwise and product rigidity}\label{sect: Fiberwise and global measure rigidity}

In this section, we establish the fiberwise and product rigidity 
and then complete the proof of Theorem \ref{main theorem}. 

First, we extend leafwise measure rigidity in 
Theorem \ref{relative SRB} to the entire fiber.  

Let $\zeta$ be an ergodic $F$-invariant probability measure on $X\times Y$. 
Recall that $\Gamma'$ denotes the regular set of $F$ defined in \S
\ref{subsect: Fiberwise SRB and Conditional Measure Rigidity}. 
For every $(x,y)\in \Gamma'$, we denote by $\widehat W_\zeta^u(x,y)$ 
and $\widehat W_\zeta^s(x,y)$ the fiberwise unstable and stable Pesin manifolds. 
They are contained in the fiber $\{x\}\times Y$ 
and are tangent at $(x,y)$ to $\{0\}\times T_yY$. 
When $\zeta$ is fixed and the point $(x,y)$ is clear from context,
 we omit the subscripts.

\iffalse
Since $\mu=\pi_*\zeta$ is ergodic, we have 
$\displaystyle \frac{S_n\tau(x)}{n}\longrightarrow\mu(\tau) $ 
for $\zeta$-almost every $(x,y)$. 
Consequently, the vertical derivative cocycle of $F$ 
is subexponentially equivalent to that generated by $\rho_{\mu(\tau)}$. 
Hence, the fiberwise stable and unstable Lyapunov distributions, 
and consequently the corresponding Pesin foliations, 
coincide with those associated with $\rho_{\mu(\tau)}$.
\fi

By Theorem \ref{relative SRB}, every ergodic MME has absolutely continuous
conditional measures along the fiberwise stable and unstable foliations. 
To extend this leafwise absolute continuity to the fiber, 
we follow the strategy of \cite{Pe77} (see also \cite[Theorem 8.6.8]{BP-nonuniform}), 
based on successive stable and unstable saturations.

\begin{proposition}\label{fiberwise rigidity proposition}
Let $\zeta$ be an ergodic $F$-invariant measure and $\mu:=\pi_*\zeta$. 
 Assume that $\zeta$ has absolutely continuous conditional measures 
along the fiberwise stable and unstable foliations. 
If the diffeomorphism $\rho_{\mu(\tau) }$ is accessible along its stable and unstable foliations, 
then for $\mu$-almost every $x\in X$, the conditional measure $\zeta_x$ is absolutely continuous with respect to the Riemannian volume $m_Y$ on $Y$. 
\end{proposition}

\begin{proof}
Let $\xi$ be the measurable partition of $X\times Y$ into 
the fibers $\{x\}\times Y$. 
By the Rokhlin disintegration theorem, 
there exists a measurable family $\{\zeta_x\}_{x\in X}$ such that
$$
\zeta=\int_X\zeta_x\,\rd\mu(x), 
$$
where $\zeta_x$ is the conditional measure of $\zeta$ on $\{x\}\times Y$. 
By the disintegration theorem, 
there exists a measurable set $X_0\subset X$ with $\mu(X_0)=1$ such that, 
for every $x\in X_0$, the measure $\zeta_x$ is well defined 
and its conditional measures along the fiberwise stable and unstable foliations 
are absolutely continuous with respect to the corresponding leaf volumes. 

Since $\displaystyle \zeta\left(\bigcup_{l=1}^\infty \widehat{\Gamma}_{l,\epsilon}\right) = 1$ 
for a fixed small $\epsilon > 0$, the disintegration formula implies that 
\[
\zeta_x\left(\bigcup_{l=1}^\infty \widehat{\Gamma}_{l,\epsilon}\right) = 1 \quad \text{for } \mu\text{-almost every } x \in X.
\]
By replacing $X_0$ with a subset of full $\mu$-measure if necessary, 
we may assume this holds for every $x \in X_0$. 
For each $x \in X_0$, we define the set of regular points in the fiber by
\[
\mathcal R_x := \left\{ y \in Y : (x,y) \in \operatorname{supp}\zeta_x \cap \left(\bigcup_{l=1}^\infty \widehat{\Gamma}_{l,\epsilon}\right) \right\}.
\]
By construction, $\zeta_x(\mathcal R_x) = 1$. Furthermore, for every $y \in \mathcal R_x$, 
there exists an integer $l = l(x,y) \ge 1$ such that 
$(x,y) \in \widehat{\Gamma}_{l,\epsilon}$.

Fix $x \in X_0$ and choose $y\in \mathcal R_x$. 
It follows that $(x,y) \in \widehat{\Gamma}_{l_0,\epsilon}$ for some $l_0 \ge 1$.

To complete the proof, we divide the argument into three steps. 

\medskip
\noindent\textbf{Step 1: Construction and absolute continuity on the local box $P_1$.}
\medskip

Fix a sufficiently small $r>0$.
Let $\widehat{W}^u_{\operatorname{loc}}(x,y)$ and $\widehat{W}^s_{\operatorname{loc}}(x,y)$ 
denote the corresponding local leaves. 
We define the local smooth submanifold 
\begin{align}\label{def-P1}
P_1=P_1(x,y):=\bigcup\left\{\widehat{W}_{\operatorname{loc}}^s(x,y'):(x,y')\in\widehat{W}^u_{\operatorname{loc}}(x,y)\right\}.
\end{align}
Let $B_{P_1}((x,y),r):=P_1\cap B((x,y),r)$, 
where $B((x,y),r)$ denotes the ball in $X\times Y$ centered at $(x,y)$ 
with radius $r$ with respect to the Riemannian metric.

After shrinking the neighborhood if necessary, 
local submanifolds of the form \eqref{def-P1} form a measurable partition 
of a neighborhood of $(x,y)$ in $\{x\}\times Y$. 
We denote by $\zeta_{P_1}$ the conditional measure of $\zeta_x$ 
associated with the element $P_1$ of this partition.

There exists a family of local leaves $\{T(w)\}_{w\in B_{P_1}((x,y),r)}$ of $\widehat{W}^u$ satisfying
\begin{itemize}
\item[(1)] $T(w)$ is a local leaf of $\widehat{W}^u$ and depends smoothly on $w\in B_{P_1}((x,y),r)$;
\item[(2)] $T(w_1)\cap T(w_2)=\varnothing$ if $w_2\notin T(w_1)$, and $T(w_1)=T(w_2)$ if $w_2\in T(w_1)$;
\item[(3)] $\displaystyle \bigcup_{w\in B_{P_1}((x,y),r)}T(w)\supset B_{P_1}((x,y),r)$.
\end{itemize}
This family generates a measurable partition $\eta$ of the ball $B_{P_1}((x,y),r)$. 
For $w\in B_{P_1}((x,y),r)$, let $\{\widetilde{\zeta}_w\}$ be the system of conditional measures 
and let $\widehat{\zeta}$ be the factor measure induced by $\eta$. Define
$$
Q_1(x,y):=\bigcup\left\{\widehat{W}^s(w):w\in{\widehat{\Gamma}}_{l_0,\epsilon}\cap B_{P_1}((x,y),r)\right\}.
$$

Consider the measurable partition $\xi^s$ of $Q_1(x,y)$ into local leaves of $\widehat{W}^s$. 
For every $w\in B_{P_1}((x,y),r)$, the factor space $Q_1(x,y)/\xi^s$ can be identified with the subset
\begin{align*}
A(w)=\left\{z\in T(w):\text{there exists }z'\in{\widehat{\Gamma}}_{l_0,\epsilon}\cap B_{P_1}((x,y),r)\text{ such that }z=T(w)\cap\widehat{W}^s(z')\right\}.
\end{align*}
Let $m_{\widehat{W}^s(w)}$ be the leaf volume on $\widehat{W}^s$ at $w$, 
let $m_{T(w)}$ be the leaf volume on $T(w)$, 
and let $\zeta_{\widehat{W}^s(w)}$ be the conditional measure on $\widehat{W}^s(w)$ 
induced by the partition $\xi^s$ and the conditional measure $\zeta_{P_1}$.

For every $w_0\in{\widehat{\Gamma}}_{l_0,\epsilon}\cap B_{P_1}((x,y),r)$ and $z\in A(w_0)$, 
the factor space $B_{P_1}((x,y),r)/\eta$ can be identified with the local leaf $\widehat{W}^s(z)$.
For $\widehat{\zeta}$-almost every $w$, 
the conditional measure $\widetilde{\zeta}_w$ is absolutely continuous with respect to $m_{T(w)}$.
Combining this with \cite[Theorem 8.6.8]{BP-nonuniform}, 
there exists a strictly positive smooth function $g(w,z)$ such that
$$
\rd\widetilde{\zeta}_w(z)=g(w,z)\,\rd m_{T(w)}(z)
$$
for $w\in{\widehat{\Gamma}}_{l_0,\epsilon}\cap B_{P_1}((x,y),r)$ and $z\in T(w)$.

For every $w_0\in{\widehat{\Gamma}}_{l_0,\epsilon}\cap B_{P_1}((x,y),r)$ and $z\in A(w_0)$, 
applying Fubini's theorem and using the smoothness of the partition $\{T(w)\}$, 
we obtain a strictly positive smooth function $h'(z,w)$ such that
$$
\rd\widehat{\zeta}(w)=h'(z,w)\,\rd\zeta_{\widehat{W}^s(z)}(w)
$$
for $w\in\widehat{W}^s(z)$. 
Combining this with \cite[Theorem 8.6.8]{BP-nonuniform}, 
there exists a strictly positive smooth function $h(z,w)$ such that
$$
\rd\widehat{\zeta}(w)=h(z,w)\,\rd m_{\widehat{W}^s(z)}(w)
$$
for $w\in\widehat{W}^s(z)$.

Let $B \subset P_1$ be a Borel subset satisfying $\zeta_{P_1}(B) > 0$, 
and let $\chi_B(w,z)$ denote the characteristic function of $B$ at $z \in A(w)$. 
We denote by $\pi_{w_0w}^s$ the holonomy map between the transversals $T(w_0)$ and $T(w)$, 
such that $z = \pi_{w_0w}^s(v)$, and $\operatorname{Jac}(\pi_{w_0w}^s)(v)$ 
is the Jacobian of the map $\pi_{w_0w}^s$ at $v$. 
Note that the holonomy $\pi_{w_0w}^s$ is absolutely continuous and 
$\operatorname{Jac}(\pi_{w_0w}^s)$ is bounded from above and bounded away from $0$; 
see \cite[Theorem 8.6.1]{BP-nonuniform}. 
 Observing that $\pi_{w_0w}^s(A(w_0)) = A(w)$, we have
\begin{align*}
\zeta_{P_1}(B)&=\int_{B_{P_1}((x,y),r)/\eta}\int_{T(w)}\chi_B(w,z)\,\rd\widetilde{\zeta}_w(z)\,\rd\widehat{\zeta}(w)\\
&=\int_{B_{P_1}((x,y),r)/\eta}\int_{A(w)}\chi_B(w,z)g(w,z)\,\rd m_{T(w)}(z)\,\rd\widehat{\zeta}(w)\\
&=\int_{B_{P_1}((x,y),r)/\eta}\int_{A(w_0)}\chi_B(w,z)g(w,z)\operatorname{Jac}(\pi_{w_0w}^s)(v)\,\rd m_{T(w_0)}(v)\,\rd\widehat{\zeta}(w).
\end{align*}
Applying Fubini's theorem, we deduce
\begin{align*}
\zeta_{P_1}(B)&=\int_{A(w_0)}\int_{B_{P_1}((x,y),r)/\eta}\chi_B(w,z)g(w,z)\operatorname{Jac}(\pi_{w_0w}^s)(v)\,\rd\widehat{\zeta}(w)\,\rd m_{T(w_0)}(v)\\
&=\int_{A(w_0)}\int_{\widehat{W}^s(z)}\chi_B(w,z)g(w,z)\operatorname{Jac}(\pi_{w_0w}^s)(v)h(v,w)\,\rd m_{\widehat{W}^s(v)}(w)\,\rd m_{T(w_0)}(v).
\end{align*}
This shows that the conditional measure $\zeta_{P_1}$ on the box $P_1$ 
is absolutely continuous with respect to the volume on $P_1$.

If $\dim P_1=\dim Y$, then, after shrinking $P_1$ if necessary,
$P_1$ is an open subset of the fiber $\{x\}\times Y$. 
In this case, the construction of a full-dimensional box is already complete.
Thus, we assume that
$$
\dim P_1<\dim Y.
$$
Suppose first that there exists a point $w_1\in P_1$ such that 
$\widehat E^u(w_1)\not\subset T_{w_1}P_1$. 

\medskip
\noindent\textbf{Step 2: Absolute continuity on the full-dimensional box $P_N$.}
\medskip

Before constructing a full-dimensional local box 
with absolutely continuous conditional measure, 
we recall key facts regarding Pesin blocks and Rokhlin disintegrations, 
which justify our subsequent choices of points and local boxes. 

The preceding construction of $P_1$ was carried out starting from a
point $y\in\mathcal R_x$. In the remainder of the proof, whenever a
new measurable partition is introduced, 
we restrict attention to the family of elements of full factor measure 
on which the corresponding conditional measures are defined 
and assign full measure to $\mathcal R_x$. More precisely, 
let $\mathcal P$ be one of the measurable partitions appearing below, 
let $\overline\zeta_{\mathcal P}$ be the corresponding factor measure, 
and let $\{\zeta_P\}_{P\in Y/\mathcal P}$ be the associated conditional measures.  
Since $\zeta_x(\mathcal R_x)=1$, we have 
$$
1= \int_{Y/\mathcal P} \zeta_P(\mathcal R_x\cap P) 
\,\rd\overline\zeta_{\mathcal P}(P).
$$
Consequently, $\zeta_P(\mathcal R_x\cap P)=1$ for 
$\overline\zeta_{\mathcal P}$-almost every $P\in Y/\mathcal P$.

For almost every element retained in the construction, 
almost every point belongs to some Pesin block $\widehat{\Gamma}_{l,\epsilon}$. 
The local stable and unstable leaves and the associated holonomy maps 
used below are always chosen inside sufficiently small boxes. 
Therefore, the required absolute continuity properties of 
the stable and unstable holonomies hold on the local boxes under consideration. 
Moreover, the relevant conditional measures along the local stable and unstable leaves 
are absolutely continuous with respect to the associated leaf volumes.

At every stage, the points and local boxes used in the construction
are chosen from these full-measure families. 
We impose one additional requirement concerning the transport steps.

Fix an increasing countable family of Pesin blocks
$$
\widehat{\Gamma}_{l_1,\epsilon}\subset \widehat{\Gamma}_{l_2,\epsilon}\subset \cdots
\subset \widehat{\Gamma}_{l_n,\epsilon}\subset\cdots ,\quad n\in \mathbb{N} ,
$$
whose union has full $\zeta$-measure in the Pesin regular set. 
For $\gamma\in\{s,u\}$, by disintegrating along the local
$\widehat W^\gamma$-plaques and using $\zeta_x(\mathcal R_x)=1$, 
we obtain a full-measure family of local plaques $L$ such that
$$
\zeta_L(\mathcal R_x\cap L)=1.
$$
By applying \cite[Theorem 8.6.8]{BP-nonuniform} 
to the leafwise conditional measures, 
$\zeta_L$ is equivalent to the corresponding leaf volume.
Thus, $m_L(L\setminus\mathcal R_x)=0$ 
for every plaque in this full-measure family. In particular, 
every nonempty relatively open subset of 
such a plaque contains points of $\mathcal R_x$.

Whenever a box is transported along local stable or unstable leaves,
we choose the transport on one of these plaques. Thus the initial
point $w^-$ and the terminal point $w^+$ of the transport can both
be chosen in $\mathcal R_x$. Since the Pesin blocks above form an
increasing exhaustion, there exists an integer $N$ such that
$$
w^-,w^+\in\widehat{\Gamma}_{l_N,\epsilon}.
$$
After subdividing the transport and shrinking the local boxes if
necessary, the corresponding holonomy is therefore taken inside the
local Pesin boxes associated with $\widehat{\Gamma}_{l_N,\epsilon}$.

If the purpose of the transport is to reach a position at which a
stable or unstable saturation increases the dimension, the terminal
point can still be chosen in $\mathcal R_x$, since the
dimension-increasing condition
$\widehat E^\sigma(w)\not\subset T_wP_k$ is open. 
Consequently, every nonempty relatively open set 
of admissible terminal points contains a point in $\mathcal R_x$. 

Since each finite stage involves only finitely many transports, partitions, 
and disintegrations, these full-measure restrictions are imposed 
simultaneously and implicitly henceforth.  

We now address the remaining case in the construction of the full-dimensional local box. 
Suppose, on the other hand, that
$$
\widehat E^u(w)\subset T_wP_1
$$
for every $w$ in the present local box $P_1$. 
We transport $P_1$ along finitely many local stable or unstable leaves 
until a position admitting a dimension-increasing saturation is reached, 
yielding a finite sequence of smooth cube-like local submanifolds
$$
Q_0=P_1,Q_1,\ldots,Q_r
$$
such that $\dim Q_j=\dim P_1$ for every $0\leq j\leq r$.

For each $0\leq j<r$, let $\gamma_j\in\{s,u\}$ denote the type of
the leaves used in the corresponding transport. By the preceding
selection convention, we choose regular points
$$
w_j^-\in Q_j\cap\mathcal R_x,
\quad
w_j^+\in\mathcal R_x,
$$
lying on the same local $\widehat W^{\gamma_j}$-plaque. A
sufficiently small cube-like neighborhood of $w_j^-$ in $Q_j$ is
then transported along the local $\widehat W^{\gamma_j}$-leaves to
a cube-like box $Q_{j+1}$ containing $w_j^+$. Both endpoints are
chosen in a common Pesin block, as explained above.

We apply stable or unstable holonomy depending on whether 
the movement is along local stable or unstable leaves. 
Recall that all local boxes below are chosen from 
the families of full factor measure
on which the corresponding conditional measures are defined.

We claim that
$$
\zeta_{Q_j}\ll m_{Q_j}
$$
for every $0\leq j\leq r$, 
where $m_{Q_j}$ denotes the induced Riemannian volume on $Q_j$. 
The assertion holds for $Q_0=P_1$ by the preceding argument.

Suppose that $\zeta_{Q_j}\ll m_{Q_j}$ for some $0\leq j<r$. 
Let $\gamma_j\in\{s,u\}$ denote the type of the local leaves 
used to move $Q_j$ to $Q_{j+1}$.

Take a sufficiently small local $\widehat W^{\gamma_j}$-foliation
box containing the corresponding transport. 
By the preceding selection convention, we can choose regular points
$$
w_j^-\in Q_j\cap\mathcal R_x, \quad
w_j^+\in Q_{j+1}\cap\mathcal R_x,
$$
lying on the same local $\widehat W^{\gamma_j}$-leaf. 
Next, select smooth local transversals
$$
w_j^-\in S_j^-\subset Q_j, \quad
w_j^+\in S_j^+\subset Q_{j+1},
$$
both transverse to the local $\widehat W^{\gamma_j}$-leaves, 
such that every relevant local leaf meets each transversal in exactly one point. 
The corresponding holonomy map is given by
$$
H_j:S_j^-\to S_j^+, \quad H_j(z)
= \widehat W_{\operatorname{loc}}^{\gamma_j}(z)\cap S_j^+.
$$
In particular, we have $H_j(w_j^-)=w_j^+$.  
After shrinking the foliation box if necessary, 
$H_j$ is well defined and absolutely continuous, 
and its Jacobian is bounded from above and bounded away from zero.

Let $\widehat\zeta_j^-$ and $\widehat\zeta_j^+$ be the factor
measures on $S_j^-$ and $S_j^+$ obtained by disintegrating
$\zeta_{Q_j}$ and $\zeta_{Q_{j+1}}$, respectively, 
along the local $\widehat W^{\gamma_j}$-leaves. 
Since $\zeta_{Q_j}\ll m_{Q_j}$, 
the local product coordinates in the foliation box and Fubini's theorem imply that
$\widehat\zeta_j^-\ll m_{S_j^-}$. 

By the uniqueness and transitivity of Rokhlin disintegration, 
the factor measure $\widehat\zeta_j^+$ is, up to normalization, 
the pushforward of $\widehat\zeta_j^-$ under the holonomy $H_j$. 
Since $H_j$ is absolutely continuous, 
we conclude that $\widehat\zeta_j^+\ll m_{S_j^+}$. 

By assumption, the conditional measures of 
$\zeta_{Q_{j+1}}$ along the local $\widehat W^{\gamma_j}$-leaves
 are absolutely continuous with respect to the corresponding leaf volumes. 
Together with $\widehat\zeta_j^+\ll m_{S_j^+}$ 
and Fubini's theorem, this implies 
$\zeta_{Q_{j+1}}\ll m_{Q_{j+1}}$, 
thereby proving the claim by induction. 

Next, we replace $P_1$ by $Q_r$ and retain the notation $P_1$. 
For notational simplicity, we assume that 
the dimension-increasing direction is unstable 
(the stable case is analogous). 

Choose a sufficiently small cube-like neighborhood
$P_1'\subset P_1$ of $w_1$ such that
$$
d_2:=\dim\bigl(T_wP_1+\widehat E^u(w)\bigr)
$$
is constant on $P_1'$ and satisfies $d_2>\dim P_1$.
We define
\begin{align*}
P_2:=\bigcup\left\{\widehat W^u_{\operatorname{loc}}(w):w\in P_1'\right\},
\end{align*}
where the local unstable leaves are chosen sufficiently small.

After shrinking $P_1'$ and the local unstable leaves if necessary,
the constant-rank theorem implies that $P_2$ is a smooth cube-like
local submanifold of dimension $d_2$. 
Moreover, the local sets constructed in this way form a 
measurable family of boxes in a neighborhood of $P_2$. 
We denote by $\zeta_{P_2}$ the corresponding conditional measure of $\zeta_x$ on $P_2$.

Let $\eta_2$ be the measurable partition of $P_2$ into the connected
components of the local unstable leaves used in the construction of $P_2$. 
After shrinking the local box once more, 
we may choose a smooth local transversal $S_2\subset P_2$ 
meeting every element of $\eta_2$ in exactly one point. 
Let $p_2:P_2\to S_2$ be the projection along the elements of $\eta_2$, 
and let $\widehat\zeta_2:=(p_2)_*\zeta_{P_2}$ be the corresponding factor measure.

We first prove that
$$
\widehat\zeta_2\ll m_{S_2},
$$
where $m_{S_2}$ denotes the induced Riemannian volume on $S_2$.

Choose the local partitions defining the $P_1$- and $P_2$-boxes so that 
the partition into local $P_1$-type boxes refines the partition into local $P_2$-boxes. 
Thus, inside $P_2$ there is a measurable partition $\mathcal Q_2$ 
into local $P_1$-type boxes $H$ transverse to the elements of $\eta_2$.

By the transitivity of Rokhlin disintegration, 
for the factor measure of $\mathcal Q_2$-almost every $H$, 
the conditional measure $\zeta_H$ gives full measure to $\mathcal R_x\cap H$ 
and is one of the conditional measures arising in the preceding $P_1$-construction. 
The set $\mathcal R_x\cap H$ is covered, up to a $\zeta_H$-null set, 
by countably many local pieces contained in fixed Pesin blocks. 
The preceding $P_1$-argument applies on each such piece, 
with constants depending on the corresponding block. 
Hence, by countable additivity, 
we have $\zeta_H\ll m_H$ for almost every $H\in\mathcal Q_2$, 
where $m_H$ denotes the induced Riemannian volume on $H$. 

Let $A\subset S_2$ be a Borel set satisfying $m_{S_2}(A)=0$. 
For every local $P_1$-type box $H\in\mathcal Q_2$, 
let $\pi_H:S_2\to H$ be the holonomy map along the elements of $\eta_2$. Then
$$
H\cap p_2^{-1}(A)=\pi_H(A).
$$
By the absolute continuity of the corresponding unstable holonomy, 
we have 
$m_H\bigl(\pi_H(A)\bigr)=0$. 
Since $\zeta_H\ll m_H$, it follows that for almost every $H\in\mathcal Q_2$, we have 
$$
\zeta_H\bigl(H\cap p_2^{-1}(A)\bigr)=\zeta_H\bigl(\pi_H(A)\bigr)=0. 
$$
Integrating the conditional measures over the factor space of $\mathcal Q_2$, 
we obtain
$$
\widehat\zeta_2(A)=\zeta_{P_2}\bigl(p_2^{-1}(A)\bigr)=0.
$$
Hence, we have $\widehat\zeta_2\ll m_{S_2}$. 
We disintegrate $\zeta_{P_2}$ with respect to $\eta_2$:
$$
\zeta_{P_2}=\int_{S_2}\widetilde\zeta^{\,2}_v\,\rd\widehat\zeta_2(v),
$$
where $\widetilde\zeta^{\,2}_v$ is supported on the local unstable leaf through $v$. 
By assumption and \cite[Theorem 8.6.8]{BP-nonuniform}, for $\widehat\zeta_2$-almost every $v\in S_2$, 
the conditional measure $\widetilde\zeta^{\,2}_v$ is absolutely continuous 
with respect to the corresponding unstable leaf volume.
Combining this with $\widehat\zeta_2\ll m_{S_2}$,
and applying the same change-of-variables and Fubini argument as in
the proof for $P_1$, we obtain 
$$
\zeta_{P_2}\ll m_{P_2},
$$
where $m_{P_2}$ denotes the induced Riemannian volume on $P_2$.

We repeat the preceding construction inductively. Suppose that,
for some $k\geq2$, a smooth cube-like local submanifold $P_k$ has
been constructed and satisfies $\zeta_{P_k}\ll m_{P_k}$.
If $\dim P_k=\dim Y$, the construction terminates. 
Thus, we assume that
$$
\dim P_k<\dim Y.
$$
Suppose first that there exist a point $w_k\in P_k$ and a direction
$\sigma_k\in\{s,u\}$ such that
$$
\widehat E^{\sigma_k}(w_k)\not\subset T_{w_k}P_k.
$$
In this case, no preliminary transport of $P_k$ is needed. 
Otherwise, repeating the finite holonomy-transport argument used for $P_1$, 
we obtain a smooth cube-like box $\widetilde P_k$ such that 
$$
\dim \widetilde P_k=\dim P_k,
\quad
\zeta_{\widetilde P_k}\ll m_{\widetilde P_k},
$$
and there exist a point $w_k\in\widetilde P_k$ and a direction
$\sigma_k\in\{s,u\}$ satisfying
$$
\widehat E^{\sigma_k}(w_k) \not\subset T_{w_k}\widetilde P_k.
$$
Indeed, if no finite sequence of local stable and unstable transports
reached a dimension-increasing direction, then every smooth endpoint
map associated with a finite local $su$-word starting from $P_k$
would have rank at most $\dim P_k$. Hence, the points reachable from
$P_k$ by finite local $su$-paths could not fill the fiber $Y$, 
contradicting accessibility.
For convenience, we replace $P_k$ with $\widetilde P_k$ and retain the notation $P_k$. 

Choose a sufficiently small cube-like neighborhood $P_k'\subset P_k$ of $w_k$ such that
$$
d_{k+1}
:=
\dim\bigl(
T_wP_k+\widehat E^{\sigma_k}(w)
\bigr)
$$
is constant on $P_k'$ and satisfies $d_{k+1}>\dim P_k$. 
Define
$$
P_{k+1}
:=
\bigcup_{w\in P_k'}
\widehat W^{\sigma_k}_{\operatorname{loc}}(w),
$$
where the local leaves are chosen sufficiently small.

After shrinking $P_k'$ and the local leaves if necessary, 
the constant-rank theorem implies that $P_{k+1}$ is a smooth cube-like
local submanifold of dimension $d_{k+1}$. 
Repeating the argument used for the passage from $P_1$ to $P_2$, 
with the stable and unstable directions interchanged when necessary, 
we obtain $\zeta_{P_{k+1}}\ll m_{P_{k+1}}$.

Since the dimension strictly increases at every such step, 
the construction terminates after finitely many iterations. 
We therefore obtain smooth cube-like local submanifolds
$$
P_1,P_2,\ldots,P_N\subset\{x\}\times Y
$$
such that $\dim P_1<\dim P_2<\cdots<\dim P_N=\dim Y$ 
and $\zeta_{P_j}\ll m_{P_j}$ for every $1\leq j\leq N$.

Since $P_N$ is a full-dimensional local submanifold of the fiber,
after shrinking it if necessary, $P_N$ is an open subset of $\{x\}\times Y$. 
By the full-measure choices made above, we may take 
$P_N$ so that its Rokhlin conditional measure is defined and $\zeta_x(P_N)>0$. 
In the full-dimensional case, $\zeta_{P_N}$ coincides, up to normalization, 
with the restriction of $\zeta_x$ to $P_N$. It follows that 
$$
\zeta_x|_{P_N}\ll m_Y|_{P_N}.
$$

\medskip
\noindent\textbf{Step 3: Extension of absolute continuity across the entire fiber via accessibility.}
\medskip

Since $\zeta_x(P_N)>0$ and $\zeta_x(\mathcal R_x)=1$,
we choose $a\in P_N\cap\mathcal R_x$. 
Fix $y'\in\mathcal R_x$. By accessibility, there exists a finite
$su$-path
$$
a=z_0,z_1,\ldots,z_q=y',
$$
such that for each $0\le j<q$, 
the points $z_j$ and $z_{j+1}$ lie on the same local stable or unstable leaf. 
By the preceding selection convention, each point $z_j$ can be chosen in $\mathcal R_x$. 
By subdividing the path segments if necessary, 
we may assume that each segment is contained in a sufficiently small local foliation box.

We prove inductively that, for every $0\leq j\leq q$, 
there exists an open neighborhood $U_j$ of $z_j$ such that
$$
\zeta_x|_{U_j}\ll m_Y|_{U_j}.
$$
For $j=0$, we take $U_0:=P_N$.

Suppose given an open neighborhood $U_j$ of $z_j$ satisfying 
$\zeta_x|_{U_j}\ll m_Y|_{U_j}$. 
Let $\gamma_j\in\{s,u\}$ denote the type of the leaf containing
$z_j$ and $z_{j+1}$. Choose a sufficiently small local
$\widehat W^{\gamma_j}$-foliation box $\mathcal B_j$ containing the
corresponding leaf segment. 

Inside $\mathcal B_j$, 
we choose smooth local transversals $S_j^-\subset U_j$ and $S_j^+$, 
both transverse to the local $\widehat W^{\gamma_j}$-leaves,
and basepoints $z_j\in S_j^-$, $z_{j+1}\in S_j^+$ 
such that every relevant local leaf meets each transversal in exactly one point. 
Let $H_j:S_j^-\to S_j^+ $ be the corresponding local holonomy map. 
In particular, we have $H_j(z_j)=z_{j+1}$.

Choose small open local product boxes $V_j^-$ and $V_j^+$ over
$S_j^-$ and $S_j^+$, respectively, such that
$$
V_j^-\subset U_j, \quad
z_j\in V_j^-, \quad
z_{j+1}\in V_j^+.
$$
The boxes $V_j^-$ and $V_j^+$ are chosen over the same family of local plaques. 
Thus, every plaque under consideration meets both $V_j^-$ and $V_j^+$ 
in a nonempty relatively open subset. Set
$$
\nu_j^-:=\zeta_x|_{V_j^-}, \quad \nu_j^+:=\zeta_x|_{V_j^+},
$$
and denote by $\widehat\nu_j^-$ and $\widehat\nu_j^+$ the factor measures 
on $S_j^-$ and $S_j^+$ obtained by disintegrating 
$\nu_j^-$ and $\nu_j^+$ along the local $\widehat W^{\gamma_j}$-plaques.

% We first show that $\widehat\nu_j^-\ll m_{S_j^-}$. 
Let $p_j^-:V_j^-\to S_j^-$ be the projection along the local
$\widehat W^{\gamma_j}$-plaques. By definition, we have 
$\widehat\nu_j^-=(p_j^-)_*\nu_j^-$. 
Let $\widehat m_j^-:=(p_j^-)_*\bigl(m_Y|_{V_j^-}\bigr)$ 
be the factor measure induced by the Riemannian volume. 
The absolute continuity of the local $\widehat W^{\gamma_j}$-foliation 
implies that $\widehat m_j^-\ll m_{S_j^-}$. 
Since $\nu_j^-\ll m_Y\vert{}_{V_j^-}$, 
taking pushforwards under $p_j^-$ yields  
$$
\widehat\nu_j^-=(p_j^-)_*\nu_j^- \ll (p_j^-)_*\bigl(m_Y|_{V_j^-}\bigr)
=\widehat m_j^- \ll m_{S_j^-}.
$$

We next compare the factor measures on the two transversals. 
Disintegrate the restriction of $\zeta_x$ to the common foliation
box $\mathcal B_j$ along its local
$\widehat W^{\gamma_j}$-plaques:
$$
\zeta_x|_{\mathcal B_j} 
=\int_{S_j^-}\zeta_{j,t}^{\gamma_j} \,\rd\overline\zeta_j(t),
$$
where $\zeta_{j,t}^{\gamma_j}$ is supported on the local plaque
$L_t$ through $t\in S_j^-$. Let
$$
I_t^-:=L_t\cap V_j^-, \quad I_t^+:=L_t\cap V_j^+.
$$
Under the identification of the quotient of $\mathcal B_j$ with
$S_j^-$, we define
$$
c_j^-(t):=\zeta_{j,t}^{\gamma_j}(I_t^-), \quad
c_j^+(t):=\zeta_{j,t}^{\gamma_j}(I_t^+).
$$
The transitivity of Rokhlin disintegration gives
\begin{align}\label{fiberwise rigidity lemma-(1)}
\rd\widehat\nu_j^-(t) =c_j^-(t)\,\rd\overline\zeta_j(t),
\quad \text{and} \quad
\rd\bigl((H_j^{-1})_*\widehat\nu_j^+\bigr)(t)=c_j^+(t)\,\rd\overline\zeta_j(t).
\end{align}

By the countable Pesin-block exhaustion and \cite[Theorem 8.6.8]{BP-nonuniform}, 
there exists a measurable set $D_j\subset S_j^-$ of full $\overline\zeta_j$-measure 
such that $D_j$ is the union of countably many measurable sets $D_{j,n}$, 
where for every $t\in D_j$, the conditional measure $\zeta_{j,t}^{\gamma_j}$ 
is equivalent to the leaf volume on the corresponding local plaque $L_t$; 
more precisely,
$$
\rd\zeta_{j,t}^{\gamma_j}
=
\rho_{j,t}^{\gamma_j}\,\rd m_{L_t},
\quad
\rho_{j,t}^{\gamma_j}>0
\quad
m_{L_t}\text{-almost everywhere}.
$$ 
Moreover, for every $n$, the set $D_{j,n}$ is contained in the
transverse domain of a fixed local Pesin box. On the corresponding
transversal pieces, the holonomy $H_j$ and its inverse are absolutely
continuous, and the holonomy Jacobian is bounded from above and
bounded away from zero. Consequently, for every Borel set
$E\subset S_j^+$ satisfying
$m_{S_j^+}(E)=0$, 
one has
$$
m_{S_j^-}\bigl(D_{j,n}\cap H_j^{-1}(E)\bigr)=0.
$$
Since $I_t^-$ and $I_t^+$ are nonempty relatively open subsets of
$L_t$, they have positive leaf volume. 
Hence, for $\overline\zeta_j$-almost every $t$, we have  
$c_j^-(t)>0 $ and $c_j^+(t)>0$. 
It follows from \eqref{fiberwise rigidity lemma-(1)} that
$(H_j^{-1})_*\widehat\nu_j^+\sim \widehat\nu_j^-$,
or equivalently, $\widehat\nu_j^+ \sim (H_j)_*\widehat\nu_j^-$.

Let $E\subset S_j^+$ be a Borel set satisfying $m_{S_j^+}(E)=0$. 
For every $n$, the absolute continuity of the restriction of $H_j$
to $D_{j,n}$ gives
$$
m_{S_j^-}\bigl(D_{j,n}\cap H_j^{-1}(E)\bigr)=0.
$$
Since $\widehat\nu_j^-\ll m_{S_j^-}$, for every $n$, 
we have $\widehat\nu_j^-\bigl(D_{j,n}\cap H_j^{-1}(E)\bigr)=0$. 
Combining this with the fact that $D_j$ has full $\widehat\nu_j^-$-measure, 
we obtain 
\begin{align*}
(H_j)_*\widehat\nu_j^-(E)
=\widehat\nu_j^-\bigl(H_j^{-1}(E)\bigr)
=\widehat\nu_j^-\bigl(D_j\cap H_j^{-1}(E)\bigr)
\leq \sum_{n\geq1} \widehat\nu_j^- \bigl(D_{j,n}\cap H_j^{-1}(E)\bigr)
=0.
\end{align*}
This proves $(H_j)_*\widehat\nu_j^-\ll m_{S_j^+}$. 
Since $\widehat\nu_j^+ \sim (H_j)_*\widehat\nu_j^-$,
we conclude that $\widehat\nu_j^+\ll m_{S_j^+}$. 

On the other hand, 
the conditional measures of $\nu_j^+$ along the local
$\widehat W^{\gamma_j}$-plaques are restrictions and normalizations
of the conditional measures $\zeta_{j,t}^{\gamma_j}$. 
Hence, by \cite[Theorem 8.6.8]{BP-nonuniform}, 
they are absolutely continuous with respect to 
the corresponding leaf volumes.

Disintegrating the Riemannian volume $m_Y|_{V_j^+}$ along the same local plaques, 
the absolute continuity of the local foliation implies that 
its factor measure is equivalent to $m_{S_j^+}$ and that 
its conditional measures are equivalent to the corresponding leaf volumes. 
Since $\widehat\nu_j^+\ll m_{S_j^+}$
and the conditional measures of $\nu_j^+$ along the plaques are
absolutely continuous with respect to the corresponding leaf volumes, 
the transitivity of Rokhlin disintegration and Fubini's theorem imply that
$$
\nu_j^+\ll m_Y|_{V_j^+}.
$$
Equivalently, $\zeta_x|_{V_j^+}\ll m_Y|_{V_j^+}$.
Thus, we can take $U_{j+1}:=V_j^+$.

Repeating this argument along the finitely many legs of the $su$-path, 
for each $y'\in\mathcal R_x$, 
we obtain an open neighborhood $U(y')$ of $y'$ such that
$\zeta_x|_{U(y')}\ll m_Y|_{U(y')}$. 
Since $Y$ is second-countable, 
there exists a countable collection 
$\{y_n\}_{n\in\mathbb N} \subset \mathcal R_x$ such that 
$$
\mathcal R_x \subset \bigcup_{n\in\mathbb N}U(y_n).
$$
Let $A\subset Y$ be a Borel set satisfying $m_Y(A)=0$.
For every $n\in\mathbb N$, we have $m_Y\bigl(A\cap U(y_n)\bigr)=0$.
Since $\zeta_x|_{U(y_n)} \ll m_Y|_{U(y_n)}$,
it follows that $\zeta_x\bigl(A\cap U(y_n)\bigr)=0$ for all $n\in\mathbb N$. 
Together with the fact that $\zeta_x(\mathcal R_x)=1$, this yields 
\begin{align*}
\zeta_x(A)
=\zeta_x(A\cap\mathcal R_x)
\leq\sum_{n\in\mathbb N}\zeta_x\bigl(A\cap U(y_n)\bigr)=0.
\end{align*}
Therefore, $\zeta_x$ is absolutely continuous with respect to $m_Y$.

Since $x\in X_0$ was arbitrary, the conclusion holds for
$\mu$-almost every $x\in X$.
\end{proof}

Next, we extend leafwise measure rigidity to the entire fiber. 
Combined with the finiteness of the projections of MMEs, 
this yields the product rigidity in Theorem \ref{main theorem}. 

\begin{proposition}\label{unique measure}
Each ergodic measure of maximal entropy $\zeta$ for $F$ is of the form
$$
\zeta=\mu\times m_{\MME},
\quad
\mu=\pi_*\zeta.
$$
\end{proposition}

\begin{proof}
Fix an ergodic measure of maximal entropy $\zeta$ and set $\mu = \pi_*\zeta$. 
By Corollary \ref{projection not MME} and condition (C2), 
the diffeomorphism $\rho_{\mu(\tau)}$ is accessible along its stable and unstable foliations. 
Combining this with Proposition \ref{fiberwise rigidity proposition}, 
we conclude that for $\mu$-almost every $x \in X$, 
the conditional measure $\zeta_x$ is absolutely continuous with respect to the volume measure on $Y$. 
Since $m_{\mathrm{MME}}$ is a smooth measure, we obtain 
$\zeta_x \ll m_{\mathrm{MME}}$ for $\mu$-almost every $x \in X$. 
Consequently, $
\zeta \ll \mu \times m_{\mathrm{MME}}.
$
By \cite[Theorem 3]{GH88} and condition (C1), $\tau$ satisfies the mixing LLT. 
Combining this with \cite[Theorem 2.5 (1)]{DDKN22-1}, 
we deduce that the measure $\mu \times m_{\mathrm{MME}}$ is ergodic and $F$-invariant. 
Since $\zeta$ is also $F$-invariant and absolutely continuous with respect to $\mu \times m_{\mathrm{MME}}$, the ergodicity implies that
$
\zeta = \mu \times m_{\mathrm{MME}}.
$
\end{proof}

We can now complete the proof of Theorem \ref{main theorem}.

\begin{proof}[Proof of Theorem \ref{main theorem}]
It is a direct consequence of 
Theorem \ref{goal theorem} and Proposition \ref{unique measure}. 
\end{proof}

\section{Homogeneous actions, bounds and examples}\label{sect:homogeneous-bounds}

Finally, we examine the homogeneous setting and 
investigate the number of ergodic MMEs 
and the role of the joint nontriviality assumptions. 

\subsection{Algebraic actions and root foliations}\label{subsect: Algebraic actions and root foliations}

In this subsection, we show how the preceding rigidity argument applies
when the $\mathbb R^k$-action on $Y=G/\Gamma$ is induced by 
a maximal Cartan subgroup of a simple, split, connected Lie group $G$. 
This illustrates the algebraic mechanism underlying
Theorem \ref{main theorem-algebra}, which follows from
Theorem \ref{main theorem} and Corollary \ref{main theorem-mixing}. 

We recall the background on foliations for split Lie groups, 
following the presentation in \cite{EK03}. 
For a comprehensive treatment of Lie groups, 
we refer the reader to \cite{Knapp02}. 

Let $G$ be a simple, split, connected Lie group of rank greater than $1$,
with Lie algebra $\mathfrak g$. 
Let $\mathfrak h \subset \mathfrak g$ be a Cartan subalgebra, 
and let $\Lambda \subset \mathfrak h^*$ be its root system, 
so that $\mathfrak g$ admits the root space decomposition
$$
\mathfrak g
=
\mathfrak h
\oplus
\bigoplus_{\lambda\in\Lambda}\mathfrak g_\lambda,
$$
where $\mathfrak h$ is a Cartan subalgebra and each
$\mathfrak g_\lambda$ is one-dimensional. 
The bracket relations
$$
[\mathfrak g_\lambda,\mathfrak g_{\lambda'}]
\subset
\mathfrak g_{\lambda+\lambda'}
$$
show that, whenever $\lambda+\lambda'\in\Lambda$, the root directions
$\mathfrak g_\lambda$ and $\mathfrak g_{\lambda'}$ generate the new
direction $\mathfrak g_{\lambda+\lambda'}$. Thus, successive Lie brackets
provide precisely the mechanism for generating the additional fiber
directions required in condition (C2). In particular, when the iterated
brackets generate the whole fiber tangent algebra, they also provide the
corresponding accessibility mechanism.

In this algebraic setting, the same box-construction argument of 
Proposition \ref{fiberwise rigidity proposition} applies directly
to individual root subfoliations. 
Absolute continuity thus propagates to the foliation they generate. 
Since the bracket-generated direction is already available at the current box, 
this obviates the need for preliminary holonomy transport. 

In the following two examples, let $Y = G/\Gamma$ be a homogeneous space, 
and let $\mu$ be a probability measure on $Y$ that is ergodic 
and invariant with respect to the action of a fixed nontrivial Cartan element. 
We consider two distinct cases: 
one involving a stable and an unstable foliation, 
and the other involving two stable foliations.

\begin{example}[$\operatorname{SL}(2,\mathbb R)$]
Let $Y=\operatorname{SL}(2,\mathbb R)/\Gamma$. 
For a fixed nontrivial Cartan element, 
the root directions $E_{12}$ and $E_{21}$ have opposite Lyapunov exponents, 
so one of the corresponding root foliations $\widehat W_{1,2}$ and $\widehat W_{2,1}$ 
is stable and the other is unstable. Since
$$ 
 [E_{12},E_{21}] = E_{11}-E_{22}, \quad 
 \operatorname{Lie}\langle E_{12},E_{21}\rangle = \mathfrak{sl}(2,\mathbb R),  
 $$
these two root foliations generate the whole space $Y$. 
Therefore, the absolute continuity of the conditional measures of $\mu$ 
along $\widehat W_{1,2}$ and $\widehat W_{2,1}$ implies that 
$\mu$ is absolutely continuous with respect to the Haar measure $m_Y$. 
\end{example}

\begin{example}[The Heisenberg group $\mathbb H$]
Let $Y = \operatorname{SL}(3,\mathbb R)/\Gamma$. 
Consider the $3$-dimensional Heisenberg subgroup $\mathbb{H} \subset \operatorname{SL}(3,\mathbb R)$ 
defined by the upper-triangular unipotent matrices,
$$
\mathbb H=
\left\{
\begin{pmatrix}
1&x&z\\
0&1&y\\
0&0&1
\end{pmatrix}
:x,y,z\in\mathbb R \right\}, 
$$
whose Lie algebra is $\mathfrak{h} = \operatorname{span}\{E_{12}, E_{23}, E_{13}\}$. 
For a fixed Cartan element $A = \operatorname{diag}(e^{t_1},e^{t_2},e^{t_3})\in \operatorname{SL}(3,\mathbb R)$ 
with $t_1 < t_2 < t_3$, 
the root directions $E_{12}$ and $E_{23}$ both have negative Lyapunov exponents, 
so the corresponding root foliations $\widehat W_{1,2}$ and $\widehat W_{2,3}$ are both stable. Since
$$
[E_{12},E_{23}] = E_{13}, \quad \operatorname{Lie}\langle E_{12},E_{23}\rangle = \mathfrak{h},
$$
these two root foliations generate the $3$-dimensional subgroup foliation $\widehat W_{\mathbb{H}}$, 
whose leaves are the ${\mathbb{H}}$-orbits ${\mathbb{H}} \cdot x$. 
Therefore, by a box-construction argument, 
the absolute continuity of the conditional measures of $\mu$ along 
$\widehat W_{1,2}$ and $\widehat W_{2,3}$ implies that 
$\mu_x^H{\mathbb{H}}\ll m_x^{\mathbb{H}}$ for $\mu$-almost every $x \in Y$, 
where $m_x^{\mathbb{H}}$ denotes the Haar measure on the leaf ${\mathbb{H}} \cdot x$. 
\end{example}

\subsection{On the number of ergodic MMEs}
\label{subsect: On the number of ergodic MMEs}

This subsection studies the number of ergodic MMEs
and the role of the joint nontriviality condition.
We first consider the case $k=1$, where the upper
bound of two is attained. We then explain how the
general upper bound is determined by the number of
Weyl chambers, and finally discuss possible behaviors
when condition (C1) or (C2) fails. 

For $k=1$, by analogy with Theorem \ref{higher le:fiber-entropy}, 
the corresponding results for the $\mathbb{R}^1$-action $\{\rho_{\bf t}\}$ 
can be derived similarly; we thus omit the detailed proofs.

\begin{lemma}\label{skew product variational principle}
Given an ergodic $f$-invariant measure $\mu$, we have 
$$
\sup\left\{ h_\zeta(F):\zeta\in \mathcal{M}(F,X\times Y) \text{ and } \pi_*\zeta
=\mu\right\}=\cH(\mu):=h_\mu(f)+|\mu(\tau)|h_\mathrm{top}(\rho_1). 
$$
Moreover, we have 
\begin{align*}
    h_{\top}(F)=\sup \left\{ h_\mu(f)+|\mu(\tau)|h_\mathrm{top}(\phi_1):\mu\in\mathcal{M}_e(f,X) \right\}. 
\end{align*}
\end{lemma}

\iffalse
\begin{proof}
By Birkhoff ergodic theorem, for $\mu$-a.e. $x\in X$, 
\begin{align*}
\lim_{n\to\infty} \frac{1}{n}\sum_{i=0}^{n-1} \tau(f^ix)=\int_X\tau(x)\,\rd\mu(x) =\mu(\tau). 
\end{align*}
Combining this with Lemma \ref{sequence entropy}, 
we conclude that for $\mu$-a.e. $x\in X$, 
$$ 
h_{\top}(F, \pi^{-1}(x))=| \mu(\tau)| {\bf a}. 
$$
Note that the continuous surjective map $\pi: X\times Y\to X$ satisfies $\pi\circ F=f\circ \pi$, 
and by Theorem \ref{LW77-Theorem 2.1}, we have 
\begin{align*}
\sup\left\{ h_\zeta(F) : \zeta\in \mathcal{M}(F,X\times Y) \text{ and } \pi_*\zeta=\mu\right\}
&= h_\mu(f) + \int_X h_{\top}(F, \pi^{-1}(x))  \,\rd\mu(x)\\
&=h_\mu(f) + \int_X |\mu(\tau)|{\bf a} \,\rd \mu(x)\\
&= h_\mu(f) + |\mu(\tau)|{\bf a}. 
\end{align*}
This completes the proof of the lemma. 
\end{proof}

Combining Lemma \ref{le:fiber-entropy} and the variational principle, 
we derive the following corollary. 

\begin{corollary}{skew product variational principle}
\begin{align*}
    h_{\top}(F)=\sup \{ h_\mu(f)+|\mu(\tau)|{\bf a}:\mu\in\mathcal{M}_e(f,X) \}. 
\end{align*}
\end{corollary}
\fi

Let $\mu^\pm$ be the unique equilibrium measure for the potentials $\pm\tau(x) h_{\top}(\rho_1)$. 
By Lemma \ref{skew product variational principle}, 
any ergodic measures of maximal entropy $\zeta$ of $F$ satisfies 
$\pi_*(\zeta)=\mu^{+}$ or $\pi_*(\zeta)=\mu^{-}$.
Thus, the $F$ has at most two ergodic measures of maximal entropy. 
In fact, $F$ may admit two such measures, 
as illustrated in the following example. 

\begin{example}\label{nonunique example}

Let $f$ be a topologically mixing Anosov diffeomorphism on a closed manifold $X$, 
and let $\{\rho_{\bf t}\}_{{\bf t}\in \mathbb{R}}$ be a flow on 
a compact homogeneous space $Y$ with $h_{\top}(\rho_1)=1$ preserving 
an ergodic smooth measure of maximal entropy 
and satisfying the accessibility condition. 
Given a $C^{1+\beta}$ function $\tau: X \to\R$ that is not cohomologous to a constant, 
we set 
$$
c=\frac12\bigl(P_{\top}(f,-\tau)-P_{\top}(f,\tau)\bigr),\quad 
\text{ and } \quad \varphi(x):=\tau(x)+c.
$$
Then $\varphi$ is also not cohomologous to a constant.
Consider the skew product
\[
F(x,y)=\bigl(f(x),\rho_{\varphi(x)}y\bigr). 
\]
Clearly, the joint nontriviality condition in Theorem \ref{main theorem} is satisfied. Then, we have 
\begin{align*}
h_{\mu^+}(f)+\int\varphi\,\rd\mu^+
&=P_{\top}(f,\varphi)=P_{\top}(f,\tau)+c
=P_{\top}(f,-\tau)-c\\
&=P_{\top}(f,-\varphi)
=h_{\mu^-}(f)-\int\varphi\,\rd\mu^-.
\end{align*}
By Lemma \ref{skew product variational principle},
this common value equals $h_{\top}(F)$.
Because $\varphi$ is not cohomologous to a constant, 
\cite[Proposition 20.3.10]{KH95} ensures that $\mu^+\neq\mu^-$. 
Thus, by Theorem \ref{main theorem}, 
 $F$ admits exactly two distinct ergodic measures of maximal entropy, namely 
 $\mu^+\times m_{\operatorname{Haar}}$ and $\mu^-\times m_{\operatorname{Haar}}$. 
\end{example}

While the above example might suggest that the upper bound is independent of 
the dimension of the manifold $Y$, we will show that they are, in fact, 
closely related. Specifically, the upper bound is determined by 
an interplay between the rank $k$ and $\dim Y$; however, 
this dependence is obscured in the degenerate case where $k=1$. 
To clarify this, we next introduce the Weyl chambers of the group action, 
which not only determine this upper bound but also 
provide its underlying intuition, 
as previously noted in Remark \ref{main theorem-remark 1}.

For the higher dimensional case, 
we first introduce the Lyapunov hyperplanes and Weyl chambers associated with the action. 
By applying Proposition \ref{chart1} to $\{\rho_{\bf t}\}$ with respect to the ergodic measure $m_{\MME}$, 
each nonzero Lyapunov exponent of the $\mathbb{R}^k$-action can be identified 
with a nonzero linear functional $\chi_i$ ($1\leq i\leq L$). 
The kernels $\ker \chi_i$ are called {\bf Lyapunov hyperplanes}, 
and the connected components of 
$\displaystyle \mathbb{R}^k \setminus \bigcup_{\chi_i\neq 0} \ker \chi_i$
are defined as the {\bf Weyl chambers} of the action $\rho_{\bf t}$. 

  {\bf Explanation of the upper bound:}
By combining the definition of Weyl chambers with 
the argument in the proof of Theorem \ref{goal theorem}, 
each region $\mathcal{R}$ is uniquely determined 
by the signs of the Lyapunov functionals, thereby directly corresponding to 
a Weyl chamber. Consequently, there is a one-to-one correspondence 
between these regions and the Weyl chambers. 
Furthermore, as illustrated in Example \ref{nonunique example}, 
the corresponding H\"older potentials vary across different regions. 
Since each region admits at most one equilibrium state 
(which is the projection of the ergodic MME of the skew product), 
combining this with Proposition \ref{unique measure} ensures that 
each Weyl chamber corresponds to at most one ergodic MME. 
As investigated in Example \ref{nonunique example}, 
distinct potentials give rise to distinct ergodic MMEs for the skew product, 
thereby providing a natural explanation for 
why the number of ergodic MMEs is bounded above by 
the number of Weyl chambers of the group action.

It is worth mentioning that 
the number of Weyl chambers depends not only on the rank $k$ of the action 
but also on the number of Lyapunov functionals, 
and thus ultimately depends on the dimension of the manifold.

\begin{lemma}\label{lem:hyperplane-regions}
Let $H_1,\ldots,H_L$ be hyperplanes in $\mathbb R^k$ passing through the origin,
where $L\geq 1$ and $k\geq 1$. Then the number of connected components of
$\displaystyle \mathbb R^k\setminus \bigcup_{i=1}^L H_i$ is at most 
\[
2\sum_{j=0}^{k-1}\binom{L-1}{j}.
\]
Moreover, this upper bound is attained when the hyperplanes are in general
position.
\end{lemma}

\begin{proof}
Let $R(L,k)$ denote the maximum number of connected components into which
$L$ hyperplanes through the origin can divide $\mathbb R^k$.

We first derive an upper bound recurrence. Suppose that
$H_1,\ldots,H_{L-1}$ have already been fixed, and then the last
hyperplane $H_L$ is added. Any connected component of
$\displaystyle \mathbb R^k\setminus \bigcup_{i=1}^{L-1}H_i$ 
is either untouched by $H_L$ or divided into two subcomponents by $H_L$.
Hence, the number of new components created by adding $H_L$ is at most the
number of connected components into which the hyperplanes
$H_i\cap H_L$, $i=1,\ldots,L-1$,divide $H_L$.

Since $H_L$ is naturally isomorphic to $\mathbb R^{k-1}$, and since each
$H_i\cap H_L$ is either a hyperplane in $H_L$ or equal to $H_L$ itself, the induced
arrangement on $H_L$ consists of at most $L-1$ central hyperplanes in
$\mathbb R^{k-1}$. Therefore,
\begin{align}\label{lem:hyperplane-regions-(1)}
R(L,k)\leq R(L-1,k)+R(L-1,k-1).
\end{align}

We now prove by induction that
\[
R(L,k)\leq B(L,k),
\quad \text{where }
B(L,k):=
2\sum_{j=0}^{k-1}\binom{L-1}{j}.
\]
Clearly, for all $L\geq 1$ and $k\geq 1$, we have $R(1,k)=2$ and $R(L,1)=2$. 
The base cases are immediate since
\[
B(1,k)=2\sum_{j=0}^{k-1}\binom{0}{j}=2, \quad 
\text{ and }\quad 
B(L,1)=2\binom{L-1}{0}=2.
\]
Assume that the estimate holds for $(L-1,k)$ and $(L-1,k-1)$. 
Then, by \eqref{lem:hyperplane-regions-(1)}, we deduce 
\[
\begin{aligned}
R(L,k)&\leq R(L-1,k)+R(L-1,k-1)\leq B(L-1,k)+B(L-1,k-1)\\
&=2\sum_{j=0}^{k-1}\binom{L-2}{j}+2\sum_{j=0}^{k-2}\binom{L-2}{j}
=2\sum_{j=0}^{k-1}\left(\binom{L-2}{j}+\binom{L-2}{j-1}\right)\\
&=2\sum_{j=0}^{k-1}\binom{L-1}{j}
=B(L,k).
\end{aligned}
\]

Finally, if the hyperplanes $H_1,\ldots,H_L$ are in general position, then the induced hyperplanes
$H_i\cap H_L$ for $i=1,\ldots,L-1$ are also in general position inside
$H_L\cong\mathbb R^{k-1}$. Hence, at each step, the number of new regions created
by adding $H_L$ is precisely the maximum number of regions cut out by $L-1$
central hyperplanes in $\mathbb R^{k-1}$. Thus, equality is achieved in the recurrence,
and the stated upper bound is sharp.
\end{proof}

Let $\rho_{\bf t}$ be a locally free $C^{1+\beta}$ action of $\mathbb{R}^k$ 
on a manifold $Y$ preserving a smooth ergodic invariant measure $\mu$. 
Since the action is locally free, the orbit distribution has the dimension $k$.
Hence, the number of distinct nonzero Lyapunov functionals satisfies
$L\le \dim Y-k$. 
Therefore, the number of Weyl chambers is bounded above by 
\[
2\sum_{j=0}^{k-1}\binom{\dim Y-k-1}{j}\leq 2^{\dim Y-k}
\]
with the convention that $\binom{a}{j}=0$ for $j>a$.

Next, we apply the upper bound to the following two examples.

For a $C^{1+\beta}$ flow (a locally free $\mathbb{R}$-action) on a closed manifold, 
regardless of the dimension of the manifold, 
the number of Weyl chambers is at most two. 
Thus, the corresponding skew product has at most two ergodic MMEs, 
and this bound is sharp, as shown in  Example \ref{nonunique example}.

  As a second example, 
consider a maximal rank $\mathbb{R}^k$-action on a $(2k+1)$-dimensional manifold. 
If the corresponding Lyapunov hyperplanes are in general position, 
the number of Weyl chambers is bounded by 
\[
2\sum_{j=0}^{k-1}\binom{k}{j}=2^{k+1}-2.
\]
Note that in this case, all sign combinations except for the 
all-positive and all-negative ones appear in certain Weyl chambers. 
However, the existence of a Weyl chamber where all signs are negative 
would imply that the measure is atomic \cite{Kat96}, 
contradicting our assumption that $\mu$ is smooth and ergodic.

It is worth mentioning that for a locally free $C^{1+\beta}$ 
$\mathbb{R}^k$-action on a $(2k+1)$-dimensional manifold, if the Lyapunov exponents of 
an ergodic measure are in general position 
and at least one element in $\mathbb{R}^k$ has positive entropy with respect to $\mu$, 
then $\mu$ is absolutely continuous (\cite[Main Theorem]{KKH11}). 
This implies that the smoothness assumption on the measure is 
not overly restrictive in Theorem \ref{main theorem}.

Next, we investigate the role of the joint nontriviality
condition and discuss cases in which (C1) or (C2) fails.
Throughout the following discussion, we assume that
$f$ is a topologically mixing Anosov diffeomorphism on a closed manifold $X$, 
and denote its unique measure of maximal entropy by $\mu_{\operatorname{MME}}$.

We first discuss condition (C1), 
which is a natural cohomological nondegeneracy assumption on the cocycle.
For a fixed topologically mixing Anosov base $f$,
the cocycles satisfying (C1) form an open and dense subset of $C^{1+\beta}(X,\mathbb R^k)$.
Indeed, by Liv\v{s}ic's theorem (\cite[Theorem 19.2.1]{KH95}), 
(C1) is equivalent to the existence of $k+1$ distinct periodic orbits
whose $\tau$-averages are affinely independent.
This independence persists under small perturbations
and can be achieved by arbitrarily small smooth perturbations of 
the cocycle near these orbits.
Thus, (C1) is a robust nondegeneracy condition.
It is relevant here because ergodicity of $m$ under 
the entire $\mathbb R^k$-action alone does not guarantee
ergodicity of the corresponding skew product.

We now analyze the scenarios in which joint nontriviality fails.

\textbf{Case 1:} Failure of condition (C1).

If (C1) fails, there exist a proper linear subspace
$V\subsetneq\mathbb R^k$, a vector $\mathbf c\in\mathbb R^k$,
and continuous maps $\psi:X\to\mathbb R^k$ and $\eta:X\to V$
such that
$$
\tau(x)=\psi(x)-\psi(f(x))+\mathbf c+\eta(x).
$$
Under the change of variables $g(x,y)=(x,\rho_{\psi(x)}y)$, we compute
\begin{align*}
g\circ F\circ g^{-1}(x,y)
&=g\circ F\bigl(x,\rho_{-\psi(x)}y\bigr)
=g\bigl(f(x),\rho_{\tau(x)-\psi(x)}y\bigr)
=g\bigl(f(x),\rho_{-\psi(f(x))+{\bf c}+\eta(x)}y\bigr)\\
&=\bigl(f(x),\rho_{\psi(f(x))-\psi(f(x))+{\bf c}+\eta(x)}y\bigr)
=\bigl(f(x),\rho_{{\bf c}+\eta(x)} y\bigr).
\end{align*}
Thus, the skew product $F$ is topologically conjugate to the map
$$
\bar F(x,y)=(f(x),\rho_{{\bf c}+\eta(x)}y)=\bigl(f(x),\rho_{\mathbf c}\circ\rho_{\eta(x)}y\bigr).
$$
Thus, the fiber dynamics are generated by the restricted action $\rho|_V$ and $\rho_{\bf c}$.

Suppose that $\mathbf c=\mathbf 0$ and
$Y=\mathbb T^{k-1}\times\mathbb T^{d-k+1}$, where $d>k-1$.
Assume that, for some continuous map
$a:X\to\mathbb T^{k-1}$,
$$
\rho_{\eta(x)}(u,z)=(u+a(x),z)
$$
for every $x\in X$ and $(u,z)\in\mathbb T^{k-1}\times\mathbb T^{d-k+1}$.
For each $z\in\mathbb T^{d-k+1}$, let $m_z$ denote 
the normalized Haar measure on $\mathbb T^{k-1}\times\{z\}$.
Then $\mu_{\operatorname{MME}}\times m_z$ is $\bar F$-invariant, 
and this product measure is a measure of maximal entropy for $\bar F$. 
Since $\mathbb T^{d-k+1}$ is uncountable and these
invariant sets are pairwise disjoint, 
choosing one such ergodic component for each $z$ yields uncountably
many pairwise singular ergodic measures of maximal entropy for $\bar F$. 
Their pushforwards under $g^{-1}$ give uncountably many ergodic measures
of maximal entropy for $F$. 

Similarly, suppose that $\mathbf c\neq\mathbf 0$.
Under the same assumption that $\rho|_V$ acts by
torus translations, if a probability measure $\nu$
is invariant under the entire $\mathbb R^k$-action
$\rho$ and is a measure of maximal entropy for
$\rho_{\mathbf c}$, then
$\mu_{\operatorname{MME}}\times\nu$ is an
$F$-invariant measure of maximal entropy.

If $\rho_{\mathbf c}$ admits uncountably many
distinct ergodic measures of maximal entropy
that are also invariant under the entire action
$\rho$, the ergodic decompositions of the
corresponding product measures yield uncountably
many ergodic measures of maximal entropy for $F$. 

\textbf{Case 2:} Failure of condition (C2) with $\dim Z_\rho =k$. 

With a suitable choice of the fiber action, 
$F$ admits uncountably many ergodic measures of maximal entropy. 
Specifically, for any $\rho_{\bf t}$-invariant measure $\nu$, 
the measure $\mu_{\operatorname{MME}} \times\nu$ is an $F$-invariant measure of maximal entropy.

\textbf{Case 3:} Failure of condition (C2) with $\dim Z_\rho <k$. 

Let $k\ge2$. Let
$Q=\Lambda\backslash\operatorname{PSL}(2,\mathbb R)$
be compact, and let $h_t$ denote the horocycle flow on $Q$, 
preserving the normalized Haar measure $m_Q$.
Let $a_{\mathbf s}$ denote the Weyl chamber action
on $N=\operatorname{SL}(k,\mathbb R)/\Gamma$, where
$\Gamma$ is a cocompact lattice and $\mathbf s\in\mathbb R^{k-1}$.
Write $m_N$ for the normalized Haar measure on $N$.
Set $Y=Q\times Q\times N$ and define
$$
\rho_{(t,\mathbf s)}(q_1,q_2,z)
=
\bigl(h_tq_1,h_tq_2,a_{\mathbf s}z\bigr).
$$
This is a locally free $C^\infty$ $\mathbb R^k$-action.
The probability measure $m=m_Q\times m_Q\times m_N$ 
is smooth and invariant under the entire action.
Since the horocycle flow is mixing, its diagonal
action on $Q\times Q$ is ergodic. Together with
ergodicity of the Weyl chamber action, this implies
that $m$ is ergodic for the entire $\mathbb R^k$-action.

Every horocycle element has zero topological entropy,
and its iterates and inverse iterates have at most
polynomial derivative growth. 
Thus, the stable and unstable leaves of every positive-entropy element
are contained in sets of the form
$\{q_1\}\times\{q_2\}\times N$.
Thus, the accessibility requirement in (C2) fails.

For each $r\in\mathbb R$, we define
$$
G_r=\{(q,h_rq):q\in Q\},
\qquad
\nu_r=(q\mapsto(q,h_rq))_*m_Q.
$$
Each $G_r$ is invariant under the diagonal horocycle
flow, and $\nu_r$ is an invariant probability measure
supported on $G_r$. Since the horocycle flow on a
compact quotient has no periodic orbits, the sets
$G_r$ are pairwise disjoint.
The measure $\nu_r\times m_N$ is invariant under the
entire action $\rho$ and is a measure of maximal
entropy for every action element. 
This implies that $F$ has a MME supported on $E_r:=X\times G_r\times N$.
Since the sets $E_r$ are pairwise disjoint, 
Selecting one ergodic component for each $r$ yields 
uncountably many mutually singular ergodic measures of maximal entropy for $F$. 

This construction retains the other hypotheses, 
including the existence of a smooth common measure of maximal entropy that is ergodic
for the entire action, while violating only the accessibility requirement in (C2).

The preceding discussion illustrates the role of conditions (C1) and (C2) 
in our finiteness and fiberwise measure rigidity results. 
When either condition fails, similar conclusions may still hold,
but may require more refined assumptions on the
cocycle and the fiber dynamics. 
In particular, further analysis of the cohomological structure
of the cocycle and the invariant structures of the induced fiber dynamics may be needed.

\text{\quad}\\

\textsc{Chern Institute of Mathematics and LPMC, Nankai University, Tianjin 300071 China}

{Email: dongchg@nankai.edu.cn}

\textsc{Chern Institute of Mathematics and LPMC, Nankai University, Tianjin 300071 China}

{Email: qiujieqiao@mail.nankai.edu.cn}

\end{document}